\documentclass[leqno]{article}
\usepackage{blindtext}
\usepackage{geometry}
\usepackage{amsmath}
\usepackage{amsthm}
\usepackage{amsfonts}
\usepackage{amssymb}
\usepackage{amssymb}
\usepackage{mathrsfs}
\usepackage[bottom]{footmisc}
\numberwithin{equation}{section}  
\usepackage{comment}
\usepackage{graphicx} 
\usepackage{mathtools}
\usepackage{soul}
\usepackage{mathabx}
\usepackage{xcolor}
\usepackage[hidelinks]{hyperref}
\usepackage{bbm}
\usepackage{tikz}
\usetikzlibrary{arrows}

\def\sha{{\cal A}}
\def\shb{{\cal B}}
\def\shc{{\cal C}}
\def\shd{{\cal D}}

\def\shf{{\cal F}}

\def\shm{{\cal M}}

\def\shl{{\cal L}}
\def\shp{{\cal P}}
\def\shs{{\cal S}}

\def\P{\mathbb P}
\def\Q{\mathbb Q}

\definecolor{LB}{rgb}{0.5, 0.3, 0.80}
\newcommand{\g}{\gamma}
\newcommand{\al}{\alpha}
\newcommand{\N}{\mathbb N}
\newcommand{\be}{\beta}
\newcommand{\R}{\mathbb R}
\newcommand{\C}{\mathcal{C}}

\newcommand{\eps}{\varepsilon}
\newcommand{\E}{\mathbb E}

\newenvironment{ideaproof}{%
\proof}{\endproof}

\newtheorem{assumption}{Assumption}

\newtheorem{theorem}{Theorem}[section]
\newtheorem{notation}[theorem]{Notation}
\newtheorem{remark}[theorem]{Remark}
\newtheorem{definition}[theorem]{Definition}
\newtheorem{proposition}[theorem]{Proposition}
\newtheorem{lemma}[theorem]{Lemma} 
 
\newtheorem{corollary}[theorem]{Corollary} 

\title{McKean-Vlasov equations with singular coefficients - a review of recent results}
\author{Luca Bondi, Elena Issoglio and Francesco Russo}

\begin{document}
\date{July 28th, 2025}
\maketitle

\begin{abstract}
  
This paper focuses on recent works on McKean-Vlasov stochastic differential equations (SDEs) involving singular coefficients. 
After recalling the classical framework, we review existing recent literature   depending on the type of singularities of the coefficients: on the one hand they  satisfy some integrability and measurability conditions only, while on the other hand the drift is allowed to be a generalised function.   Different types of dependencies on the law of the unknown and different noises will also be considered.   
    McKean-Vlasov SDEs are closely   related to non-linear Fokker-Planck equations that are satisfied by the law (or its density) of the unknown.   These connections are often established also in this singular
setting and will be reviewed here. 
Important tools  for dealing with singular coefficients are also included in the paper, such as
Figalli-Trevisan superposition principle, Zvonkin transformation,  Markov marginal uniqueness, and stochastic sewing lemma.  
  \end{abstract}



\tableofcontents

\section{Introduction}

This  paper is a review of recent works on McKean-Vlasov equations, also referred to as McKean-Vlasov SDEs, focussing on the cases of singular coefficients,
namely coefficients that do not satisfy the standard regularity assumptions like Lipschitz continuity. 
These  equations can be handled because of the
regularisation effect produced by (generally non-degenerate) stochastic noise,
similarly to the case of usual singular stochastic differential equations. 
This phenomenon is called {\it regularisation by noise} in the literature.

 McKean-Vlasov equations were originally introduced by McKean 
 in the mid 1960s \cite{McKean} 
and are stochastic differential equations where the drift and/or the diffusion coefficient depend on the
law of the unknown process  hence they are equations of the  form 
\begin{equation}\label{eq:McK-V}
\left\{
\begin{array}{l}
dX_t = b(t, X_t, \mathcal L^\P_{X_t}) dt + \sigma(t, X_t, \mathcal L^\P_{X_t}) dW_t\\
\mathcal L^\P_{X_t}\textrm{ is the law of } X_t \text{ under }\P, \\
   X_0\sim \nu,
\end{array}
\right.
\end{equation}
where 
  $ b:$ (resp. $\sigma:$) $[0,T]\times\R^d \times \shp(\R^d) \rightarrow \R^d$ (resp. $\R^{d\times m}$)
  and $\shp(\R^d)$ denotes the set of Borel probabilities on $\R^d$,
where $\nu$ is a prescribed initial probability, the unknown process is $X$, and $W$ is some $\R^m$-valued stochastic noise. The noise 
is typically a Brownian motion but different types of processes have been studied, such as $\alpha$-stable processes and fractional
Brownian noises.

\medskip

McKean-Vlasov equations have become very popular in recent years, after the well-known Saint-Flour lecture notes by A.S. Sznitman \cite{Sznitman}. 
They typically arise 
 as limit of interacting particle systems when the number of particles goes to infinity and when the dynamics of each particle depends on the empirical measure of all other particles. This classical case is reviewed in Section \ref{sec:sznitman}.
The specific dependence of $b$ and $\sigma$ on the law $\mathcal L^\P_{X_t}$ can be very different, for example  $b$ could be a function that depends on $\mathbb E (X_t)$, or a generic functional written as an integral with respect to
the (marginal) law  $\mathcal L^\P_{X_t}$, or it could be depending on the density of the law $\mathcal L^\P_{X_t}$
(if it exists). In the first case the dependence will be generally smooth (e.g.\ Lipschitz) with respect
  to the Wasserstein metric. 
  Concerning the density dependence case,  if we denote by $v$ the density of $\mathcal L^\P_{X_t}$, a typical dependence that can be found in the literature,
  is {\em pointwise} on the density $v(t,x)$.
  Another type of dependence is of {\it convolutional type}
  $(K \ast v)(t,x)$, where $K$ is a singular kernel; if $K$ is a bounded function,
  no existence of the density is necessary.
 When the dependence is pointwise,
 the general McKean-Vlasov equation takes the form
\begin{equation} \label{general_mckean_SDE}
 \left\{
 \begin{array}{l}
   dX_t=b(t,X_t,v(t,X_t))dt+\sigma(t,X_t,v(t,X_t))dW_t,\\
\mathcal L^\P_{X_t}(dx) = v(t,x)dx \\
   X_0\sim \nu,
 \end{array}
 \right.
\end{equation}
where 
$b: [0,T] \times \R^d \times \R \to\R^d$, (resp.\
$\sigma: [0,T] \times \R^d \times \R \to \R^{d \times m})$.

The dependence on the density of the law $\mathcal L^\P_{X_t}$ (if it exists) in a density pointwise way arises from particle systems governed by a so-called moderate interaction, first introduced by \cite{Oelschlaeger} and then further explored by  \cite{coppoletta, JourMeleard}, for smooth coefficients. This is the topic of Section \ref{sec:pdd}. We remark that in this paper we will    use the terminology {\it McKean-Vlasov SDE with density  dependence}  when the coefficients only depend on the marginal density $v$ of the process, even though often in the literature such SDEs are simply called McKean SDEs.
A modern application of McKean SDEs is the mean-field game models
started by \cite{lasrylions}, which models the limiting  behaviour
of interacting players at the Nash equilibrium.
The average agent behaves according to a solution of a
particular McKean SDE, see Section \ref{sec:stoc-contr}.
Related more recent references are  \cite{cardaliaguet, carmona-delarueI, carmona-delarueII,  lacker, cardaliaguetPrinceton}.

Departing from the classical case  explained in Section \ref{SClass}, we consider in Section \ref{sec:4} the case when the  coefficients $b$ and $\sigma$ are functions being only measurable with respect $(t,x)$,
with some integrability conditions.
In the case when the coefficients of the McKean-Vlasov SDE, namely  \eqref{eq:McK-V},
are smooth in the sense of Wasserstein, we review
 a series of works relaxing the continuity assumption on the coefficients to some  $L^p$-$L^q$ condition, following the idea of the well-known work by Krylov and R\"ockner \cite{kry-rock} for SDEs.  
In \cite{HuangWangSPA} the authors consider
SDEs, where the drift satisfies a
  Krylov-R\"ockner $L^p$-$L^q$-type dependence.
  In \cite{deRaynal} the author considers McKean-Vlasov equations with coefficients $b$ and $\sigma$ which depend on the law of
the process in a H\"older continuous way.
In \cite{HuangWang2021}, the authors study McKean-Vlasov SDEs with $\sigma$ independent of $\mu$,
and with drift discontinuous in Wasserstein distance, but continuous under total variation distance. 
Independently \cite{RocknerZhang} considers similar assumptions as in \cite{HuangWangSPA, HuangWang2021} but with a localised version of
$L^p$-$L^q$. 
We mention then \cite{GaleatiLing23}, where regularity assumptions on the coefficients are further relaxed. We review in details all the above papers in Section \ref{sec:LpLq_w}. Finally,  the contributions \cite{Olivrichtoma23, Olivrichtoma} deal with the $L^p$-$L^q$
space-time condition, but 
with a singular  density dependence  of convolutional type:
they are reviewed in Section~\ref{sec:LpLq_p}.

\medskip
 
The well-known link between the law of a solution to an SDE and
Fokker-Planck equations can be shown also  in the McKean-Vlasov case. Indeed,
   given a (weak) solution $(X,\P)$ of the McKean SDE above, one can apply It\^o's formula to the process
  $\varphi(X_t)$, where $\varphi$ is a suitable smooth function.
  One then  takes the expectation under $\P$ to obtain that the measure $ \mathcal L^\P_{X_t}$ solves the non-linear Fokker-Planck 
PDE (in the sense of distributions)
\begin{equation} \label{general_nonlinear_FP}
  \left\{
  \begin{array}{l}
   \partial_t \mu_t(dx)= \frac12\sum_{i,j=1}^d \partial_{ij} (a_{ij}(t,x,\mu_t(dx)) \mu_t(dx))
   -\sum_{i=1}^d \partial_i(b_i(t,x, \mu_t(dx))\mu_t(dx)) \\
   \mu_0(dx) = \nu,
   \end{array} 
   \right.
 \end{equation}
where $a:=\sigma\sigma^{\top}$.
The fact that the law of $X$ solves \eqref{general_nonlinear_FP} suggests that the McKean-Vlasov SDE constitutes a {\it probabilistic representation} of the PDE.
 A more difficult task is the converse relation, i.e.~given a solution
 $(\mu_t)$  to the non-linear Fokker-Planck PDE \eqref{general_nonlinear_FP},
 we want to construct a stochastic process $X$ and a probability $\P$
 such that $(X, \P)$  is a solution in law of the McKean-Vlasov SDE \eqref{general_mckean_SDE}.
This is reviewed in Section \ref{sc:PR} and is based on the so-called {\em superposition principle} by Figalli, Trevisan
 \cite{figalli, Trevisan} extended recently by \cite{bogachev_superposition}, explained in Section \ref{sc:figalli}.

 Consider the Fokker-Planck equation \eqref{general_nonlinear_FP}, in the case when $\mu$ is absolutely continuous with density $v$,  
  $b=0$ and
 $a(t,x,\mu) =  2 v(t,x)^{m-1} I_d, $ with $I_d$ the unit matrix on $\R^d$
   and  $m\geq 1$. This very popular PDE is indeed
    the classical porous media equation
 $$ \partial_t v= \Delta (v^m).$$ 
 This equation admits an important
class of solutions of the form  \eqref{eq:Barenblatt}, called {\it Barenblatt-Pattle} solution and first introduced in \cite{Baren}.
In  \cite{Ben_Vallois} the authors  show that the marginal laws of the solution to suitable McKean SDEs are Barenblatte-Pattle solutions,
providing (at the best of our knowledge) the first probabilistic representation in the classical porous media framework.
The classical porous media equation has been generalised to 
\begin{equation} \label{PMEIntro}
  \partial_t v= \Delta (\beta(v)), 
\end{equation}
where $\beta: \R \rightarrow \R$ is a monotone (possibly multivalued) function. 
A significant example of $\beta$ is the discontinuous
case $\beta(v) = v H(v-v_c)$, where $H$ is the Heaviside
function. It is motivated by the literature of complex systems,
in particular the so called {\it self-organised criticality}, see the monograph \cite{bak86}
and more specifically \cite{BanJa}.
Probabilistic representation of that PDE was given first in the case
of dimension $d=1$, again via a McKean type SDE,
by \cite{BRR, BRR2}.
Other general non-linear Fokker-Planck equations of porous media type, and related McKean SDEs,  with more general diffusion and drift coefficients were investigated by \cite{BCR2, BarbuRockSIAM, BarbuRoeckJFA, RehmeierPLap}.
We review all the above papers in Section \ref{sec:pm}. 

The solutions of a porous media equation with radially symmetric initial condition
  is also radially symmetric and naturally produce a new class of one-dimensional
  Fokker-Planck equations, which admit a significant probabilistic representation
  given by a new class of one-dimensional  McKean SDEs, see Section \ref{sec:radially}.
  Another class of conservative second order PDEs, which was recently interpreted
  (see \cite{RehmeierPLap})
  as Fokker-Planck type PDEs is the $p$-Laplace equation, whose associated McKean SDEs
  produce the so called $p$-Brownian motion, see Section \ref{sec:pLap}.
  The case $ p =2$ corresponds to the classical heat equation and the
  corresponding process is the classical Brownian motion.

 In some cases McKean SDEs
  can also be probabilistic representation of
  non-conservative Fokker-Planck type PDEs which involve a zero-order
  term, see Section    \ref{sec:non-con}.
  In this case the link between the solution of the PDE and
  the law of the process is not so direct and it is expressed via
  the so called {\it linking equation}, see e.g.\ \cite{LOR2, LOR1}.
  Some more unusual  relations between McKean type SDEs and
  Fokker-Planck type PDEs exist. 

  In \cite{bossy2011conditional} a conditional Langevin-type dynamics is considered as an alternative approach to Navier-Stokes equations for turbulent flows. Another example involving a conditional law in the McKean SDE is related to local stochastic volatility models \cite{labordere}, which is based on the Markov projection idea \cite{gyongy},
see
    Section \ref{sc:con_exp}.
Fokker-Planck type PDEs with terminal conditions
  can be also represented by McKean SDEs, see e.g. \cite{Izydorczyk}, which are
  related to the time-reversal of a diffusion \cite{haussmann_pardoux}, see also Section \ref{sc:time-rev}.

\medskip

Section \ref{sc:sing_McK} is dedicated to McKean SDEs with coefficients involving generalised functions, such as elements of negative Besov spaces. Due to the fact that generalised functions cannot be evaluated at points, the equations are only formal  at this stage, and one has to give a mathematically precise meaning to them before proceeding to study well-posedness. This is in contrast to the singular McKean SDEs presented in Section \ref{sec:4}.

In the framework of density pointwise dependence, the well-posedness of the 
singular McKean-Vlasov SDE \eqref{general_mckean_SDE} was first studied in \cite{issoglio_russoMK} in the case of additive Brownian noise and a drift of the form $b(t, x, v)=F(v)\shb(t, x)$ where $\shb(t,\cdot)$ is an element of a negative Besov-H\"older space and $F$ a smooth non-linearity. Weak well-posedness is established in \cite{issoglio_russoMK}.
Similar equations were studied in \cite{chaudru_jabir_menozzi22,chaudru_jabir_menozzi23} where the additive noise is an  $\alpha$-stable process and the  drift  is of convolution type, namely of the form 
$b(t,x,\mu) = \left(\shb(t,\cdot) \ast \mu\right)(x)$ with $\shb(t,\cdot)$  an element of a negative Besov space. 
 They prove well-posedness of said SDE and heat kernel estimates.  
Finally we cite the work  \cite{GaleatiGerencser25} 
where the drift  depends on the law in a convolutional way like in \cite{chaudru_jabir_menozzi22,chaudru_jabir_menozzi23} but the additive noise is a  fractional Brownian motion. They prove strong existence and path-by-path uniqueness.

A generalisation of (\ref{general_mckean_SDE}) that has been considered in the literature is the kinetic or
degenerate setting, where some components of the diffusion coefficient are zero.
We first mention \cite{veretennikov2023weak, rondelli},
where weak well-posedness was established,
when the drift (resp. the diffusion) involve
a convolution of a continuous bounded function $\shb$ (resp. $\sigma)$,
with respect to the marginal law, see Section \ref{sc:sing_deg_McK}.
We then mention  \cite{zhang2021second}, who also considers a
convolutional model with respect to a more irregular function.
We also mention \cite{hao2021singular}  in which the drift has a convolutional dependence of the law with
a generalised function living in some negative Besov space. 
An  irregular degenerate model beyond the kinetic equation
is investigated in \cite{issoglio_et.al24}, where the drift of the non-degenerate  components of $X$ is  of the form $b(t, x, v)=F(v)\shb(t, x)$ as in \cite{issoglio_russoMK} and where the authors prove weak well-posedness.

Our goal for Section \ref{sc:sing_McK} is to provide a comprehensive and detailed methodology for defining and solving McKean SDEs with (Schwartz) distributional coefficients. To this aim, we  present in a unified framework the works \cite{issoglio_russoMK, issoglio_russoPDEa, issoglio_russoMPb, issoglio_et.al24}. 
Section \ref{sec:top-prelim}  covers useful topological preliminaries that introduce the function spaces in which we will work and key analytical tools such as Schauder estimates and Bernstein inequalities. In Section \ref{sec:sol-singular-SDE} we review some possible ways to define solutions to singular  SDEs, which is a delicate point since distributions cannot be evaluated at points.  This definition will be at the basis of the corresponding definition for McKean-Vlasov SDEs, which will be introduced and used below in Section \ref{sec:McKean-point}.
Since weak (in law) solutions of classical SDEs are equivalent to martingale problems, see e.g.
Section 5.4 A and B of
\cite{karatzasShreve}, 
we will often interpret rough SDEs by the notion of  rough martingale problem, see Definition \ref{def:MK_MP}.
Section \ref{ssc:singularPDE} reviews results on the singular backward Kolmogorov PDE, which is fundamental to define the rough martingale problem. The Kolmogorov PDE well-posedness is also a delicate issue because of the presence of the
distributional coefficients.
Section \ref{sec:linearSDE} is dedicated to finding a solution to the singular SDE via rough martingale problem formulation. This result is used together with a well-posedness result for the corresponding singular non-linear Fokker-Planck equation (established in Section \ref{sec:singularFP}) to finally prove existence and uniqueness of a solution to the rough McKean SDE in Section \ref{sec:McKean-point}. 
In  Section \ref{sc:sing_deg_McK} we illustrate the application of this blueprint in a degenerate case.

\medskip 

Some of the  recurrent tools and methodologies that are used throughout the analysis of singular McKean-Vlasov SDEs in Sections \ref{sec:4} and \ref{sc:sing_McK} are collected and explained  in Section \ref{sc:Tools}. These include the
Figalli-Trevisan {\em superposition principle}  that allows to construct a stochastic process solution to an SDE, where the law of the process is prescribed as the solution to a Fokker-Planck equation, see Sections \ref{sc:figalli} and \ref{sc:PR}; the {\em Zvonkin transformation}, that allows to remove the drift of an SDE via the solution of a Kolmogorov-type PDE, see Section \ref{ssc:zvonkin};  a way to obtain the {\em uniqueness} in law of martingale problems in a Markovian context, namely a tool to lift the uniqueness of time-marginal distributions to the uniqueness of the whole distribution, see Section \ref{sc:ethier-kurtz}; and the {\em stochastic sewing lemma} that is useful to construct a stochastic process whose increments are close to  a given family of stochastic increments, see Section \ref{sec:ssl}.

\medskip 
This paper is organised as follows. In Section \ref{SClass} we review the classical setting, which includes the case of  Wasserstein-type dependence on the law, the case of density pointwise dependence on the law, and some other examples, all involving smooth enough coefficients. Section \ref{sc:Tools} covers specific tools and methodologies which are useful for the study  of singular McKean-Vlasov SDEs, which are then used in the last two sections: Section \ref{sec:4} is devoted to the case of non smooth coefficients which are still functions, while Section \ref{sc:sing_McK} is dedicated to the case of distributional coefficients.

\section{Classical McKean-Vlasov SDEs}

\label{SClass}

Let us consider the first equation \eqref{eq:McK-V}. We will use the terminology of {\it classical McKean SDE}
  in the following two cases:
  \begin{itemize}
  \item $b, \sigma$ depend in a smooth way
   (essentially Lipschitz) with respect to the Wasserstein metric;
  \item $b, \sigma$ have a smooth pointwise density dependence.
  \end{itemize}
 The first one involves the whole law distribution $\shl^\P_{X_t}$ at time $t$ of the process $X$, whilst the second one takes into account only the
value of the law density calculated in $X_t$ at time $t$. 
The precise dependence on the law of the process $X$ will be specified case by case.
In the sequel we will say that $a = \sigma \sigma^\top $ is \emph{(uniformly) non-degenerate} if
\begin{equation} \label{eq:Non-deg}
a(t,x,\mu) \xi \cdot \xi \ge {\rm const} \vert \xi \vert^2, \forall \xi \in \R^d, 
\end{equation}
where the constant does not depend on the arguments of $a$.

We want to make the reader aware of the following fact:
in the pointwise density dependence case,
quantities such as moments of the corresponding law cannot be evaluated, since this would require  the evaluation of $v$ on the entire space $\R^d$, an information which is not accessible.

In the sequel of the section will focus on well-posedness and
on approximation results deriving from particle systems.

\subsection{Wasserstein-type dependence}\label{sec:sznitman}

If $E$ is a Fréchet space, we will denote by
$\mathcal{P}(E)$ the space of
Borel probability measures on $E$.
Let $\Omega_T:=C([0,T],\R^d)$
be  the so called canonical space.
We will use  
the letter $\mu$ to denote a generic measure in $\mathcal{P}(\Omega_T)$ and with $\mu_t$ its marginal at time $t$
i.e.\
\begin{equation} \label{eq:m}
  \mu_t(A):= \mu\{\omega\in\Omega_T : \omega(t)\in A\},
  \quad\forall t\in [0,T],\quad \forall A\in\mathcal{B}(\R^d),
\end{equation}
which is a Borel probability measure on $\R^d$. We consider now a drift coefficient $b$ of the  form
  \begin{equation}
    b(t,x,\mu_t)=\int_{\R^d} \Tilde{b}(x,y)\mu_t(dy)=\E^{\mu_t}[\Tilde{b}(x,\cdot)],
  \end{equation}
where $\mu_t \in \mathcal P (\R^d)$ is defined above and $\Tilde{b}:\R^d\times\R^d\to\R^d$ is some given bounded Lipschitz continuous function.
Evaluating \eqref{eq:m} for $\mu_t = \mathcal L^\P_{X_t}$, we write
\begin{equation}
  b(t,X_t,\mathcal{L}^{\P}_{X_t})=\int_{\R^d} \Tilde{b}(X_t,y)\mathcal
  L^{\P}_{X_t}(dy)=\E^{\mathcal L^{\P}_{X_t}}[\Tilde{b}(X_t,\cdot)].
  \end{equation}
 McKean-Vlasov  SDEs of the form
\begin{equation}\label{eq:sznitmanSDE}
    \left\{
\begin{array}{l}
  dX_t = 
 b(t,X_t,\mathcal{L}^{\P}_{X_t})dt +dW_t \\
\mathcal L^\P_{X_t} \textrm{ is the law of } X_t \textrm{ under }\P,
\end{array}
\right.
\end{equation}
were considered in \cite{Sznitman}.
Existence and uniqueness of a solution to (\ref{eq:sznitmanSDE})
was established using a fixed point argument. In order to introduce the result, we need the definition of the so-called \textit{Wasserstein distance}. Given any two probability measures $\mu,\nu\in\shp(\Omega_T)$ we define
\begin{equation}\label{1_wasser}
  \mathcal{W}_T(\mu,\nu):=\inf_{\pi\in\mathcal{P}(\Omega_T \times \Omega_T)}\bigg\{\int(\sup_{s\leq T}|\omega_1(s)-\omega_2(s)|\wedge1)d\pi(\omega_1,\omega_2):
\omega_1 \sim \mu, \omega_2 \sim \nu
  \bigg\}.
  \end{equation}
In addition to this, we define $\Phi$ the map that associates $\mu\in\mathcal{P}(\Omega_T)$ to
$\Phi(\mu)$
\[
\begin{array}{llll}
    \Phi:  & \mathcal P(\Omega_T)& \to  & \mathcal P(\Omega_T)\\
           & \mu & \mapsto &\Phi (\mu),
\end{array}
\] 
where $\Phi(\mu)$ 
is the law of the solution $(X_t)$ to the classical SDE
\begin{equation} \label{eq:lim}
    X_t=X_0+W_t+\int_0^t\bigg(\int_{\R^d} \tilde b(X_s,y)d\mu_s(y)\bigg)ds, \quad t\leq T.
\end{equation}
One looks for a fixed point of the map $\Phi$, because this would give a solution to \eqref{eq:sznitmanSDE}.
This can be obtained via the following. 
\begin{lemma}[Lemma 1.3 in \cite{Sznitman}] \label{lemma:Sznitman}
    For $t\leq T$ and any $\nu,\mu \in \mathcal P(\Omega_T)$ the following inequality for the Wasserstein distance holds
    \begin{equation}\label{gronwall_formulation}
        \mathcal{W}_t(\Phi (\mu),\Phi(\nu))\leq C_T\int_0^t \mathcal{W}_u(\mu,\nu)du\leq C_T T \mathcal{W}_T(\mu,\nu),
    \end{equation}
where $C_T$ depends on $T$ and on the Lipschitz constant of $b$.
  \end{lemma}
  
  This lemma is proven through an easy application of Gr\"onwall's inequality and it represents the main tool to prove the theorem below,
  of which we provide the proof for completeness.
\begin{theorem}[Theorem 1.1 in \cite{Sznitman}] \label{thm_sznitman}
    There exists a unique strong solution to (\ref{eq:sznitmanSDE}).
  \end{theorem}

\begin{proof}

  A solution will be in fact a  fixed point of $\Phi$ with respect to $\mathcal{W}_T$.
To construct it, let $\mu \in \shp(\Omega_T)$.
For each $ k \in \N$, we apply $k$ times the first inequality
and once the second one in  
\eqref{gronwall_formulation} in
Lemma \ref{lemma:Sznitman}
to the pairs of measures $(\Phi^{k+1}(\mu),\Phi^{k}(\mu))$ for all $k\in\N$ to get
\begin{equation*}
  \sum_{k=0}^{+\infty}  \mathcal{W}_T(\Phi^{k+1} (\mu),\Phi^k(\mu))\leq  \sum_{k=0}^{+\infty} C_T^k \frac{T^{k}}{k!}\mathcal{W}_T(\Phi(\mu),\mu) < +\infty.
\end{equation*}
Consequently the sequence  $ (\Phi^k(\mu))$ is Cauchy, therefore there is $\ell \in \shp(\Omega_T)$
such that $(\Phi^k(\mu))$ converges to $\ell$.
There, being $\Phi$ continuous we get $\Phi(\ell) = \ell$.
The uniqueness of the fixed point comes again by a repeated application
of Lemma \ref{lemma:Sznitman} to two given fixed points $\mu^1, \mu^2$.
Indeed for every positive integer $k$ we have
$$ \mathcal{W}_T(\mu^1,\mu^2) = \mathcal{W}_T(\Phi^k(\mu^1),\Phi^k(\mu^2))
\le C_T^k \frac{T^{k}}{k!}\mathcal{W}_T(\mu^1,\mu^2).$$
Choosing $k$ large enough so that $ C_T^k \frac{T^{k}}{k!}$ is strictly
smaller than $1$ we get $\mathcal{W}_T(\mu^1,\mu^2) =0$.
\end{proof}

By a direct application of It\^o's formula to $\varphi(X_t)$ where $\varphi$
is a smooth test function on $\R^d$ with compact support and $X$
is the solution of \eqref{eq:sznitmanSDE}, one easily sees that 
the function $\mu:[0,T] \rightarrow \shp(\R^d)$ characterizing
the marginal laws of $X$ is a solution of the integro-PDE
\begin{equation}
  \partial_t \mu_t = \frac12 \Delta \mu_t-\textrm{div}\left(\int_{\R^d} \tilde b(x,y)\mu_t(dy)\mu_t\right),
\end{equation}
in the sense of distributions.

A significant property (see Theorem 1.4 of \cite{Sznitman})
of \eqref{eq:lim} is that it can be seen as the asymptotic behaviour,
as $N$ goes to $\infty$,
of a single particle $X^{i,N}$ in 
the
interacting particle system
\begin{equation}
\left\{
\begin{array}{l}
    dX_t^{i,N}=\frac{1}{N}\sum_{j=1}^N \tilde b(X_t^{i,N},X_t^{j,N})dt+dW^i_t\\
    X^i_{t=0}=X^i_0,
\end{array}
\right.
\end{equation}
where $\tilde b$ is as above and $X_t^{i,N}$ (or $X_t^{j,N}$) denotes the $i$-th (or $j$-th) particle at time $t\leq T$ in a system of $N$ particles, for all $i,j=1,\dots,N$ with $N\in\N$.
This is equivalent with the fact that the empirical measure
$\frac{1}{N} \sum_{j=1}^N \delta_{X^{j, N}}$ converges weakly 
to the (deterministic) law of $X$ (under $\P$).
In particular, concerning the marginal laws, the random probability measure $\frac{1}{N} \sum_{j=1}^N \delta_{X^{j, N}_t} \in \shp(\R^d)$  converges in law to the constant probability measure
$\mu_t( dx)$ for each $t \in [0,T]$, 
 where 
  $\mu_t(dx)$ is the 
 time marginal $\shl^\P_{X_t}$
 at time $t$.
This is explained in the ``laboratory example I.~1)'' in \cite{Sznitman}.

\subsection{Pointwise density dependence}\label{sec:pdd}

We consider here a class of coefficients
depending on the probability density $v(t,x)$ of the law $\shl^\P_{X_t}$.
This type of dependence is more singular than the one just presented, and, in particular, it is not Wasserstein continuous,
therefore fixed point theorems, used before, can no longer be employed.

We consider here three cases of pointwise density dependence in the McKean type equation:
1. The case with additive noise.
2.  The case of smooth diffusion and drift coefficients.
3.  The power-like diffusion term without drift.

\medskip
{\bf 1. The case with additive noise}\\
In \cite{coppoletta,Oelschlaeger} the authors consider the  particle system
\begin{equation}\label{oel_coppoletta_particles}
       \left\{
\begin{array}{l}
dX^{i,N}_t=b(X^{i,N}_t,V^N*\mu_t^N(X^{i,N}_t))dt+dW^i_t,\quad 1\leq i\leq N,\quad 0<t\leq T,\\
X^{i,N}_0=\xi^i\sim \nu,
\end{array}
\right.
\end{equation}
where $b:\R^d\times \R\to\R^d$ is bounded, Lipschitz continuous and
the map $(x,p) \mapsto p b(x,p)$ is Lipschitz,
$W^i_t, 1 \le i \le N,$ are independent $\R^d$-Brownian motions on some probability space $(\Omega, \mathcal F, \P)$, $\xi^i,1 \le i \le N,$  are random variables independent of the $W^i_t$ and $\mu_t^N = \frac{1}{N} \sum_{i=j}^N \delta_{X^{j, N}_t}$ is the empirical measure. 
A smooth radially symmetric probability density $V$ on $\R^d$ is used to define a sequence of mollifiers
\begin{equation}\label{eq:V_N}
  V^N(x):= \chi_N^d V(\chi_N x),
  \end{equation}
which converges to the Dirac measure
when the smoothing parameter $\frac{1}{\chi_N}$ converges to zero.
The parameters $\frac{1}{\chi_N}$ and $N$ are chosen fulfilling the
trade-off relation $\chi_N = N^{\frac{\beta}d}, 0 <  \beta <1$. 
Consequently 
\begin{equation*}
     V^N*\mu_t^N(x)=\frac{N^{\be}}{N}\sum_{j=1}^N V(N^{\frac{\be}{d}}(x-X^{j,N}_t)),\quad \be\in(0,1),
   \end{equation*}
which is a mollified version of the empirical measure $\mu_t^N$.
The authors proved in \cite{Oelschlaeger,coppoletta} that 
the limit dynamics as $N\to \infty$ of a single particle follows
the McKean equation 
\begin{equation}\label{oel_coppoletta_eq}
       \left\{
\begin{array}{l}
dX_t=b(X_t, v(t,X_t))dt+dW_t\\
\mathcal L^\P_{X_t}(dx) = v(t,x) dx \\
X_0 \sim \nu,
\end{array}
\right.
\end{equation}
where $W$ is a Brownian motion on $(\Omega, \mathcal F, \P)$.
A basic tool for proving previous convergence in  \cite{Oelschlaeger} or even \cite{coppoletta}
was the tightness of the laws of the processes
\eqref{oel_coppoletta_particles}.
The fact that the diffusion is constant simplifies the verification
of Kolmogorov-Centsov theorem, see e.g. Theorem 2.2.8 of \cite{karatzasShreve}.

In the sequel we will also say that the couple $(X,v)$
is a solution of the McKean SDE \eqref{oel_coppoletta_eq}.
Weak well-posedness  to the McKean equation \eqref{oel_coppoletta_eq}  is equivalent to 
solving the  following martingale problem on the canonical space:
find a probability measure $\mathbb Q \in \mathcal P(\Omega_T)$ such that 
 for all $f\in C_b^2(\R^d)$ and for all $0<t\leq T$,
\begin{equation}\label{MP_coppoletta}
    f(Z_t)-f(Z_0)-\int_0^t \left( b(Z_s,v(s,Z_s))\nabla f(Z_s)+\frac{1}{2}\Delta f(Z_s)\right)ds
\end{equation}
is a $\mathbb Q$-martingale, where $(Z_t)$ is the canonical process on $\Omega_T$,
$v(s,\cdot)$ is the marginal law density of $\Q$ at time $s\in[0,T]$, and the initial condition satisfies $ v(0,x) dx=\nu(dx)$.
In  \cite{coppoletta} it is   established
   that
   $$ \E\left[ \int_{[\varepsilon,T]\times \R^d} |V^N*\mu_t^N(x)- v(t,x))|^2dt dx \right]  \rightarrow 0,$$
   for every $\varepsilon > 0$. Moreover the authors also prove the asymptotic independence of the particles which constitutes the so called {\it propagation of chaos}. 
  This shows in particular existence of \eqref{oel_coppoletta_eq}.
On the other hand, as in Section \ref{sec:sznitman}, an easy application of It\^o's formula shows that the density $v$ verifies
\begin{equation}\label{weak_PDE_oel}
    \langle v(t), f\rangle-\langle \nu,f \rangle=\int_0^t\langle v(s),b(\cdot,v(s,\cdot))\nabla f+\frac{1}{2}\Delta f\rangle ds,\quad t\leq T,\quad\forall f\in C_b^2(\R^d),
  \end{equation}
  where the notation $\langle \cdot, \cdot \rangle$ denotes the dual pairing  in $L^2(\R^d)$.
Consequently $v$
fulfills, in the sense of distributions, the non-linear Fokker-Planck equation 
\begin{equation}\label{nonlin_FP}
    \partial_tv(t,x)=\frac{1}{2}\Delta v(t,x)-\textrm{div}(v(t,x)b(x,v(t,x))), \quad v_0=\nu\in\mathcal{P}(\R^d),
\end{equation}
where non-linearity arises from the fact that the coefficient  $b$ depends itself on the unknown $v(t,x)$. Such 
$v$ is also called {\it weak solution} of the aforementioned PDE.
  
  Theorem 1.4 of \cite{coppoletta}  establishes uniqueness for (\ref{oel_coppoletta_eq}) as follows.
First it is proved that
  given any solution $(Y,v)$, $v$ is a weak solution of \eqref{nonlin_FP}.
In  \cite{Oelschlaeger}  the author focuses on the existence and uniqueness of such   weak solution.
On the other hand, in  Proposition 1.5 in \cite{coppoletta}, the authors  prove well-posedness of a mild solution of \eqref{nonlin_FP}, which means that $v$ satisfies
\begin{equation}\label{mild_PDE_meleard}
    v(t)=P_t\nu-\int_0^t P_{t-s}\textrm{div}(v(s,\cdot)b(\cdot,v(s,\cdot)))ds,  t \ge 0,
\end{equation}
where $(P_t)$ is the heat semigroup. 
The two formulations are  equivalent, even in more general and irregular cases, 
see Proposition \ref{lemma:WeakMild} in Section \ref{sc:sing_McK}.
In this way $v$ is determined: the uniqueness of $X$ follows by classical Stroock-Varadhan arguments for martingale problems associated with classical SDEs.

{\bf 2.  The case of smooth diffusion and drift coefficients.}\\
A  more delicate case, treated in \cite{JourMeleard}, is when the model presents a non-constant diffusion coefficient $\sigma$. In this case, 
for instance the proof tightness cannot be proven using
the  tools of \cite{Oelschlaeger} or \cite{coppoletta}.

The authors studied the  system of particles
\begin{equation}
    \left\{
    \begin{array}{l}
         dX^{i,N}_t= b(V^N*\mu_t^N(X_t^{i,N}))dt+\sigma(V^N*\mu_t^N(X_t^{i,N}))dW_t^i\quad 1\leq i\leq N, \\
         X_0^i= \xi^i\sim v_0(x)dx,
    \end{array}
    \right.
\end{equation}
where $b:\R\to\R^d$ is $C^2$, 
Lipschitz 
and $\sigma:\R\to\R^d$ is $C^3$, 
Lipschitz,
such that $a:=\sigma \sigma^\top$ is non-degenerate,
$\xi^i, 1 \le i \le N,$ are i.i.d. random variables with smooth law density $v_0\in \shc^{2+\al}$,
see Section \ref{sec:top-prelim} for the definition of the space.
This leads to the study of the  McKean SDE
\begin{equation}\label{jourdain_eq}
    \left\{
    \begin{array}{l}
       dX_t= b(v(t,X_t))dt+\sigma(v(t,X_t))dW_t\\
       v\in C_b^{1,2}([0,T]\times\R^d) \text{ is the law density of }X_t\\
       X_0\sim v_0(x)dx,
    \end{array}
    \right.
\end{equation}
and its associated Fokker-Planck PDE 
\begin{equation}\label{PDE_jourdain}
\left\{
\begin{array}{l}
    \partial_t v=L_v^{\star}v\\
     v(0,\cdot)=v_0(\cdot),
\end{array}
\right.
\end{equation}
where, given a generic
function $u:[0,T]\times \R^d \rightarrow \R$,
$L_u^{\star}$ is the adjoint operator of $L_u$, where
\begin{equation*}
    L_u  =\frac{1}{2}\sum_{i,j=1}^d a_{ij}(u(t,x))\partial_{ij} +\sum_{i=1}^d b_i(u(t,x))  \partial_i .
\end{equation*}

In this case we want to focus on the strategy pursued by the authors to establish well-posedness of \eqref{jourdain_eq}. This can be summed up as follows, for details see Proposition 1.1-1.3 in \cite{JourMeleard}.
First, one studies the existence of a solution to the non-linear Fokker-Planck PDE given in Equation \eqref{PDE_jourdain}. As shown in~\cite{JourMeleard}, the existence of such a solution follows from classical results, see also~\cite{ladyzhenskaya}. Moreover, the solution $\tilde{v}$ of \eqref{PDE_jourdain} can be proved to be
a classical solution, belonging to $C_b^{1,2}([0,T]\times\R^d)$.
From equation~\eqref{jourdain_eq}, one can recover a classical SDE, by substituting $\tilde{v}$ into the drift and diffusion coefficients, as detailed in Proposition 1.1 in \cite{JourMeleard}. The resulting SDE is
\begin{equation}\label{eq:lin_SDE_jour}
\left\{
    \begin{array}{l}
       dX_t= b(\tilde v(t,X_t))dt+\sigma(\tilde v(t,X_t))dW_t\\
       \tilde v\in C_b^{1,2}([0,T]\times\R^d) \\
       X_0\sim v_0(x)dx,
    \end{array}
\right.
\end{equation}
and it is strongly well-posed because both $b \circ \tilde{v}$ and $\sigma \circ \tilde{v}$ are Lipschitz continuous.
The solution $X$ of this classical SDE admits a law density $v$, as established in Chapter 6 in \cite{friedman} and Theorem 2.3.1 in~\cite{nualart}. By applying Itô's formula, one finds that this density $v$ satisfies the linear Fokker-Planck PDE
\begin{equation}\label{eq:linPDE_jour}
\left\{
\begin{array}{l}
\partial_t v = L^{\star}_{\tilde{v}} v, \\
v(0,\cdot) = v_0(\cdot),
\end{array}
\right.
\end{equation}
in the sense of distributions.
According to Chapter 6 in \cite{friedman}, the solution $v$ to equation~\eqref{eq:linPDE_jour} can be represented by integrating against a fundamental solution. Furthermore, as discussed in Section IV.14 \cite{ladyzhenskaya}, this fundamental solution is sufficiently smooth, which ensures that $v$ is smooth as well. Since $\tilde{v}$ solves the original non-linear Fokker-Planck equation~\eqref{PDE_jourdain}, it also solves the linear equation~\eqref{eq:linPDE_jour}. By the uniqueness of classical solutions to~\eqref{eq:linPDE_jour}, guaranteed by maximum principle arguments, it follows that $v = \tilde{v}$. Therefore, $\tilde{v}$ represents the marginal law densities of the solution process $X$ of equation~\eqref{eq:lin_SDE_jour}, and existence of a solution to the McKean SDE~\eqref{jourdain_eq} is thereby established.

Regarding the proof of uniqueness,
the authors employ the more standard technique,
already pointed out in case {\bf 1.\ }of the present section.
Assume there are two solutions $(X^1,v^1)$ and $(X^2,v^2)$ to~\eqref{jourdain_eq}. Applying Itô’s formula to functions $f(X^1)$ and $f(X^2)$ with $f \in C_b^2$ shows that both $v^1$ and $v^2$ satisfy the non-linear Fokker-Planck equation~\eqref{PDE_jourdain} in the distributional sense, and they are also smooth. Since this PDE admits a unique solution, we conclude that $v^1 = v^2$, and hence the McKean SDE~\eqref{jourdain_eq} reduces to a standard SDE with Lipschitz coefficients, completing the uniqueness argument.

{\bf 3.  The power-like diffusion term without drift.}\\
The third case corresponds to the 1-dimensional case  when the Fokker-Planck
  type equation is
the so called porous media equation,
given by the Cauchy problem
\begin{equation}\label{porous_media_PDE}
\left\{
\begin{array}{l}
     \partial_t v=\frac{1}{2}\partial_{xx}(v^{2m+1})  \\
  v(0,dx)=\mu(dx),
\end{array} 
\right.
\end{equation}
where $m >0$ and $\mu$ is a given Borel probability measure.
We mention \cite{Ben_Vallois} as one of the first articles to have
introduced the porous media equation
in relation to the solution $X$ of the McKean equation
\begin{equation} \label{eq:BenVallois}
\left\{
    \begin{array}{l}
        X_t=X_0+\int_0^tv^m(s,X_s)dW_s\\
      v(t,\cdot )\text{ is the law density of }X_t \\
        v(0,dx)=\mu(dx).
    \end{array}
    \right.
  \end{equation}
  Particular solutions of \eqref{eq:BenVallois} are recovered by providing
  particular solutions of \eqref{porous_media_PDE}.
  In \cite{Ben_Vallois}   the marginal law densities
  are analytically expressed as the so called  \emph{Barenblatt-Pattle} functions  
\begin{equation} \label{eq:Barenblatt}
    v(t,x):=\frac{1}{t^{\be}}\bigg(a_m-\frac{m}{(2m+1)(m+1))}\frac{x^2}{t^{2\be}}\bigg)^{1/2m}_+,\quad \be=\frac{1}{2m+2},
\end{equation}
where $a_m$ is a constant depending only of $m$
and $x_+:=\max\{x,0\}$.
Those functions solve \eqref{porous_media_PDE} with initial condition $\mu (dx)= \delta_0(dx)$
so that \eqref{eq:BenVallois} is solved for that $\mu$.
For this case, the authors prove
existence and  uniqueness in law, see  Theorem II.1.
Proposition III.1 establishes 
strong existence and pathwise uniqueness if the initial data $\mu$  is a strictly positive probability density, which is differentiable  with bounded derivative.
Theorem III.3 extends previous well-posedness result of Theorem II.1  to the case when $\mu$ is not necessarily strictly positive  and the derivative
is not necessarily bounded, fulfilling some further technical assumptions.
In this case the solution is strictly positive at each time and the paths are H\"older continuous
with parameter smaller than $\frac{1}{2}$.
In general, as for previous Fokker-Planck PDEs, a crucial step at some point should be the uniqueness of
some associated linearised Fokker-Planck one. This point seems to be missing, but this gap is
certainly filled in Chapter 6.\ of \cite{profeta}.

\subsection{About non-conservative Fokker-Planck PDEs and related McKean SDEs}
\label{sec:non-con}

Some work has been done also for so called non-conservative
Fokker-Planck PDEs and their probabilistic representation.
They constitute a natural extension of the conservative ones.
 One typical motivation is the following:
let us consider a  non-linear Kolmogorov type PDE
  with terminal condition,  such as Hamilton-Jacobi-Bellman equations,
  when its solution is integrable. The time-reversal of that
PDE 
 can be seen as a Fokker-Planck PDE perturbed by a zero-order term of the
type 
\begin{equation}
\label{epdeIntro0}
\left \{
\begin{array}{l}
\partial_t v = \frac{1}{2} \displaystyle{\sum_{i,j=1}^d} \partial_{ij}^2 \left( (\sigma \sigma^\top)_{i,j}(t,x,v) v \right) - {\rm div} \left( b(t,x,v, \nabla v) v \right)
 + \Lambda(t,x,v,\nabla v) v\ , \quad \textrm{for}\  t\in (0,T]\ ,\\
v(0, \cdot) = \nu,
\end{array}
\right .
\end{equation}
where $\nu$ is a Borel probability measure,
 $\sigma: [0,T] \times \R^d \times \R \rightarrow  \R^{d\times m}$, 
$b: [0,T] \times \R^d \times \R \rightarrow  \R^d$,
$\Lambda: [0,T] \times \R^d \times \R \times \R^d \rightarrow \R$ and $\nabla$  denotes the gradient operator. 
When $\nu$ admits a density, we denote it by $v_0$.
The unknown function $v:(0,T] \times \R^d \rightarrow \R$
is supposed to verify $v(t,\cdot) \in L^1(\R^d)$,
considered as a subspace
of the space of finite Radon measures $\shm(\R^d)$. 
The idea consists in finding a probabilistic representation via the solution of an SDE
whose coefficients do not depend only on time and the position
of the {\it particle} but also on its probability law. 
The {\it target  microscopic equation}  one has in mind is
\begin{equation}
\label{eq:MckeanExtIntro}
\left\{
\begin{array}{l}
Y_t=Y_0+\int_0^t b\Big (s,Y_s,v(s,Y_s)\Big )ds+\int_0^t \sigma\Big (s,Y_s,v(s,Y_s)\Big )dW_s\\ 
Y_0\,\sim \nu\\ 
{\displaystyle \int \varphi(x)v(t,x)dx=\E\left [\,\varphi(Y_t)\,\,\exp\Big \{\int_0^t\Lambda\big (s,Y_s,{v}(s,Y_s),\nabla{v}(s,Y_s)\big )ds\Big \}\,\right]}\ ,\quad \textrm{for}\ t\in (0,T]\ ,
\end{array}
\right . 
\end{equation}
for any continuous bounded real-valued test function $\varphi$. 
Sometimes,  the third line equation of
\eqref{eq:MckeanExtIntro},
which involves a Feynman-Kac type representation,
is denominated the {\it linking equation}.
When $\Lambda=0$ in equation~\eqref{eq:MckeanExtIntro},  the linking
equation simply says that $v(t,\cdot)$ 
coincides with the density of the marginal distribution $\mathcal{L}^\P_{Y_t}$. 
In this specific case, equation~\eqref{eq:MckeanExtIntro} reduces to a
McKean SDE,  
which (as mentioned earlier) is  an SDE whose coefficients also depend
on the marginal law $\mathcal{L}^{\P}_{Y_t}$. With more general functions $\Lambda$,
the role of the linking equation is more intricate since the whole history of the process $(Y_s)_{0\leq s\leq t}$ is involved. This fairly general type of equations can be called {\it McKean Feynman-Kac Equation} to emphasise the fact that $v(t,x)dx$ now corresponds to a
 non-conservative Feynman-Kac measure. 
Those equations have been discussed in \cite{LOR2, LOR1}.
Various theorems of existence and uniqueness have been formulated,
pathwise or in law.
We drive the attention also on the recent paper
\cite{BKRRS} for another alternative approach.

\subsection{Stochastic control and McKean type SDEs}\label{sec:stoc-contr}

We conclude the section by mentioning a research field which has known a significant expansion in the recent literature,
i.e.\ the McKean type SDEs in the framework of the so called
{\it Nash equilibrium}
in mean-field games, 
see e.g.\ \cite{lasrylions, carmona-delarueI, carmona-delarueII,
  lacker, cardaliaguet, cardaliaguetPrinceton}.
One typical such  SDE describes the dynamics of an average agent,
who minimises running and terminal costs depending not only on her position,
but also on the distribution of the the infinitely many other agents.
A typical mathematical model, that we borrow from the introduction of
 \cite{cardaliaguet},
is 
constituted by
$$ dX_t = \alpha_t dt + \sqrt 2\sigma dW_t, $$
where $W$ is a standard Brownian motion and $\sigma\in\R$.
The typical agent determines her velocity $\alpha_s$ by
solving a stochastic control problem of the type
$$ \E\left[\int_0^T L(X_s,\mu_s, \alpha_s)ds + G(X_T, \mu_T)\right],$$
  where $L$ (resp.\ $G$) is a running (resp.\ terminal) cost,
  which also depends on the law $\mu$ of the position of a typical agent.
  The optimal control will be (heuristically) of the
  form
  $ \alpha_t = \alpha^\star(t,X_t),$
  and $\alpha^\star(t,x) = - \nabla_p H(x,\mu_t,\nabla_xu(t,x))$ where $H(x,m,p)$ is
  a corresponding Hamiltonian.
  After optimisation, the dynamics of a general agent
  is described by the McKean SDE
  $$ dX_t = - \nabla_pH (X_t,\shl^{\P}_{X_t}, \nabla_x u(t,X_t)) dt + \sqrt{2} \sigma dW_t,$$
  where $u$ solves a Hamilton-Jacobi-Bellman equation of the type
  $$ \partial_t u + \sigma \Delta u = H(x,\mu_t, \nabla_x u), \ (t,x) \in [0,T] \times \R^d, u(T,x) = G(x,\mu_T).$$
  For precise details one can consult  \cite{cardaliaguet},
  which explains why previous SDE constitutes the
   Nash equilibrium of an infinite number of players
   interacting via the running and terminal costs.

\section{Tools and methodology for the study of singular SDEs}
\label{sc:Tools}

In this section we illustrate some useful analytical and probabilistic tools that can be used to construct solutions of singular SDEs not of McKean type  and investigate their existence and uniqueness.

\subsection{Figalli-Trevisan superposition principle and uniqueness issues for the PDE}
\label{sc:figalli}

One modern tool, extremely useful  to study a general McKean SDE of the type
\eqref{general_mckean_SDE} is the so called Figalli-Trevisan superposition principle.
Let us consider the linear Fokker-Planck PDE from
Equation 2 in \cite{figalli}, i.e.
\begin{equation} \label{eq:Figalli2}
\left\{
\begin{array}{l}
\partial_t \mu_t= \sum_{i=1}^d \partial_i (- b_i \mu_t)+\frac{1}{2}\sum_{i,j=1}^d\partial_{ij}(a_{ij}\mu_t) \text{ in }[0,T]\times\R^d,\\
\mu_0=\nu,
\end{array}
\right.
\end{equation}
where $\nu \in \shp(\R^d)$ is a given initial probability law and $a: [0,T] \times \R^{d\times d} \rightarrow \R^d,
  b:  [0,T] \times \R^{d} \rightarrow \R^d$. 
A measure-valued function $\mu: [0,T] \rightarrow \shp(\R^d)$ is a solution of \eqref{eq:Figalli2} (in the sense of distributions) if, for every Schwartz test function $\varphi \in \shs(\R^d)$, it holds that
\begin{equation}
  \int_{\R^d} \varphi(x)\mu_t(dx)= \int_{\R^d} \varphi(x)\nu(dx) + \int_0^t \int_{\R^d}\bigg[ \sum_{i=1}^d b_i(s,x) \partial_i \varphi(x)+\frac{1}{2}\sum_{i,j=1}^d a_{ij}(s,x)\partial_{ij}\varphi(x) \bigg]\mu_s(dx)ds\label{eq:Figalli2sol}.
  \end{equation}
  We observe that  the PDE makes sense, since it is in divergence form and all the derivatives in \eqref{eq:Figalli2sol} do not involve the coefficients
  $a$ and $ b$,
 which are free to satisfy only measurability and mild integrability properties,
  see \eqref{eq:trevisan}.

An important question consists in associating a solution
of the previous PDE with a solution to the following martingale
problem, which involves the linear differential operators
\begin{equation}\label{eq:Lt}
    L_t := \sum_{i=1}^d b_i \partial_i  +\frac{1}{2}\sum_{i,j=1}^d a_{ij}\partial_{ij}, t \in [0,T].
\end{equation}
When $L$ is time independent it is the {\it generator}
of a Markovian process. The operator
$\partial_t + L_t$
is sometimes called {\it Dynkin operator} or {\it
  Kolmogorov operator}.

\begin{definition} \label{DefMP}
Let $Z = (Z_t)_{t \in [0,T]}$  be the canonical process on
the canonical space $\Omega_T = \shc([0,T],\R^d)$.
  A Borel probability measure $\P$
  is a  solution to the martingale problem related to 
  $(a,b)$
  and initial law $\nu\in \mathcal P(\R^d)$
 if the following holds:
 \begin{itemize}
 \item $Z_0 \sim \nu$ under $\P$;
 \item for every $\varphi\in C^2_b(\R^d)$
\begin{equation} \label{FigalliMP}
  \varphi(Z_t) - \varphi(Z_0) -\int_0^t(L_s\varphi)(s,Z_s)ds
  \end{equation}
  is a $\P$-local martingale.
  \end{itemize}
  For shortness, we denote by MP$(a,b, \nu)$ the martingale problem related
  to the coefficients $(a,b)$ and with initial condition $\nu$. 
  \end{definition}
  It is well-known that
  this martingale problem is equivalent to
  $Z$ being a (weak)  solution in law under $\P$ to the SDE
\begin{equation}
\begin{cases}
dZ_t=b(t,Z_t)dt+\sigma(t,Z_t)dW_t,\\
Z_0\sim\nu,
\end{cases}
 \end{equation}
 where $a(t,x) = \sigma\sigma^{\top}(t,x)$.

  Historically, the techniques employed to solve a McKean SDE
  (with pointwise dependence)
  were based essentially 
on the uniqueness of the
Fokker-Planck PDE \eqref{eq:Figalli2}.
 Nowadays 
 \emph{Figalli-Trevisan superposition principle}
has become a very efficient tool.
 This was introduced by A.\ Figalli in Theorem 2.6 in \cite{figalli}, 
 in the case of bounded coefficients and later extended in
 Theorem 2.5 of \cite{Trevisan}
 and finally generalised to the following case in
 \cite{bogachev_superposition}.
It was  inspired
 by the Ambrosio superposition principle obtained for the continuity equation
  (first order), see e.g.\ Theorem 12 of \cite{Ambrosio2004}.

  \begin{theorem}[Theorem 1.1 in \cite{bogachev_superposition}]
    \label{thm:figalli}
  Let $\mu: [0,T] \rightarrow \shp(\R^d)$ be a  solution of the PDE in \eqref{eq:Figalli2} which is  continuous with respect to the weak-star topology and 
  such that
\begin{equation} \label{eq:trevisan}
  \int_0^T \int_{\R^d}
  \frac{(\vert b(t,x) \vert + \vert a(t,x) \vert)}{1+\vert x\vert^2} \mu_t(dx) dt < +\infty ,
  \end{equation}
  Then there exists a Borel probability $\P$ on the canonical space $\Omega_T$, 
  which is the law of the solution of the martingale
  problem related to
  $(\sigma \sigma^\top,b)$ and initial condition $\nu$,  see Definition \ref{DefMP},
  such that  
\begin{equation} \label{eq:identification}
\int_{\R^d}\varphi(x) \mu_t(dx)= \E^\P[\varphi(Z_t)].
\end{equation}
for every bounded continuous function $\varphi: \R^d \rightarrow \R, t \in [0,T].$
\end{theorem}
\begin{remark} \label{rmk:fig}
  \begin{enumerate}
    \item
  It is obvious that if $a$ and $b$ are bounded then
any
measurable map $\mu:[0,T] \rightarrow \shp(\R^d)$  fulfills \eqref{eq:trevisan}.
\item
There exists a recent generalisation of Figalli-Trevisan superposition principle for non-local operators, in particular for Lévy measure driven operators, see e.g.\ \cite{nonloc_SP_Roeckner}.
  \end{enumerate}
\end{remark}

\medskip
Although the uniqueness of the Fokker-Planck equation is not required for the superposition principle, it remains a significant issue for Fokker-Planck type PDEs.
One consequence of the  uniqueness of the martingale problem
and previous superposition principle of Theorem \ref{thm:figalli}
is uniqueness of the Fokker-Planck equation \eqref{eq:Figalli2} a posteriori.
\begin{corollary} \label{cor:figalli-uniqueness} 
  If the martingale problem introduced
  in Definition \ref{DefMP} admits existence and uniqueness
then the PDE \eqref{eq:Figalli2}
has at most one solution among the
probability valued functions.
 \end{corollary}

As far as uniqueness of \eqref{eq:Figalli2} is concerned,
the literature includes various analytical a priori results addressing this question.
In the case of smooth coefficients we refer the reader to Theorem 3.2.6 in \cite{stroock_varadhan}, while, for example, for the case of drift-less PDE with bounded and measurable, but possibly degenerated,
 diffusion coefficient,  we state the theorem below from \cite{BCR2}. 

 \begin{theorem}[Theorem 3.1 of \cite{BCR2}]\label{P3.5}
Let $a$ be a Borel non negative bounded  function
on $[0,T] \times \R^d$. Let  $\mu^i : [0,T] \rightarrow \shm_+(\R^d)$,
$i = 1,2$,  be continuous 
with respect to the weak-star topology on $\shm(\R^d)$.
Let $\nu$ be an element of the space of finite and positive Radon measures $\shm_+(\R^d)$. 
Suppose that both $\mu^1$ and $\mu^2$ solve  the problem 
$\partial_t \mu = \Delta (a \mu) $
in the sense of distributions  with given initial condition
$\nu \in  \shm_+(\R^d)$.
More precisely, they solve
\begin{equation} \label{E3.5a}
\int_\R  \varphi(x) \mu_t(dx) = \int_\R 
 \varphi(x)  \nu(dx) + \int_0^t \int_\R \Delta \varphi(x)  a(s,x) \mu_s(dx)ds,
\end{equation}
for every $t \in [0,T]$ 
and any $\varphi \in \shs (\R^d) $.
If moreover 
 $\mu:= \mu^1 - \mu^2$ admits a density in 
 $ L^2([0,T] \times \R^d)$,
then $\mu_t^1 -  \mu_t^2$ is identically  zero for almost
 every $t$. 
\end{theorem}
\begin{remark}[Non-uniqueness of Fokker-Planck equation] \label{RMistake1}
  In the possibly degenerate case, as we said before, uniqueness does not hold in general in the class
  of measure-valued functions $\mu:[0,T] \mapsto \shp(\R^d)$.
  Remark 3.11 of \cite{BRR}  illustrates
  a case of function $a: \R \rightarrow \R$ such that
  the Fokker-Planck PDE \eqref{E3.5a} does not admit
  uniqueness, with initial condition $\nu= \delta_{0}$.
  This happens for $ a \ge 0$ such that $a(0) = 0$,
  $a$ is strictly positive on $\R\setminus\{0\}$ and
  $\frac{1}{a}$ is integrable in a neighborhood of zero.
  One solution is $\mu^1 \equiv \delta_0$, the second
  one $\mu^2$ is given by the marginal laws of a process
  constructed by Engelbert-Schmidt type arguments,
  see e.g.\ \cite{engelbert}.
\end{remark}

\subsection{Methodology for the probabilistic representation}
\label{sc:PR}

We recall that one of the targets of this paper 
consists in discussing different techniques
for the resolution of a McKean SDE of the type
\eqref{general_mckean_SDE}.
Since that equation involves a dependence of pointwise type,
the Wasserstein-type probabilistic methods fail.
In this case the well-posedness of  \eqref{general_mckean_SDE}
is generally performed
by studying the PDE  \eqref{general_nonlinear_FP}
via the so called {\it probabilistic representation}.

At the methodological level, to simplify, we suppose the coefficients
to be bounded and measurable, even though
in the literature one can allow some growth:
for instance \cite{BarbuRoeckSuperposition} shows existence
for a large class of SDEs of the type
\eqref{general_mckean_SDE} and
\cite{BarbuRoeckJFA} investigates uniqueness.
In the framework of those articles one
shows the existence of the corresponding
  PDE \eqref{general_nonlinear_FP}.   
\begin{theorem}\label{thm:PR}
  Let us suppose $b, \sigma \sigma^\top$
  bounded and measurable and suppose
  the existence of a solution (in the sense
  of distributions)
$\bar v\in L^1([0,T]\times\R^d)$
of the PDE \eqref{general_nonlinear_FP}. Then
$(X, \bar v)$ is a solution in law  
of the McKean SDE \eqref{general_mckean_SDE}.
\end{theorem}

\begin{proof}[Proof of Theorem \ref{thm:PR}]
  Since $\bar v\in L^1([0,T]\times\R^d)$ is  a solution  of
  \eqref{general_nonlinear_FP} in the sense of distributions then
  $v = \bar v$ solves the following associated linearised Fokker-Planck PDE  with drift $b(t,x,\bar v(t,x))$, and diffusion $\sigma(t,x,\bar v(t,x))$, where now $\bar v$ is known:
\begin{equation}\label{general_linearised_FP}
 \left\{
 \begin{array}{l}
   \partial_t v(t,x)= \frac{1}{2}\sum_{i,j=1}^d \partial_{ij}(a_{ij}(t,x,\bar v(t,x))v(t,x))-\sum_{i=1}^d  \partial_i(b_i(t,x,\bar v(t,x))v(t,x))\\
   v(0,x)=\bar v_0(x).
   \end{array}
   \right.
 \end{equation}
 Now making use of Figalli-Trevisan's superposition Principle in Theorem
 \ref{thm:figalli}, taking into account the equivalence result
 stated in Proposition 4.16 in \cite{karatzasShreve},
 there is a solution in law of
\begin{equation}
 \left\{
 \begin{array}{l}
   dX_t=b(t,x,\bar v(t,x))dt+\sigma(t,x,\bar v(t,x))dW_t\\
   {\rm The \ marginal \ law \ density \ of \ } X_t  \ {\rm is} \  v(t,\cdot)
   \\
   X_0\sim \bar v_0(x)dx.
 \end{array}
 \right.
\end{equation}
Since $v = \bar v$, 
it is clear that $(X,\bar v)$ is now a solution to \eqref{general_mckean_SDE} and this completes the proof.
\end{proof}

If the drift $b$ in \eqref{eq:trevisan} is distributional it is not clear how to formulate
a version of Theorem \ref{thm:figalli}.
A natural way to proceed consists in smoothing the drift $b$ into $b_\varepsilon$,
to make use of the superposition principle to the corresponding Fokker-Planck PDE, 
to show a tightness result of the solutions to the related SDEs and
a continuity results  of the solutions of the PDE with respect to the distribution $b$.
Details about the regularisation and the continuity results will be given Sections \ref{sec:linearSDE} and \ref{sec:McKean-point}.

\subsection{Zvonkin-type transformation}\label{ssc:zvonkin}

Following an idea of Zvonkin \cite{zvonkin},
a useful technique for treating
SDEs with distributional drift 
 is the so called Zvonkin  transformation,
which was first used in \cite{zvonkin} for studying strong existence
and pathwise uniqueness for SDEs without
Lipschitz coefficients.

Let us consider a $d$-dimensional SDE of the form
\begin{equation}\label{eq:orig_zvonkinSDE}
X_t =X_0+ \int_0^t  b(s, X_s) ds +  \int_0^t \sigma(s,X_s)dW_s,
\end{equation}
where $W$ is an $m$-dimensional Brownian motion,   $\sigma: [0,T]\times \R^{d}\to \R^m$ is continuous, bounded and strongly elliptic, but the drift $b:[0,T]\times \R^{d}\to \R^d$ is only bounded and measurable. 
Since $b$ is not Lipschitz, classical results on existence and uniqueness of the solution do not apply.
The original idea of Zvonkin \cite{zvonkin} was to find a transformation that would remove the drift,
in particular, he considered solutions of the Kolmogorov-type PDE
\begin{equation}\label{eq:orig_zvonkin}
\left\{
\begin{array}{l}
\partial_t  \phi^i+\frac{1}{2}{\rm Tr}(\sigma\sigma^{\top}\nabla^2\cdot \phi^i)+\nabla  \phi^i \cdot b = 0 \\ 
\phi^i(T, x) = x^i,
\end{array}
\right.
\end{equation}
for $i=1, \ldots, d$. Proceeding in a formal way for now, we assume $(X_t)$ to be a solution of SDE \eqref{eq:orig_zvonkinSDE} and we apply It\^o's formula to the process $\phi(t,X_t)$, where $\phi = (\phi^1, \ldots, \phi^d)$, and $\phi^i$ is a solution to the PDE in \eqref{eq:orig_zvonkin}, to get
 \begin{align*}
    \phi^i(t,X_t)-\phi^i(0,X_0)=&\int_0^t \partial_s\phi^i(s,X_s)ds+\int_0^t \nabla \phi^i(s,X_s)dX_s+\frac12 \int_0^t \nabla^2\phi^i(s,X_s)d[X,X]_s\\
    =&\int_0^t \bigg(\underbrace{ \partial_s\phi^i(s,X_s)+\frac12{\rm Tr}(\sigma\sigma^{\top}\nabla^2\phi^i(s,X_s))+\nabla \phi^i(s,X_s)\cdot b(s,X_s)}_{=0}\bigg)ds\\
    &+\int_0^t\nabla \phi^i(s,X_s)\cdot \sigma(s,X_s)dW_s\\
    =&\int_0^t\nabla \phi^i(s,X_s)\cdot \sigma(s,X_s)dW_s.
 \end{align*}
    In order to write an SDE for $Y_t:=\phi(t,X_t)$  we need to express $X$ in terms of $Y$, hence invertibility of $\phi$ is now fundamental, which was proven by
Zvonkin in Theorem 2, Section 3  of \cite{zvonkin}  on a small time interval $[0, \varepsilon]$.   
   Therefore, for all $t\in[0,\varepsilon]$, $Y$ has been shown to be a solution of an  SDE without drift given by
  \begin{equation}
    Y_t=Y_0+\int_0^t \nabla \phi(s,\phi^{-1}(s,Y_t))\cdot \sigma(s,\phi^{-1}(s,Y_s))dW_s,
  \end{equation}
 where $\nabla \phi \cdot \sigma$ is intended component-wise, in the sense that it is a vector with $i$-th component given by $\nabla \phi^i \cdot \sigma$.  
  We now highlight {\em two facts} about this transformation.
  \begin{enumerate}
    \item $\phi$ is unbounded because the terminal condition is unbounded, in particular it cannot belong
 to $L^p$ (or Sobolev type) spaces which are more analytically tractable.
 We recall nevertheless the fact that in \cite{issoglio_russoPDEa} one considers
 a framework of linear growth functions for exhibiting existence and uniqueness of a fixed point.
 \item  $\phi$ is invertible only on small time interval. 
\end{enumerate}
 To address those two issues, for simplicity we will restrict to the case of additive noise, i.e.
 when $\sigma \equiv \text{I}$.
 
 {\em About fact 1.} 
 To overcome the problem of  $\phi$ being unbounded, we can decompose the Zvonkin transformation
as
\begin{equation} \label{eq:phi}
  \phi(t, x) = x + u(t,x),
\end{equation}
  where $u^i$ is the solution of
\begin{equation}\label{eq:Zvonkin}
\begin{cases}
\partial_t  u^i+\frac{1}{2}\Delta u^i+\nabla  u^i \cdot b = -b^i, \\ 
u^i(T, \cdot) \equiv 0,
\end{cases}
\end{equation}
 and $b^i$ is the $i$-th component of the vector $b$. In this way, one can now prove that $u$ is bounded.

 The reader can check that the right-hand side of \eqref{eq:Zvonkin}, given by $-b^i$,  is a direct consequence of
 the chosen decomposition of $\phi$ and the linearity of the PDE \eqref{eq:orig_zvonkin}.
  Proceeding formally again, and applying  It\^o's formula to $u^i(t,X_t)$ where $X$ is a solution of
 \begin{equation}\label{eq:SSDE}
X_t =X_0+ \int_0^t  b(s, X_s) ds +  W_t,
\end{equation}
  we get
\begin{align*}
u^i(t, X_t) - u^i(0, X_0) 
  & = - \int_0^t  b^i(s, X_s) ds +  \int_0^t \nabla  u^i (s, X_s) dW_s,
\end{align*} 
and from this, one has the equality 
\[
\int_0^t  b^i(s, X_s) ds = u^i(0, X_0) - u^i(t, X_t)   +  \int_0^t \nabla u^i (s, X_s) dW_s.
\]
Substituting into the SDE \eqref{eq:SSDE}  we can write 
\[
X_t = X_0 + u(0, X_0) - u(t, X_t)   +  \int_0^t \nabla u (s, X_s) dW_s +W_t,
\]
which is equivalent to the original SDE but the drift has been removed. Notice that we have used again the shorter notation
$ \int_0^t \nabla u (s, X_s) dW_s$ for the vector with components $\int_0^t \nabla u^i (s, X_s) dW_s$.
Setting $ Y_t: = X_t+  u(t, X_t) = \phi(t, X_t)$ one observes that $Y$ solves
\[
Y_t = Y_0 + \int_0^t \nabla u(s, \phi^{-1}(s, Y_s)) dW_s +W_t,
\]
 with $Y_0: = \phi(0,X_0)$, 
which is a classical SDE for $Y_t$ provided that the inverse $y \mapsto \phi^{-1}(t,y)$ is well-defined.

{\em About fact 2.}  We already pointed out that this transformation is not invertible on the whole time interval $[0,T]$, for arbitrary $T\in\R$. As we mentioned above, Zvonkin in \cite{zvonkin} proved invertibility in a small time interval $[0,\eps]$.
A way to circumvent the non-invertibility issue
is to modify the PDE \eqref{eq:Zvonkin} by adding a dumping term of the form $\lambda u^i$  (or equivalently adding $\lambda (\phi^i - \text{id})$ in PDE \eqref{eq:orig_zvonkin}) for a suitable  parameter $\lambda$ which allows to get a small bound on  the norm $\nabla u$ and deduce that $\phi$ defined in \eqref{eq:phi}
is invertible.
 The suitable modified  PDE is 
\begin{equation}\label{eq:Zvonkin_modif}
\begin{cases}
\partial_t u^i+\frac12 \Delta u^i+\nabla u^i b = \lambda u^i -b^i \\ 
u^i(T, \cdot) \equiv 0,
\end{cases}
\end{equation}
with $\lambda>0$ large enough: details on the size of the parameter $\lambda$ can be found in \cite{flandoli_et.al14}. 
We remark that the dumping term $\lambda u$ produces somehow the same effect of a discounting term $e^{-\lambda t}$
in time. 
Proceeding as above one would get the equivalent SDE
\begin{equation}\label{eq:Y}
Y_t = Y_0 +\lambda \int_0^t u(s, \phi^{-1}(s, Y_s)) ds + \int_0^t \nabla u(s, \phi^{-1}(s, Y_s)) dW_s +W_t.
\end{equation}
This formulation was used in \cite{flandoli_et.al14} to define
the so called {\it virtual} solutions to SDEs where the drift is a
(Schwartz) distribution\footnote{For example $b\in C_T \C^{-\beta}$. For a precise definition of Besov-Holder spaces $\C^\alpha, \alpha \in \R$ see Section \ref{sc:Besov_spaces}. In this case, a natural space where one can find a fixed point for the PDE \eqref{eq:Zvonkin_modif} is $C_T \C^\alpha$ for some suitable $\alpha>0$.  See Section \ref{ssc:singularPDE} for the treatment of singular PDEs.}, see also Section
\ref{sec:solution_sing_SDE} for other types of solutions.

\subsection{Markov marginal uniqueness vs uniqueness
in law}\label{sc:ethier-kurtz}

In this  section we show how to get uniqueness of the martingale problem using one technique
described in Chapter 4 of \cite{ethier_kurtz},
which is a powerful tool in order to show uniqueness in law knowing uniqueness of
the time-marginals. A crucial exploited tool is the validity of Markov property for any initial time $t_0$
and initial condition a given Borel probability measure
$\nu$ on $\R^d$.
We denote again by $Z$ the canonical process on $\mathcal C([0,T]; \mathbb R^d)$.

Ethier and Kurtz in  \cite{ethier_kurtz} consider a more general formulation of martingale problem
pa\-ra\-me\-tri\-sed by a set of Borel functions that we denote $\sha_{EK}$. 
We propose here a reformulation of Theorem 4.2 of \cite{ethier_kurtz} which takes into account
the time-inhomogeneous case.
In our case $\sha_{EK}$ will be a  set of couples of Borel functions
$\varphi, g:[0,T] \times \R^d \rightarrow \R$.
A solution to the martingale problem  MP$(\sha_{EK}, \nu; t_0)$ will be
a Borel probability $\P$ on the canonical space 
such  that
\[
\varphi(t, Z_t) - \varphi(t_0,Z_{t_0}) - \int_{t_0}^t g (s, Z_s) ds
\]
is a local martingale under $\P$
and the law of $Z_{t_0}$ is $\nu$.
This context applies to the following cases.
\begin{enumerate}
\item[(a)] The classical Stroock-Varadhan martingale problem setting for instance
  $\sha_{EK}$ as the set of $(\varphi,g)$, where $\varphi \in  C^{2}(\R^d)$ with compact support   and $g(t,\cdot) = L_t \varphi$, where  $L_t$ was defined in \eqref{eq:Lt}.
\item[(b)] The rough martingale problem with distributional drift $b$ introduced in Definition  \ref{def:MK_MP}
with corresponding domain $\shd$.
In this case  $\sha_{EK}$ is a suitable  of couples $(\varphi,g)$
where $\varphi \in \shd$ and $g = L \varphi$.
   \end{enumerate}
   Next we introduce the key uniqueness property for marginals, denoted as \emph{Property M} in Definition \ref{def:Prop_P},
   which is sufficient to imply uniqueness of the solution, see Theorem \ref{thm:EK-extended} below.
   The definition and the theorem below are taken from \cite{issoglio_et.al24} and adapted
   to the Ethier-Kurtz general framework.
\begin{definition}[Property M in \cite{issoglio_et.al24}]\label{def:Prop_P}
Let ${\nu}$ be a Borel probability measure on $\R^d$, and $t_0 \in [0,T)$.
We say that {\em Property M} holds for MP$(\mathcal A_{EK}, \nu; t_0)$ if given $ \mathbb P$ solution to the MP$(\mathcal A_{EK} , \nu; t_0)$ then the marginal laws are uniquely determined, that is if $ \mathbb P^1 $ and $ \mathbb P^2$ are two solutions then 
\begin{equation}
\mathcal L^{\P^1}_{Z_t} = \mathcal L^{\P^2}_{Z_t} ,\qquad t\in [t_0,T].
\end{equation}
\end{definition}
The theorem below is an adaptation of Theorem A.2 in \cite{issoglio_et.al24}.

\begin{theorem}\label{thm:EK-extended}
If Property M holds for MP$(\sha_{EK},\nu; t_0)$  for every initial condition $\nu$ and every $t_0 \in [0,T)$, then we have uniqueness of  the MP$(\sha_{EK}, \nu; 0)$ for every $\nu$.
\end{theorem}

Next we show how to check Property M in the case when it is possible to reformulate the domain of the martingale problem using a suitable PDE in conformity with item (b) above.
We define $G$ as a linear space containing the functions
including those of the type
$g(t,x) = g_0(t) g_1(x)$, $g_0, g_1$ smooth functions with compact support. 
$\sha_{EK}$ is the set of couples $(\varphi,g)$ such that $g \in G$ and
\begin{equation}\label{eq:D}
\begin{cases}
L\varphi = g\\
  \varphi(T)=0.
\end{cases}
\end{equation}
Suppose moreover that one has the well-posedness of the solution of the PDE
\eqref{eq:D} for every $g \in G$.
Then Property M follows by an easy argument, as illustrated below. This argument is taken from the proof of Theorem 5.11 in \cite{issoglio_et.al24}.
Without loss of generality we set $t_0=0$.
Assume that $ \P^1$ and  $ \mathbb P^2$ are two solutions with initial condition $\nu$. Let now   $g \in G$ of previous product type
and $\varphi$ be the unique solution of $ L\varphi = g $ with terminal condition $\varphi_T =0$.
By the definition of solution of MP$(\sha_{EK},\nu;0)$  
\begin{equation}
\mathbb E^{\mathbb P^i} \left[ \underbrace{ \varphi_T( Z_T)}_{=0} -  \varphi_0(Z_0) - \int_0^T g_0(s) g_1(Z_s) d s\right] =0,\qquad \text{for } i=1,2,
\end{equation}
and $\mathcal L^{\P^1}_{Z_0} = \mathcal L^{\P^2}_{Z_0}=\nu$. As a result,
taking the expectations, we get
\begin{equation}
  \int_0^T g_0(s) \E^{\P^1} [g_1(Z_s)]ds =
\int_0^T g_0(s) \E^{\P^2}[ g_1(Z_s)]ds.
\end{equation}
Finally $\P^1$ and $\P^2$ have the same marginals for almost all $s\in[0,T]$,
then Property M is satisfied for MP$(L,\nu)$, for any initial condition $\nu$, as wanted. Since the canonical process is continuous,
Property M is satisfied for MP$(L,\nu)$, for any initial condition $\nu$,
and for all $s \in [0,T]$.

\subsection{Stochastic sewing lemma}\label{sec:ssl}

One fundamental tool in recent works related to SDEs with irregular coefficients is the so called {\it stochastic sewing lemma},
introduced by \cite{KhoaLe}. That fundamental lemma constitutes a stochastic extension of the deterministic sewing lemma introduced
in \cite{gubinelli,FeyelDeLaPradelle1} in order to show the convergence of Riemann sums in the context of rough integrals. Some applications
can be found for example in  \cite{athreya2020, butkovsky.et.al2021,  goudenege23,  anzelletti24}.

We highlight now the link between the stochastic sewing lemma and SDEs. As an example, let us consider the $\R^d$-valued SDE
\begin{equation}\label{eq:fBm}
   X_t = X_0  + \int_0^t b(X_s) ds+  G_t,
\end{equation}
where $G$ is a Gaussian process
and $b$ an irregular drift, but still a function for the moment.

 When $G$ is a standard Brownian motion, one basic idea introduced by Davie \cite{Davie} in the context of time-dependent bounded Borel drifts  is  to show uniqueness of the solution by shifting the solution process
 by the noise: the author proves in fact a stronger property than pathwise uniqueness, the so called
 {\it path by path uniqueness}, which means in our case that for almost all $\omega$
 the random differential equation \eqref{eq:fBm} admits uniqueness.
 In particular, setting $\psi_t : = X_t -G_t$, this process is supposed to solve the random ODE
\begin{equation} \label{eq:shifted}
  \psi_t = X_0 + \int_0^t b(\psi_s + G_s) ds.
\end{equation}
This shifting idea has been implemented by many authors for studying various
SDEs with irregular drift.
We recall that the crucial point for proving well-posedness is pathwise uniqueness,
since strong existence
can be often obtained by mean of existence in law and
Yamada-Watanabe theorem.

Inspired by  \cite{KhoaLe} we discuss \eqref{eq:shifted} in the case
when $G$ is a fractional Brownian motion $B^H$ and
$b$ is a locally bounded function. 
Given two solutions $\psi$
and $\bar \psi$, the idea consists in  showing
that $\psi - {\bar \psi}$ solve one linearised differential equation,
for which uniqueness can be established.

In the case that $ b$ is smooth one can write
\begin{equation} \label{eq:linearized}
  \psi_t -{\bar \psi}_t = \int_0^t  (\psi_r -{\bar \psi}_r) dV_r(b),
  \end{equation}
where
$$  V_t(\phi) = \int_0^t  \int_0^1 (\nabla \phi_r)(B^H_r +  \theta \psi_r + (1-\theta) \bar \psi_r) d\theta dr,$$
and $\phi \in C^{0,2}_b([0,T] \times \R^d)$.

The idea is to extend previous map $\phi \mapsto V(\phi)$
to all Borel bounded functions on $[0,T] \times \R^d$.
For a similar task \cite{KhoaLe} makes use of the stochastic sewing lemma, for which 
there exist several versions. Below we recall one of them, whose proof can be found in Theorem 2.1 and Theorem 2.3 in \cite{KhoaLe}.
Let $(\varphi_s)$ be a progressively measurable stochastic process.
Expressing $B^H_t = \int_0^t K(t,r)dW_r, $ where $K$ is a suitable  kernel and $W$ a standard Brownian motion, see e.g.\ 
Chapter 5 \cite{nualart}, it is possible to  show the 
property
\begin{equation} \label{eq:semig}
  \E^{\shf_s}\left[\int_s^t \nabla \phi_r (B^H_r + \varphi_s)dr\right]=\int_s^t
  (\nabla P_{\sigma^2_H(s,t)}\phi_r(\cdot))(\E^{\shf_s}\left[B^H_r\right] +
  \varphi_s)dr,
\end{equation}
where $(P_t)$ is the heat semigroup and $\sigma^2_H(s,t) = \int_s^t K^2(t,r) dr$.
We remark that the right-hand side makes sense also when
$\phi$ is not smooth, which allows to extend the left-hand side to
bounded Borel $\phi$.

Inspired by previous expression one defines a proxy
random field $(A_{s,t}(\phi))$ of
the conditional expectation
$\E^{\shf_s}(V_t(\phi) - V_s(\phi))$,
setting
\begin{equation} \label{eq:sha}
A_{s,t}[\phi] = \int_0^1
\int_s^t (\nabla P_{\sigma^2_H (s,t)}\phi_r(\cdot))
(\E^{\shf_s}\left[B^H_r\right] +
  \theta \psi_s + (1-\theta) \bar \psi_s))dr d\theta,
\end{equation}
having in view $\varphi_s =  \theta \psi_s + (1-\theta) \bar \psi_s$
for any $\theta \in [0,1]$.
Then the sewing lemma stated below
will provide a map $\phi \mapsto \sha(\phi)$
such that, for every $\phi$
$$ \sha_t(\phi) - \sha_s(\phi) \simeq A_{s,t}[\phi] ,$$
in the sense of \eqref{eq:thetaProxy}.
This allows to extend $\phi \mapsto V(\phi)$ from $C^{0,2}_b([0,T] \times \R^d)$
to bounded Borel functions $\phi$ and $V(\phi) = \sha(\phi)$
 continuously with respect to  some suitable norm.
Moreover  \eqref{eq:linearized} will  hold replacing
$V(b)$ with $\sha(b)$,
where the integral is in the non-linear Young sense
expressed as limit of a non-linear Riemann sum of the type
\begin{equation} \label{eq:RiemannYoung}
 \sha_t:= \int_0^t \sha_{ds}(\varphi_s) := \lim \sum_{i} A_{t_i, t_{i+1}} (\varphi_{t_i}),
\end{equation}
where the mesh of the subdivision $(t_i)$ (typically dyadic) goes to zero.
Indeed  equation \eqref{eq:linearized} with $V(b)$ replaced by its extension  $\sha(b)$
 will still conserve the uniqueness property.

\begin{lemma}[Stochastic sewing lemma of \cite{KhoaLe}]\label{lm:ssl}
 Let $(\mathcal F_t)_{t\geq0}$ be a filtration on a complete probability space $(\Omega,  \mathcal F, \P)$.
Let $p\geq 2$, $0\leq S<T$ and let $A_{s,t} \in L^p (\Omega)$ for $S\leq s\leq t \leq T$  be a two-parameter stochastic process, with $A_{s,t} \in \mathcal F_t$. Let $\delta A_{s,u,t}:= A_{s,t}-A_{s,u}-A_{u,t}$ for $s\leq u \leq t$. Suppose that for some $\varepsilon_1, \varepsilon_2>0$ and $C_1, C_2>0 $ we have 
\begin{align}  \label{eq:thetaProxy}
&\|\mathbb E [\delta A_{s,u, t}\vert \mathcal F_s]\|_{L^p (\Omega)} \leq C_1 |t-s|^{1 + \varepsilon_1}, \nonumber \\
& \\
&\| \delta A_{s,u, t} -\mathbb E [\delta A_{s,u, t}\vert \mathcal F_s]\|_{L^p (\Omega)} \leq C_2 |t-s|^{\frac12 + \varepsilon_2}. \nonumber
\end{align}
Then there exists  a unique adapted process $(\mathcal A_{t})_{t\in[S,T]}$ such that $\mathcal A_S=0$, $\mathcal A_t$ is $\mathcal F_t$-measurable and $L^p(\Omega)$-integrable,  and for all $S\leq s\leq t \leq T$ 
\begin{align*}
&\| \mathcal A_{t} -\mathcal A_{s} - A_{s,t}  \|_{L^p (\Omega)} \leq N_1 |t-s|^{1 + \varepsilon_1}+ N_2 |t-s|^{\frac12 + \varepsilon_2},\\
&\|\mathbb E [\mathcal A_{t} -\mathcal A_{s} - A_{s,t}  \vert \mathcal F_s]\|_{L^p (\Omega)} \leq N_1 |t-s|^{1 + \varepsilon_1},
\end{align*}
for some constants $N_1>0, N_2>0$.
If moreover $ \| A_{s, t} \|_{L^p (\Omega)} \leq C_3 |t-s|^{\frac12 + \varepsilon_3}$ then it also holds 
\[
\| \mathcal A_{t} -\mathcal A_{s}\|_{L^p (\Omega)} \leq N_1 |t-s|^{1 + \varepsilon_1} + N_3 |t-s|^{\frac12 + \varepsilon_3},
\]
for some constant $N_3>0$. 
\end{lemma}

The sewing lemma technique, useful to study irregular SDEs in the strong sense, 
can be applied to the study of McKean SDEs, see
Sections \ref{sec:Boltzmann} and \ref{sec:McKean-point}.

\section{Singular McKean-Vlasov SDEs: the function case}\label{sec:4}

In this section we consider the case where the drift and the diffusion coefficient are not Lipschitz in time and space,
for example when they are $L^p$-$L^q$ functions or bounded measurable functions.
In particular, in Sections \ref{sec:LpLq_w} and \ref{sec:LpLq_p} we review extensions and variations of the well-known work by Krylov-R\"ockner  \cite{kry-rock} in the McKean-Vlasov framework, when the dependence on the law may require the existence of the density or not.
In Section \ref{sec:pm} we review a class of generalised porous media equations,
Section \ref{sec:radially} focuses more specifically on the radially symmetric case
and Section \ref{sec:pLap} mentions the probabilistic representation of a $p$-Laplace type PDE, viewed as Fokker-Planck type equation.
  Finally we illustrate two classes of atypical  McKean-Vlasov SDEs that arise, respectively, from a reformulation of SDEs with drift  depending on conditional laws in Section \ref{sc:con_exp} and  time-reversal of diffusions in  Section \ref{sc:time-rev}.

\subsection{$L^p$-$L^q$: Wasserstein-type dependence}\label{sec:LpLq_w}

In this section and in the following one,  we focus on the class of well-known $L^p$-$L^q$ coefficients introduced in the seminal paper \cite{kry-rock} by Krylov and R\"ockner, where they proved strong well-posedness for SDEs not of McKean-type. We recall that  $L^p$-$L^q$ functions  are elements of $ L^q([0,T];L^p(\R^d))$, where the parameters $(p,q)$ satisfy a version of Ladyzhenskaya-Prodi-Serrin condition, in particular $(p,q) \in \mathcal{K}$ where
\begin{equation}
 \mathcal{K}:=\bigg\{(p,q)\in (1,\infty)^2 \bigg| \frac{d}{p}+\frac{2}{q}<1 \bigg\}.
\end{equation}
Notice that  the  weaker condition $\frac{d}{p}+\frac{2}{q}<2 $ is usually sufficient if one is interested only in weak well-posedness.  
Throughout this section, we consider McKean-Vlasov SDEs of the form
\begin{equation}\label{eq:lplq}
dX_t=b(t,X_t,\mathcal L^\P_{X_t})dt+\sigma(t,X_t,\mathcal L^\P_{X_t})dW_t,
\end{equation}
where  assumptions on $b$ and $\sigma$ will be specified on a case by case basis below.

Amongst the first articles dealing with \eqref{eq:lplq}   we cite the pair of papers \cite{HuangWangSPA,HuangWang2021}. In \cite{HuangWangSPA} 
 the authors establish weak and strong well-posedness for \eqref{eq:lplq}
under different assumptions for
the drift and diffusion coefficient, see Theorem 2.1 in  \cite{HuangWangSPA}.
In particular  weak existence is proved via tightness arguments and Krylov's type estimates.
The aforementioned theorem also states strong well-posedness
when the $L^\infty$-norm  of the drift
$b$ and the diffusion $\sigma$ are
Lipschitz with respect to  the spatial  and the measure component, the latter with respect to the $\theta$-H\"older Wasserstein distance
for some $\theta \ge 1$ and the drift is bounded by an $L^p$-$L^q$ function in time and space, uniformly with respect to the measure.
Concerning pathwise uniqueness, given two solutions, they also naturally solve
two non-McKean SDEs.
The authors use a Zvonkin transformation
for those  SDEs and later
 a stochastic version of Gr\"onwall lemma involving the Wasserstein distance.
This identifies the laws of the solutions and finally they use  classical pathwise uniqueness for non-McKean SDEs to show the equality of the processes.
In the other work \cite{HuangWang2021}, the diffusion coefficient is assumed independent of the law and the authors study strong and weak well-posedness for both initial values and initial  distributions under  weaker assumptions on the drift,
in particular Lipschitz continuity in total variation distance and H\"older-type condition in the spatial component. Under slightly stronger assumptions on the drift than in  \cite{HuangWangSPA} they also obtain stability estimates.
The proof for  well-posedness follows the same steps as the one provided in \cite{HuangWangSPA}, whilst for stability estimates Harnack-type inequalities and stochastic Gr\"onwall inequality are employed.

The assumptions in \cite{HuangWangSPA} are relaxed also in paper \cite{RocknerZhang}. In particular, they  first prove weak existence of a solution to the McKean-Vlasov SDE \eqref{eq:lplq}  in Theorem 3.6 of \cite{RocknerZhang} when $\sigma \sigma^\top$ is non-degenerate (uniformly in time and measure variable), H\"older continuous in space uniformly in the measure, and the drift belongs to a localised version of  $L^p$-$L^q$ in time and space, uniformly in the measure.
This is proved with tightness techniques. If furthermore, the gradient of the diffusion coefficient also belongs to a localised version of  $L^p$-$L^q$ in time and space, uniformly in the measure,
they prove also existence of
a strong solution in Corollary 3.7 in \cite{RocknerZhang}.
For what concerns pathwise uniqueness, the authors
have two separate results. In Theorem  4.2 of \cite{RocknerZhang} they  postulate
a Lipschitz dependence in localised version of $L^p$-$L^q$ 
(resp.\ $L^\infty$) for $b$ (resp.\ $\sigma$) 
in the measure variable, again with respect to the $\theta$-H\"older Wasserstein distance.
In Theorem 4.3 of  \cite{RocknerZhang} they instead assume (on top of strong  existence assumptions
in  Corollary 3.7 of \cite{RocknerZhang}) that $\sigma$ is independent of the measure and that $b$ has  Lipschitz dependence in a localised version of $L^p$-$L^q$ 
this time  with respect to the measure in a weighted  total variation  distance.

In \cite{deRaynal} the author establishes existence and uniqueness for a special class of McKean-Vlasov SDE with H\"older continuous drift with respect to the measure component
in the Wasserstein sense.  In particular,  both the structure of the drift and diffusion coefficients
must be of the form $f(t,x,\langle \phi, \mathcal L_X^\P \rangle)$ for some given function $\phi$, where the function  $f$ is supposed to be differentiable in the third (real) variable;  the function $\phi$ is allowed to be
H\"older for the drift, and Lipschitz for the diffusion. Consequently $\nu\mapsto b(t,x,\langle \phi, \nu \rangle)$ is indeed H\"older continuous  the Wasserstein sense.
Finally the diffusion must be  uniformly non-degenerate, Lipschitz in  the space variable uniformly in time, and its derivative with respect to the third component must be H\"older continuous in space uniformly in time.
Under these conditions the author proves strong well-posedness using the Zvonkin transformation. Notice that the underlying PDE here is infinite dimensional since the third argument is a measure and hence Lions derivatives are employed.
The PDE is solved using a parametrix expansion of the transition density of the McKean-Vlasov process.

In Theorem 3.2 in  \cite{GaleatiLing23} the authors prove strong well-posedness of McKean-Vlasov SDE \eqref{eq:lplq}, under the same non-degeneracy condition
 for the diffusion matrix $\sigma \sigma^\top$ as \cite{RocknerZhang}, an $L^p$-$L^\infty$ condition on $b$ and $\nabla \sigma$ uniformly in the measure variable.
Moreover, with respect to the measure, $\sigma$ fulfills the assumption
$\| \sigma(t, \cdot, \mu) - \sigma(t, \cdot, \nu) \|_{L^p} \leq \mathcal W_{\theta}(\mu,\nu)$, for some $\theta \ge 1$.
The main novelty here however is  that the drift $b$ is Lipschitz only with respect to the measure variable in Wasserstein distance uniformly in the
the Sobolev  $W^{-1,p}$-norm,  uniform 
in  time, 
hence much weaker than papers reviewed up to this point.
Strong well-posedness relies on stability estimates of non-McKean-Vlasov SDEs with $L^p$-$L^q$ coefficients, derived in Theorem 2.1 of \cite{GaleatiLing23},
and these estimates are used for a standard fixed point argument similar to the classical one reviewed in Section \ref{sec:sznitman}.

\subsection{$L^p$-$L^q$: Density dependence type}\label{sec:LpLq_p}

For what concerns $L^p$-$L^q$ drifts with density dependence of convolutional type,
we mention \cite{Olivrichtoma23, Olivrichtoma},
where the authors study a problem of propagation
of chaos for moderately interacting particles,
either on $\R^d$ or on  the torus $\mathbb T^d$, respectively, and their limiting McKean equation.
The two models are similar, the first one (in infinite volume) necessitates a cutoff,
the second one is in finite volume.
\cite{Olivrichtoma} includes  the case when Fokker-Planck PDE
coincides with Burgers equation, i.e. when the convolution kernel
is a Dirac measure.

We review here the case on $\R^d$  presented in \cite{Olivrichtoma23}.
The particle system considered is
\begin{equation} \label{eq:cutoff}
    \left\{
    \begin{array}{l}
         dX_t^{i,N}=F_A \left[\frac{1}{N}\sum_{k=1}^N(K*V^N)(X_t^{i,N}-X_t^{k,N}) \right] dt+\sqrt{2}dW_t^i,\quad \forall i=1,\dots,N,\quad t\leq T\\
         X_0^{i,N}\sim v_0(x)dx,
    \end{array}
    \right.
\end{equation}
where $(V^N)$ is the sequence of 
probability densities on $\R^d$ as the one introduced in equation
\eqref{eq:V_N}   in Section \ref{sec:pdd}, converging to the Dirac measure at zero.
  $K(x):=-C\frac{x}{|x|^d}$ is a singular kernel and $C>0$, and $F_A$ is a cutoff function  which is constant 
 equalizing the identity inside the hypercube $[-A,A]^d$.
 The cutoff function $F_A$ will not be present in the limit \eqref{eq:kellerSDE} below, because $A$ can be chosen large enough
 so that the argument of $F_A$ in \eqref{eq:cutoff} belongs to the  hypercube $[-A,A]^d$.
The McKean-Vlasov SDE resulting in the limit as $N$ goes to infinity is indeed
\begin{equation}\label{eq:kellerSDE}
    \left\{
    \begin{array}{l}
         dX_t=(K*v(t))(X_t)dt+\sqrt{2}dW_t\\
         X_0\sim v_0(x)dx,
    \end{array}
    \right.
\end{equation}
where $v\in C_TL^p(\R^d)$ and $v(t)$ is the law density of $X_t$ for every $t \in [0,T]$.  In fact it is naturally  expected,
that  $v$ is a solution to the non-linear Fokker-Planck PDE,
\begin{equation}\label{eq:kellerPDE}
    \left\{
    \begin{array}{l}
         \partial_t v(t,x)=\Delta v(t,x)-\textrm{div}(v(t,x)(K*v(t))(x)),\quad t\in (0,T], x\in\R^d\\
         v(0,x)=v_0(x),\quad x\in\R^d.
    \end{array}
    \right.
\end{equation}
In fact, the authors are able to show existence and uniqueness of a local in time solution $v$  of the non-linear Fokker-Planck PDE,  that is,  the terminal time $T$ must be smaller than a certain $T^*$.
Afterwards they formulate a stochastic partial differential equation fulfilled by  the mollified empirical measure $v^N:=V^N*\mu^N$, where $\mu^N=\frac{1}{N} \sum_{i=1}^N \delta_{X^i}$:
this way of proceeding turns out to be useful to study the convergence of the sequence $(v^N)$ towards the mild solution $v$ of \eqref{eq:kellerPDE}, which is done by simply comparing the mild formulations. 
This strategy was first implemented in \cite{FlandoliOlivera2019} in the framework of particle approximation of a Fisher-Kolmogorov-Petrowskii-Piskunov PDE
and it was used by the same authors
also in \cite{FlandoliOlivera2020} for the  approximation of 2D vorticity Navier-Stokes equations.

\subsection{About generalised porous media equations}\label{sec:pm}

Let us now discuss  the case of  
{\em generalised porous media equations}, which are equations of the form
\begin{equation} \label{PME}
  \partial_t v= \Delta (\beta(v)), \ v(0,\cdot) = v_0(\cdot),
  \end{equation}
where $v_0 \in (L^1 \cap L^\infty)(\R^d)$, $\beta$ is monotone increasing, continuous at zero,  such that $\beta(0) = 0$,
possibly discontinuous.  The term {\em generalised} arises from the fact that when $\beta(v) = {\rm sign}(v) |v|^m$ with $m > 1$, equation \eqref{PME} corresponds to the classical \emph{porous media equation}.
The coefficient $\beta$ is associated with
a graph (multi-valued function) obtained by ``filling the gaps''.
In this case there is $\Phi:\R \rightarrow \R_+$ such that
\begin{equation} \label{eq:Phi}
  \beta(v)  = \Phi^2(v) v.
  \end{equation}
The study of the PDE makes use of Benilan-Crandall monotonicity techniques,
see e.g.\ \cite{benilan}. For simplicity, we will set $d=1$.
Suppose for simplicity  that $\Phi$ is bounded.
By  Proposition 3.4 in  \cite{BRR}, using the techniques in  \cite{BenilanCrandall}, there exists a unique solution $v \in (L^1 \cap L^\infty)([0,T]\times \R)$
 of PDE \eqref{PME} in the sense of distributions.
In particular,  we have
\begin{equation*} 
\int v(t,x) \varphi(x) dx = \int v_0(x) \varphi(x) dx + \frac{1}{2}
 \int_0^t ds \int \eta_v(s,x) \varphi''(x) dx, \quad   \forall \varphi \in \shs(\R), 
\end{equation*}
for some $\eta_v(t,x) \in \beta(v(t,x))\ {\rm for} \ dt \otimes  dx-{\rm a.e.}
 \ (t,x) \in 
    [0,T] \times \R,  \nonumber
$ such that $\eta_v \in (L^1 \cap L^\infty)([0,T]\times \R)$.

The probabilistic representation of the solution to the generalised porous media equation \eqref{PME} consists
in finding a solution $\P$
of the McKean SDE
\begin{equation} \label{McKeanPME}
\left \{
  \begin{array}{l}
    Z_t = Z_0 + \int_0^t \Phi(v(s,Z_s))dW_s, \  t \ge 0, \\
   \shl^\P_{Z_t}(dx) = v(t,x) dx, 
  \end{array}
  \right.
  \end{equation}
where $W$ is a standard Brownian motion.
This representation is in fact a direct consequence of Figalli-Trevisan superposition principle
stated in Theorem \ref{thm:figalli}. Indeed the solution $v$ of the PDE \eqref{PME}  
solves also the linearised equation
  \begin{equation} \label{PMELinearised}
  \partial_t v= \Delta (\phi(t,x) v),
  \end{equation}
  where $\phi(t,x) = \Phi^2(v(t,x))$.
By the aforementioned principle, 
  the probabilistic representation of 
this linear PDE is automatically a probabilistic representation
of \eqref{PME}.
This provides us with a solution to McKean SDE \eqref{McKeanPME}
therefore it establishes existence.

An interesting problem is the uniqueness of a general McKean SDE.
In fact, uniqueness may fail,
even when the coefficients of the SDE do not depend on the law,
see Remark \ref{RMistake1} for an example  on non-uniqueness in a Fokker-Planck equation,
when the diffusion coefficient is possibly degenerate.
In the case of a porous media equation \eqref{PME} with non-degenerate bounded  diffusion coefficient,
i.e. $ 0 < c \le \Phi \le C $, where $c, C$ are real constants,  uniqueness was proved in  \cite{BRR2}. 
Indeed if $\P^1, \P^2$ are two solutions then the marginals 
$v^1, v^2:[0,T] \times \R \rightarrow \R$ are solutions
of the porous media equation, which is well-posed,
therefore $v^1= v^2$, that we call $v$. 
At this point the uniqueness of the McKean SDE \eqref{McKeanPME} reduces
to the case of the  ordinary SDE
\begin{equation} \label{McKeanSDE-PME}
  Y_t= Y_0 + \int_0^t \Phi(v(s,Y_s))dW_s, t \in [0,T],
  \end{equation}
  Since $ 0 < c \le \Phi \le C $ the SDE above is well-posed
  see Exercises 7.3.2–7.3.4 in
\cite{stroock_varadhan}.

\vspace{5pt}

As previously mentioned, when $\beta(v) = {\rm sign}(v) |v|^m$ with $m > 1$, equation \eqref{PME} corresponds to the classical \emph{porous media equation}. The case $m = 1$ recovers the classical \emph{heat equation}, while for $0 < m < 1$, \eqref{PME} becomes the so-called \emph{fast diffusion equation}. When $m < 0$, it is referred to as the \emph{superfast diffusion equation}. Another motivating singular case occurs when $\beta(v) = |\log(v)|$, which leads to the \emph{fast logarithmic equation}, see \cite{VazquezLog}. Several developments in this directions also exist in the framework of the stochastically perturbed equation,
see e.g.\ \cite{ciotir-goreac}.

For $m > 1$ and $0 < m < 1$, there exists an important class of explicit solutions known as \emph{Barenblatt-Pattle solutions}. A probabilistic representation of these solutions for the porous media equation was provided in \cite{Ben_Vallois}, while \cite{BCR2} extended these techniques to fast diffusion equations, at least in the range $m > \frac{2}{3}$. An important feature of fast diffusion equations is the phenomenon of mass loss: there exist non-negative solutions $v$ such that, although the initial mass satisfies $\int_{\mathbb{R}^d} v_0(x)\,dx = 1$, one has $\int_{\mathbb{R}^d} v(t,x)\,dx < 1$ for later times. According to Section 5.5.4 in \cite{VazquezDecay}, mass conservation is lost when $0 < m < \frac{d-1}{d}$; see also Chapter 9 of \cite{VazquezBookPorous}.
\medskip

\medskip
The existence (and uniqueness, in the non-degenerate case $\Phi(v) \neq 0$) for the McKean SDE \eqref{McKeanPME}, via probabilistic representation,
 was first investigated in \cite{BRR, BRR2} for dimension $d=1$. The fundamental tool was the uniqueness of an
associated linear Fokker-Planck PDE of the type \eqref{P3.5}, where the diffusion coefficient $a(t,x)$ is given by
$\Phi^2(v(t,x))$ (and $b =0$), with $v$ being the solution to \eqref{PME}. In this setting, it was shown that, if $X$ is a diffusion process
with non-degenerate bounded coefficients, then the law $t \mapsto \mathcal{L}^\mathbb{P}_{X_t}(dx)$ can be expressed as $t \mapsto v(t,x)\,dx$, where $v \in L^2([0,T] \times \mathbb{R})$.

Subsequent work extended these results to higher dimensions ($d > 1$) and to other Fokker-Planck type equations, where existence could be handled using monotonicity methods and the Figalli-Trevisan superposition principle. A typical example, treated for instance in \cite{BarbuRockSIAM} and \cite{BarbuRoeckJFA}, is the PDE
\[
\partial_t v = \Delta(\beta(v)) - \nabla\cdot(b(v)v), \quad v(0,\cdot) = v_0(\cdot),
\]
where $b: \mathbb{R} \to \mathbb{R}^d$ is bounded and continuous. 
Generalised porous media PDE with Neumann boundary conditions and their probabilistic representation via 
reflected McKean SDEs on the half-line were investigated in \cite{ciotir}.

Concerning numerical simulations, the first step of simulating generalised porous media equation
with general (possibly discontinuous) coefficient $\beta$ was performed by \cite{cuvelier1} (in the one-space
dimensional case) and by \cite{cuvelier2} (in the multidimensional case),
by making use of the probabilistic representation via a McKean SDE, making use of particle system approximation
and density estimates.

\subsection{About the probabilistic representation of radially symmetric solutions
  of a generalised porous media type equation.}
\label{sec:radially}

We summarise here some significant results of \cite{cuvelier2}.
Let $\beta: \R \rightarrow \R$ as in Section \ref{sec:pm} and consider again
the generalised porous media type equation \eqref{PME},
supposed here to be continuous.
We suppose the existence of $\ell \ge 1$ such that
$$ \vert \beta(v) \vert \le C_\beta \vert v\vert^\ell.$$
Let $ v_0 \in (L^1 \cap L^\infty)(\R^d)$ to be radially symmetric.
Let $\Phi:\R \rightarrow \R$ as in \eqref{eq:Phi} and $\Phi_d(v)=\Phi(v) I_{d}$ and  $I_d$ is the unit matrix on $\R^d$.
We consider an $\R^d$-valued process $Y$
solution of
\begin{equation}\label{NLSDE_D}
  \left\{
      \begin{array}{l}
       Y_t=Y_0+ \int_0^t \Phi_d(v(s,Y_s))dW_s\\
        v(t,\cdot) \text{ is the law density of }Y_t,  \forall t\in(0,T]\\
       v(0,\cdot)=v_0.
      \end{array}
    \right.
\end{equation} 
The existence of a process $Y$ verifying \eqref{NLSDE_D} can
be provided by Figalli-Trevisan superposition principle.
By Proposition 1.4 of \cite{cuvelier2}, if $(Y,v)$ is a solution
of \eqref{NLSDE_D} then
$v$ is a solution in the sense of distributions of
\eqref{PME}. Also, by Proposition 5.4 of \cite{cuvelier2} there is
$\tilde v:[0,T] \times \R \rightarrow \R$ such that
$ v(t,x) = \tilde v(t,\vert x \vert^d)$.

Let us consider the process
\begin{equation} \label{eq:S}
  S_t := \vert Y_t \vert^d.
\end{equation}
  According to Lemma 5.6 in \cite{cuvelier2}
the law of $S_t$ admits a density $\mathrm{v}(t,\cdot)$  which fulfills
\begin{equation}\label{nu_u_tild_rel}
\mathrm v(t,\rho)=\frac{\displaystyle{{C}}}{\displaystyle{d}}\tilde{v}(t,\rho), \ \ \ \forall \rho >  0,
\end{equation}
where
\begin{equation}\label{Const_Parite}
    {C}=\frac{\displaystyle{2(\pi)^{\frac{d}{2}}}}{{\Gamma(\frac{{d}}{{2}})}},
  \end{equation}
  $\Gamma$ is the usual Gamma function and $\tilde v$
is such that $\tilde v(t,\cdot) $ is the law of $Y_t$.
On the other hand, if $w$ is defined by the right-hand side of \eqref{nu_u_tild_rel}, then
$w$ is a solution, in the sense of distributions, of
hand \begin{equation}\label{PDE_bessel}
    \partial_t w(t,\rho)={C}(1-d) \partial_{\rho}\left[\rho^{1-\frac{2}{d}}\beta\left(\frac{\displaystyle{d}}{\displaystyle{{C}}}w(t,\rho)\right)\right]+
    \frac{\displaystyle{{C} d}}{\displaystyle{2}}\partial_{\rho\rho}\left[\rho^{2-\frac{2}{d}}\beta\left(\frac{\displaystyle{d}}{\displaystyle{{C}}}w(t,\rho)\right)\right],
    \end{equation}
    with initial condition $w_0=\frac{{C}}{d}\tilde{v}_0$, where $\tilde{v}_0$
such that $v_0(x) = {\tilde v_0}(\vert x \vert^d)$
and    ${C}$  is the constant defined in \eqref{Const_Parite}.
 
Below we provide a natural probabilistic representation of \eqref{PDE_bessel}.
It is Proposition 5.10 in \cite{cuvelier2}.

\begin{proposition}\label{EDS_Bessel_Prop}
  Let $v_0$ be radially symmetric and $\tilde{v}_0$
  such that $v_0(x) = {\tilde v_0}(\vert x \vert)$. 
  
  For $y\neq 0$, $\rho>0$, we set
\begin{eqnarray}\label{Phi_1&2}
  \left\{
      \begin{aligned}
\Psi_1(\rho,y)&=(d^2-d)\rho^{1-\frac{2}{d}}\Phi^2\left(\frac{\displaystyle{d}}{\displaystyle{{C}}}y\right),\\
\Psi_2(\rho,y)&=d\rho^{1-\frac{1}{d}}\Phi\left(\frac{\displaystyle{d}}{\displaystyle{{C}}}y\right),
\end{aligned}
    \right.
\end{eqnarray}
where, ${C}$ is defined by \eqref{Const_Parite}.

\begin{description}
\item(i) Suppose that $(Z_t)$  is  a non-negative  process solving the non-linear SDE defined by
\begin{equation}\label{SDE_S_formelle}
  \left\{
      \begin{array}{l}
       Z_t=Z_0+ \int_0^t \Psi_2(Z_s,p(s,Z_s))dB_s+\int_0^t \Psi_1(Z_s,p(s,Z_s))ds\\[0.2cm]
        p(t,\cdot)\text{ is the law density  of } Z_t ,\  \forall t\in(0,T], \ Z_0\sim \frac{{C}}{d}\widetilde{v_0}.
      \end{array}
    \right.
\end{equation}
Then,  $p$  is a solution, in the sense of
distributions, of the PDE \eqref{PDE_bessel} with initial condition $\frac{{C}}{d}\tilde{v}_0$.

\item(ii) If $S$ is defined by \eqref{eq:S}, with marginal laws denoted by
$\mathrm v$, then $S$ verifies \eqref{SDE_S_formelle} with $Z=S$ and $p=\mathrm v$.
\end{description}
\end{proposition}

Previous representation of radially symmetric porous media
type solutions was exploited in Section 8.\ of \cite{cuvelier2}
to simulate those solutions via Monte-Carlo methods and non-parametric
density estimates,
supposing the validity of the chaos propagation phenomenon.

\subsection{$p$-Laplace equation and $p$-Brownian motion}
\label{sec:pLap}

An important PDE appearing in the physical literature is the so called
$p$-{\it Laplace equation},  see e.g. \cite{Kamin, lindqvist}. 
There are two cases in the literature: $p > 2$ (degenerate case)
and $p < 2$ (singular). The case $p = 2$ corresponds to the
classical heat equation.
The recent remarkable contribution \cite{RehmeierPLap}
provides its probabilistic representation in the degenerate case.
The $p$-Laplace equation is
$$ \partial_t v(t,x) = \Delta_p v(t,x), \quad (t,x) \in (0,+\infty) \times \R^d,$$
where
$$ \Delta_p v := \text{div} (\vert \nabla v \vert^{p-2}\nabla v).$$
The corresponding McKean SDE
is
\begin{equation} \label{eq:p-BM}
  \left \{
\begin{array}{l}
  dX_t = \nabla (\vert \nabla v(t,X_t)\vert^{p-2}) dt
           + \sqrt{2} \vert \nabla v(t,X_t)\vert^\frac{p-2}{2} dW_t \\
  v(t,x) \text{ is the law density of } X_t, \quad \forall t> 0.
\end{array}
\right.
\end{equation}
This process is very much related to what the authors
call the $p$-{\it Brownian motion}, see Section 5.3 of
\cite{RehmeierPLap}.

\subsection{Conditional expectation dependence as McKean equations}\label{sc:con_exp}

An SDE which involves conditional  expectations can be often reformulated as a McKean type SDE. We illustrate here two such examples.

 A first example is given in \cite{bossy2011conditional} where the  McKean type SDE for $X_t=(V_t,U_t)\in \R^d$ is 
\begin{equation}\label{eq:bossy}
\left\{
\begin{array}{l}
V_t=V_0+\int_0^t U_s ds\\
U_t=U_0+\int_0^t B(V_s,U_s; \rho_s) ds+\int_0^t\sigma(s,V_s,U_s)ds\\
\rho_t \text{ is the law density  of } X_t=(V_t,U_t) \text{ for all }t\in[0,T],
\end{array}
\right.
\end{equation}
where $d=2n$  and for all $t\in[0,T]$ $V_t,U_t\in\R^n$. In \eqref{eq:bossy} $B(v,u; \rho)$ is defined in such a way that
replacing $\rho$ with the law of $(V_s,U_s)$ we get
\begin{equation*}
\displaystyle B(v,u;\rho) = b * \rho_{U_s \vert V_s = v}(u).
\end{equation*}
The main result in \cite{bossy2011conditional}  is Theorem 3.2 which shows uniqueness of a mild solution to the associated PDE,
see Section 4 of \cite{bossy2011conditional}. Then, in Section 5 of \cite{bossy2011conditional} existence
and uniqueness
of a martingale solution for
the particle system is proved.
The existence of a martingale solution for the limit
is established through a propagation of chaos result,
and consequently by a convergence in law result.
The model of equation \eqref{eq:bossy} falls into the class of \emph{kinetic} McKean-Vlasov SDEs, which will be discussed in detail in Section \ref{sc:sing_deg_McK}.

  A second example is a byproduct of the so called {\it Markovian projection}.
Let $(\sigma_t)$ be a progressively measurable matrix-valued
process such that $a = \sigma \sigma^\top$ is non-degenerate and $(b_t)$ a vector-valued also progressively measurable process, 
and consider the It\^o process
$$ \xi_t = X_0 + \int_0^t \sigma_s dW_s + \int_0^t b_s ds. $$
In \cite{gyongy} the author shows the existence of Borel functions $\bar \sigma, \bar b$ defined on $[0,T]\times \R^d$
such that the Markov diffusion process
$$ X_t = X_0 + \int_0^t \bar \sigma(s,X_s) dW_s +  \int_0^t \bar b(s,X_s) ds, $$
has the same marginals as the process $\xi$ and it holds
$$\bar  a(s,x) = \E[a_s \vert X_s = x], \ \bar b(s,x) = \E [b_s \vert X_s = x],$$
with $\bar \sigma \bar \sigma^\top = \bar a$. Notice that the equation for $X$ is of McKean-Vlasov type, since its coefficients involve conditional expectations. 
Having now a McKean-Vlasov type equation, one could use propagation of chaos techniques and particle systems to approximate  it.

The idea of Markovian projection  has several applications in
mathematical finance, among which the celebrated Dupire formula, illustrated below. Let us recall the
{\em calibrated local stochastic volatility models},
see e.g.\ Chapter 11 of \cite{labordere} or \cite{djete} for more recent developments, which take the McKean  type form
\begin{equation} \label{eq:Labo1123}
df_t = f_t \sigma(t,f_t,\shl_{(f_t, a_t)}) a_t dW_t.
  \end{equation}
 Here the couple $(f_t,a_t) $ corresponds to  stock price and its volatility, $\shl_{(f_t, a_t)}$ is the law of the couple,
$(t, x,\pi) \mapsto   \sigma(t,x, \pi)$ is a law-dependent
functional defined in a way that
  $$ \sigma(t,f_t,\shl_{(f_t, a_t)}) = \frac{\sigma_{\text{Dup}}(t,f_t)}
  {\sqrt{\E(a_t^2 \vert f_t}) },$$
where the local volatility function $\sigma_{\text{Dup}}$ is the one
defined in \cite{dupire}. In the special case when $\pi$ admits a density then
$$ \sigma(t,x,\pi) = \frac{\sigma_{\text{Dup}}(t,x)}
{\sqrt{\frac{\int a^2 \pi(x,a) da}{\int \pi(x,a) da}}}.$$

\subsection{Time-reversal of diffusions as McKean SDE}

\label{sc:time-rev} 
  
Let $X = (X_t, t \in [0,T])$ be
 a  diffusion process in $\mathbb R^d$,  solution of the SDE
\begin{equation}
\label{eq:X} 
X_t=X_0+\int_0^t b(s,X_s)ds+\int^{t}_{0}\sigma(s,X_s)dW_s, \ t \in [0,T],
\end{equation}
where $\sigma$ and $b$ are Lipschitz coefficients  with linear growth
and $W$ is a standard Brownian motion on $\R^d$.
Let $\mu$ be the law of $X_T$.
In Section 4.\ of \cite{Izydorczyk}, the authors focus on the particular McKean SDE (with singular density
dependence)
\begin{equation}\label{MKIntro}
\begin{cases}
\displaystyle Y_t = Y_0 - \int^{t}_{0}b\left(T-r,Y_r\right)dr +
  \int^{t}_{0}\left\{\frac{\mathop{\textrm{div}_y}\left(a_{i.}\left(T-r,Y_r\right)v_{r} \left(Y_r\right)\right)}{v_{r}\left(Y_r\right)}\right\}_{i\in[\![1,d]\!]}dr + \int^{t}_{0} \sigma\left(T-r,Y_r\right)d\beta_r 
 \\
v_t  \text{ is the density  of } \shl_{Y_t}^\P, t \in (0,T)
\\
Y_0 \sim  \mu,
\end{cases}
\end{equation}
where $\beta$ is a $d$-dimensional Brownian motion and $a = \sigma \sigma^\top$.
The unknown of \eqref{MKIntro} is the couple $(Y,v)$. 
It is well-known that, under suitable conditions, the time-reversed process $\hat X$, given by $\hat X_t = X_{T-t}$ is a diffusion process, whose coefficients depend explicitely on the marginal laws of the process $\hat X$, see \cite{haussmann_pardoux}. Indeed a solution $(Y,v)$ of the McKean equation \eqref{MKIntro} is given considering
  $Y = \hat X$ and the law density $v(t,\cdot)$ of $Y_t$.
The delicate question concerns uniqueness for \eqref{MKIntro}, which would allow
to reconstruct the time-reversal of $X$ starting from its law at $T$.
Under suitable conditions, this was the object of
\cite{Izydorczyk}, where the authors have connected $Y$ with solution of a 
{\em linear} Fokker-Planck equation with {\em terminal} condition.
In particular, they first used the fact, that given any solution $(Y,v)$ of \eqref{MKIntro},
$w_t = v_{T-t}$ is a solution to
\begin{equation} \label{EDPTerm0}
\left \{
\begin{array}{l}
\partial_t w = \frac{1}{2} 
\displaystyle{\sum_{i,j=1}^d} \partial_{ij} \left( (\sigma \sigma^{\top})_{i,j}(t,x) w \right) - div \left( b(t,x) w \right)\\
w(T) = \mu.
\end{array}
\right .
\end{equation}
In other words  $(Y, v)$ is a probabilistic representation of a solution of \eqref{EDPTerm0}.
One essential tool was the proof of uniqueness for the PDE \eqref{EDPTerm0}, see Section 3. of \cite{Izydorczyk}.
Those problems were naturally motivated by applications to inverse problems in physical sciences,
in particular the determination of the initial position of a diffusion
phenomenon, see e.g.\ \cite{tikhonov1977solutions} and \cite{lattes1969method}.
We also mention that time-reversal of diffusions is used in artificial intelligence, in particular in artificial content generation, see e.g.\ \cite{song}.

Another potential application of the McKean SDE \eqref{MKIntro} is related to the  probabilistic representation
of a semilinear Kolmogorov type PDE, e.g.\ a Hamilton-Jacobi-Bellman (HJB) equation related to stochastic control.
Classically that HJB PDE is probabilistically represented by a forward-backward SDE in the sense
of Pardoux and Peng, see \cite{PardouxPeng92}, where the forward process $X$ is often a diffusion of the type
\eqref{eq:X}.
Following the idea of time-reversal,
the resolution of the McKean backward SDE \eqref{MKIntro} helps to ``discover'' the forward diffusion $X$
from the terminal condition; this, coupled with the backward SDE, provides a fully backward representation
of the semilinear Kolmogorov PDE, which, in the case of an HJB PDE, helps the process $Y$ to remain ``close'' to the  optimal controlled process,
   especially in view of Monte-Carlo numerical simulations. 
   In \cite{LucasFully}, the authors implement previous idea in
   the simplified framework 
when the forward diffusion is an Ornstein-Uhlenbeck process.

\subsection{Boltzmann type models and Enskog equation}
\label{sec:Boltzmann}

Fokker-Planck type PDEs or integro PDE appear also as generalisations
of Boltzmann type equations. The first paper introducing
a sort of McKean-Vlasov type equation is at our knowledge
\cite{tanaka}, which provides a probabilistic representation
of a so called Boltzmann equation of Maxwellian molecules,
under the inspiration of previous contributions by Kac and McKean.
Those processes are in the class of McKean type pure jump processes.
Later, a more elaborated contribution appears in \cite{montecatini}
and many others, see e.g.\ \cite{ruediger}, which works out
a probabilistic representation of the Boltzmann-Enskog equation
which is a non-homogeneous version of the Boltzmann equation.
A significant contribution to solving a McKean-Vlasov type equation
associated to a generalised Boltzmann equation  is
\cite{BallyAlfonsi}, which makes use of the deterministic sewing lemma,
see e.g.\ \cite{FeyelDeLaPradelle1, gubinelli}.

\section{Singular McKean-Vlasov SDEs: distributional drift}
\label{sc:sing_McK}

This section  is devoted to an important class of singular McKean  SDEs, where some of the coefficients are generalised functions. In particular, we consider
\begin{equation}\label{eq:MK_dist_SDE_intro}
\left\{
\begin{array}{l}
X_t=X_0+\int_0^t F(v(s,X_s))b(s,X_s)ds+W_t\\
v(t,\cdot) \text{ is the law density of } X_t\\
X_0\sim \nu,
\end{array}
\right.
\end{equation}
where $W$ is a $d$-dimensional Brownian motion, $F$ is a non-linear function, $b$ is a suitable distribution, and $\nu$ is a given initial law. 
Since the drift coefficient is a product of a function $F\circ v$ with a distribution $b$, the SDE in \eqref{eq:MK_dist_SDE} is only formal at this level. Indeed, the drift coefficient $F(v)b$ is a well-defined element in a space of distributions provided that the function $F\circ v$ is regular enough.

Our aim is to present a way to solve these kind of singular equations.
These are critically different from singular equations reviewed in Section \ref{sec:4},
because the presence of distributional coefficients
makes their definition unclear a priori.
We will present in a unified framework the works \cite{issoglio_russoMK, issoglio_russoPDEa, issoglio_russoMPb, issoglio_et.al24} and provide specific
results proven in their setting in order to illustrate in a detailed way the general roadmap.
Other SDEs involving distributional coefficients, including the convolutional-type drift  $b\ast v$,  were investigated by many other authors,
see e.g.\ \cite{chaudru_jabir_menozzi22, chaudru_jabir_menozzi23, hao2021singular, veretennikov2023weak,  zhang2024cauchy, GaleatiGerencser25, mayorcas, zhang2021second} and we will highlight links and similarities  of each single paper when relevant.

The path to well-posedness of equation \eqref{eq:MK_dist_SDE_intro}  culminates in Section \ref{sec:McKean-point} and it  is quite involved, hence resulting  in several subsections containing analytical and probabilistic tools, definitions and methodologies. 
To help the reader navigate in this section, we illustrate the relationship between the subsections in Figure \ref{fig1}.
To characterise a notion of solution to equation \eqref{eq:MK_dist_SDE_intro} one needs to define a solution of a non-McKean SDE with distributional drift first.
In Section \ref{sec:sol-singular-SDE} we propose a definition based on the notion
of rough martingale problem, with related well-posedness in Section \ref{sec:linearSDE},
while in Section \ref{sec:solution_sing_SDE} we focus on
a concept which is closer to the classical notion of stochastic differential
equation. 
The rough martingale problem definition relies on  solutions to the Kolmogorov PDE (Section \ref{ssc:singularPDE}).
To make the link between the rough martingale problem and the McKean rough martingale problem one uses the  solution of the Fokker-Planck PDE (Section \ref{sec:singularFP}). We remark that the McKean rough martingale problem is somehow equivalent
to the non-linear Fokker-Planck PDE.
Finally, this methodology can be applied to a degenerate singular McKean SDE (Section \ref{sc:sing_deg_McK}).

\begin{figure}[h]
\begin{center}
\begin{tikzpicture}[every node/.style={draw, align=center}]
\node (s1) at (4,0.5) {On the dynamics \\ of the singular SDEs \\ Section~\ref{sec:solution_sing_SDE}};
\node (s2) at (0,2) {McKean rough  martingale problem \\ Section~\ref{sec:McKean-point}};
\node (s3) at (-4,-2) {non-linear  Fokker-Planck PDE \\ Section~\ref{sec:singularFP}};
\node (s4) at (4,-2) {rough martingale problem \\ Section~\ref{sec:sol-singular-SDE} and \ref{sec:linearSDE}};
\node (s5) at (4,-4) {Kolmogorov PDE \\ Section~\ref{ssc:singularPDE}};

\path
    (s1) edge [dashed] (s4)
    (s3) edge[<->, double distance=1pt, >=latex'] (s2)
    (s2) edge[<-, double distance=1pt, >=latex'] (s4)
    (s4) edge[<-, double distance=1pt, >=latex'] (s5);
\end{tikzpicture}
\end{center}
\caption{Different connections about subsections.}
\label{fig1}
\end{figure}
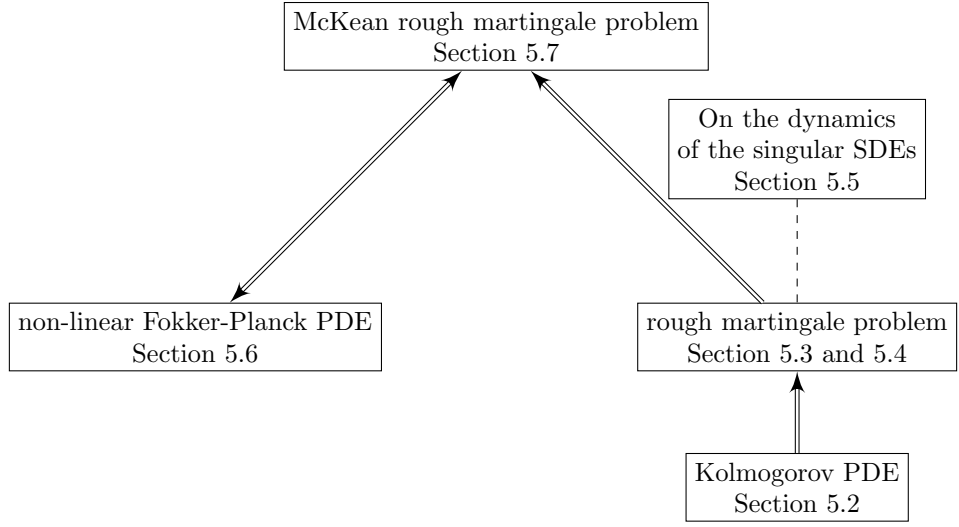

\subsection{Topological preliminaries}\label{sec:top-prelim}

From now on  $\mathcal S=\mathcal S(\mathbb R^d )$  will be the space of real-valued Schwartz functions on $\mathbb R^d$
and $\mathcal S'=\mathcal S'(\mathbb R^d )$ the space of Schwartz distributions. The corresponding dual pairing will be denoted by
$\langle \cdot, \cdot \rangle$ and $\shf: \shs \to \shs$ will be the Fourier transform on $\shs$, extended in the standard way to $\shs'$,
see e.g. Section 1.2.3 of \cite{sawano}. 

\subsubsection*{Besov spaces}\label{sc:Besov_spaces}
We adopt the notations of \cite{sawano} and we start with a very useful one.
 \begin{notation}\label{def:fourierm}
   Let us fix a function $\tau\in\shs(\R^d)$.  For all $f\in\shs'(\R^d)$ and for all $x\in\R^d$
   we introduce the so called Fourier Multiplier of $f$, as
 \begin{equation}
   \tau(D)f(x):=\shf^{-1}(\tau\cdot\shf f)(x)=(2\pi)^{-\frac{d}{2}}\langle f,\shf^{-1}\tau(x-\cdot)\rangle.
  \end{equation}
\end{notation}
In fact the operator $f \mapsto \tau(D)f$ maps $\shs'(\R^d)$ into  $\shs'(\R^d) \cap L_{loc}^1(\R^d)$.
 Using Notation \ref{def:fourierm} we introduce the Besov Spaces, see e.g.\ Section 2.1 of \cite{sawano}.
 \begin{definition}\label{def:Besov}
   Let $\g\in\R$ and $p,q\in[1,\infty]$.
   \begin{enumerate}
     \item
      Let us consider the dyadic partition of the unity, constituted by a sequence of functions $(\varphi_j)_{j\geq 0}$ and a $\psi$, both defined
   on $\R^d$ with compact support,
introduced in  Definition 2.1 of \cite{sawano}.
The Besov Space $B_{p,q}^{\g}(\R^d)$ is the set of  $f \in \shs'(\R^d)$ such that the norm 
   \begin{equation}
    \|f\|_{B_{p,q}^{\g}(\R^d)}:=\|\psi(D)f\|_{L^p}+\bigg(\sum_{j=-1}^{\infty}2^{q\g j}\|\varphi_j(D)f\|^q_{L^p}\bigg)^{\frac{1}{q}}
  \end{equation}
   is finite. 
\item 
  For $\gamma\in\mathbb R$ and $p=q=\infty$ we denote by  $\mathcal C^\gamma = \mathcal C^\gamma(\mathbb R^d):=B_{\infty,\infty}^{\g}(\R^d)$ the Besov space,
  endowed with its norm $\|\cdot\|_{B_{\infty,\infty}^{\g}}=:\|\cdot\|_{\gamma}$.
  For more details see Section 2.7, page 99 of
 \cite{bahouri2011fourier} and also \cite{issoglio_russoPDEa}, where useful facts and definitions about these spaces are recalled.
 \end{enumerate}
\end{definition}

\begin{remark} \label{rmk:classHold}
The definition of the norm, and therefore of the space, is independent of the partition of unity. 
  If $\gamma \in \R^+ \setminus \mathbb N$ then the space $\mathcal C^\gamma $ coincides with the H\"older-Zygmund space, see e.g. Section 2.2.2
  of \cite{sawano}. In the case $0 < \gamma < 1$ it coincides with the classical H\"older space. 
  We have $\mathcal C^\gamma \subset \mathcal C^\alpha$ for any $\gamma >\alpha$. Moreover it holds that $L^\infty \subset \mathcal C^0$, see \cite{issoglio_et.al24} for a proof in the case of anisotropic Besov spaces. 
\end{remark}
We denote by $C_T \mathcal C^\gamma$ the Banach space of continuous functions on $[0,T]$ taking values in $\mathcal C^\gamma$, that is $C_T \mathcal C^\gamma:= C([0,T]; \mathcal C^\gamma)$, equipped with the sup norm.
We will also make use of an equivalent norm given by 
 \begin{equation}\label{eq:rho_norm}
 \|w \|^{(\rho)}_{C_T\mathcal C^\alpha} := \sup_{t\in[0,T]}  e^{-\rho t} \|w(t) \|_{\alpha}.
\end{equation}
Consider then the $\rho$-ball in  $C_T \mathcal C^\alpha$ of radius $M$, given by 
\begin{equation}\label{eq:rho_ball}
E^\alpha_{\rho, M} := \{ v\in  C_T \mathcal C^\alpha : \|v \|^{(\rho)}_{C_T\mathcal C^\alpha} \leq M \}.
\end{equation}
The $\rho$-equivalent norm generates  the $\rho$-equivalent metric with respect to the metric of $C_T\mathcal C^\alpha$, given by 
\begin{equation}\label{eq:drho}
d_{\rho} (w,z):=   \|w - z\|^{(\rho)}_{C_T\mathcal C^\alpha} , \qquad \forall \rho\geq0,
\end{equation}
for any $w,z \in  C_T\mathcal C^\alpha$,
see (3.7) and (3.8) of \cite{issoglio_russoMK}.
The sets  $ E^\alpha_{\rho, M}$ are closed with respect to the topology of $C_T\mathcal C^\alpha$, hence they are
F-spaces\footnote{for the topological notion of F-space, see  Chapter 2.1 of \cite{dunford-schwartz}},
with respect to the metric topology of $C_T\mathcal C^\alpha$.
Let $\rho_0>0$ and $ M_0>0$  be chosen arbitrarily. The F-space $E^\alpha_{\rho_0,M_0}$ 
can be equivalently equipped with the metric
$d_{\rho} $.

For any given $\gamma\in \R$ we denote by $\mathcal C^{\gamma+}$ and $\mathcal C^{\gamma-}$  the spaces given by
\[
\mathcal C^{\gamma+}:= \cup_{\alpha >\gamma} \mathcal C^{\alpha} ,  \qquad  
\mathcal C^{\gamma-}:= \cap_{\alpha <\gamma} \mathcal C^{\alpha}.
\]
Notice that $\mathcal C^{\gamma+}$ is an inductive space. 
We will also use the spaces $C_T C^{\gamma+}:=C([0,T]; \mathcal C^{\gamma+})$, which is equivalent to the fact that for
$f\in C_T C^{\gamma+} $ there exists $\alpha>\gamma $ such that $f\in C_T \shc^{\alpha}$,  see for example Appendix B in \cite{issoglio_russoMK}. 
Similarly, we use the space   $C_T \shc^{\gamma-}:=C([0,T]; \mathcal C^{\gamma-})$, meaning that if $f\in C_T \mathcal C^{\gamma-} $ then for any $\alpha<\gamma $ we have $f\in C_T \mathcal C^{\alpha}$.  
We denote by $\mathcal C_c^{\gamma}=\mathcal C_c^{\gamma}(\R^d)$ the space of elements in $\mathcal C^{\gamma}$ with compact support, and by $\mathcal C_c^{\gamma+}:=  \cup_{\alpha >\gamma} \mathcal C_c^{\alpha}$.
Note that if $f$ is continuous and such that $\nabla f \in C_T \mathcal C^{0+}$ then $f\in C^{0,1}$.

\subsubsection*{Schauder Estimates and Pointwise product}\label{sc:Schauderpoint}

In this short subsection we are interested in the action of the heat semigroup $P_t$ on elements of Besov spaces $\C^\gamma$. These estimates are known as \emph{Schauder's estimates}: for a proof we refer to Lemma 2.5 in \cite{catellier_chouk}, see also \cite{gubinelli_imkeller_perkowski}
for similar results.
Here  $c$ will denote a constant which can vary from line to line.
\begin{lemma}[Schauder's estimates]\label{lm:schauder}
  Let $\theta\geq 0$ and $\gamma \in \R$.
  \begin{enumerate}
  \item There exists a constant $c = c(\theta)$ such that
\begin{equation}\label{eq:Pt}
\|P_t f\|_{\gamma + 2 \theta} \leq c t^{-\theta} \|f\|_\gamma,
\end{equation}  
for all $t>0$.
\item Suppose furthermore that  $\theta\in(0,1)$. There exists a constant  $c$ such that
for every $f\in \mathcal C^{\gamma + 2\theta }$  we have
\begin{equation}\label{eq:Pt-I}
\|P_t f-f\|_{\gamma} \leq c t^{\theta} \|f\|_{\gamma+2\theta }.
\end{equation}
\end{enumerate}
\end{lemma}
Notice that from \eqref{eq:Pt-I} it readily follows that if $f\in \mathcal C^{\gamma + 2 \theta}$ for some $0<\theta<1$, then for $t>s>0$ we have 
\begin{equation}\label{eq:PcontC}
\|P_tf-P_s f\|_{\gamma}\leq c (t-s)^{\theta} \|f\|_{\gamma+ 2\theta }.
\end{equation}
In particular, if $f\in  \mathcal C^{\gamma+2\theta}$ then $P_\cdot f\in C_T \mathcal C^\gamma$ and it is $\theta$-H\"older continuous in time. 
We also recall that Bernstein's inequalities hold (see Lemma 2.1 in \cite{bahouri2011fourier} and Appendix A.1 in \cite{gubinelli_imkeller_perkowski}), that is for $\gamma \in \mathbb R$ there exists a constant $c>0$ such that
\begin{equation}\label{eq:nabla}
\|\nabla g\|_{\gamma} \leq c  \|g\|_{\gamma+1},
\end{equation}  
for all $g\in \mathcal C^{1+\gamma}$. 
Using Schauder's and Bernstein's inequalities we can easily obtain a useful estimate on the gradient of the semigroup, as we see below.

\begin{lemma}\label{lm:nablaP}
Let $\gamma\in \mathbb R$ and $\theta \in (0,1)$. If $g\in \mathcal C^\gamma$ then for all $t>0$ we have $\nabla (P_tg) \in \mathcal C^{\gamma +2\theta -1}$ and
\begin{equation}\label{eq:nablaP}
\|\nabla (P_tg)\|_{\gamma+2\theta-1} \leq c t^{-\theta}  \|g\|_{\gamma}.
\end{equation}  
\end{lemma}

The following is an important estimate which allows to define the so called
{\it pointwise product} between certain distributions and functions, which is based on Bony's estimates. For details see \cite{bony} or Section 2.1 of \cite{gubinelli_imkeller_perkowski}. Let   $f \in \mathcal C^\alpha$ and $g\in\mathcal C^{-\beta}$ with $\alpha-\beta>0$ and $\alpha,\beta>0$. Then the
{pointwise product} $ f \, g$ is well-defined as an element of $\mathcal C^{-\beta}$ and  there exists a constant $c>0$ such that 
\begin{equation}\label{eq:bony}
\| f \, g\|_{-\beta} \leq c \| f \|_\alpha \|g\|_{-\beta}.
\end{equation} 
Moreover if $f$ and $g$ are continuous functions defined on $[0,T]$ with values in the above Besov spaces,
one can easily show that the pointwise 
product is also continuous with values in  $\mathcal C^{-\beta}$, and 
\begin{equation}\label{eq:bonyt}
\| f \, g\|_{C_T\mathcal C^{-\beta}} \leq c \| f \|_{C_T \mathcal C^\alpha} \|g\|_{C_T \mathcal C^{-\beta}} .
\end{equation}

\subsection{The singular Kolmogorov PDE}
\label{ssc:singularPDE}

The singular Kolmogorov PDE is a basic tool for the formulation
of the domain of the rough martingale problem, see Section
\ref{sec:sol-singular-SDE}.
 For the sequel of this section we consider Let $0<\beta<\frac12$.
We start by an assumption on the singular drift $b$ which is standing for this subsection.
\begin{assumption}\label{ass:b}
 $b \in C_T \shc^{(-\beta)+}(\R^d)$. In particular $b \in C_T\shc^{-\beta}(\R^d)$.
\end{assumption}

 The Kolmogorov equation we consider is 
\begin{equation}\label{general_Kolmogorov_PDE}
\left\{
\begin{array}{l}
  \partial_t u(t,x)+\frac{1}{2}\Delta u(t,x) +\nabla u(t,x,)\cdot b(t,x)=g(t,x)\\
  u(T,\cdot)=u_T(\cdot),
  \end{array}
  \right.
\end{equation}
 where $g:[0,T]\to\shs '$ is continuous and the terminal condition $u_T\in\shs'$.
 
\begin{definition}\label{def:L}
The singular differential operator $L$ is defined as
\begin{equation*}
\begin{array}{cll}
L:&\shd_0&\to\{\shs'\text{-valued continuous functions defined on } [0,T] \}\\ 
      & f &\mapsto L f:= \partial_tf+\frac12\Delta f+\nabla f\cdot b,
\end{array}
\end{equation*}
where $\shd_0:=C_T \C^{1+\be}\cap C^1([0,T];\shs')$.
\end{definition}

\begin{definition}[Weak solution -- Definition 4.1 \cite{issoglio_russoPDEa}]\label{def:weak_sol_kol}
Given the above Kolmogorov PDE \eqref{general_Kolmogorov_PDE}, we say that $u\in C_T\C^{1+\be}$ is a weak solution to \eqref{general_Kolmogorov_PDE} if for all $\varphi\in\shs(\R^d)$
\begin{equation}
  \langle u_T, \varphi\rangle-\langle u(t),\varphi \rangle+ \int_t^T \frac12 \langle u(s) , \Delta \varphi \rangle ds + \int_t^T  \langle \nabla u(s) b(s) , \varphi \rangle ds =\int_t^T \langle g(s) ,  \varphi \rangle ds. 
\end{equation}
\end{definition}
We point out that assuming only $u\in C_T\C^{1+\be}$
in previous definition guarantees that $u$
is differentiable (with continuous derivative in $\shs'$) with respect to $t$,
therefore $u \in \shd_0$.

\begin{definition}[Mild solution -- Definition 4.2 \cite{issoglio_russoPDEa}]\label{def:mild_sol_kol}
A {\it mild solution} to \eqref{general_Kolmogorov_PDE} is a $u\in\C_T\C^{1+\be}$ such that 
\begin{equation}\label{eq:mild_sol_kol}
u(t)= P_{T-t}u_T+\int_t^T P_{s-t}(\nabla u(s)b(s))ds-\int_t^T P_{s-t}g(s)ds.
\end{equation}
\end{definition}

We remark that previous definition is meaningful taking into account Lemmata \ref{lm:nablaP} and \ref{lm:schauder}.
It can be shown that weak  and mild solutions for \eqref{general_Kolmogorov_PDE} are equivalent, see Proposition 4.5 in \cite{issoglio_russoPDEa}.
A well-posedness result also holds, we recall it below. 
\begin{theorem}[Existence and Uniqueness -- Theorem 4.7 in \cite{issoglio_russoPDEa}]\label{thm:u}
 Let $b$ fulfilling Assumption \ref{ass:b} and assume $u_T\in\C^{(2-\be)-}$ and $g\in C_T\C^{-\be} $(in particular $u_T$ is bounded).
Then there exists a mild solution $u\in C_T\C^{(2-\be)-}$ of PDE \eqref{general_Kolmogorov_PDE} which is bounded, and it is unique in $C_T \C^{1+\be}$.
\end{theorem}
This theorem has been proven in Theorem 4.7 in \cite{issoglio_russoPDEa}. In fact in the original reference the framework is more general, because the spaces used include unbounded functions. The idea of the proof is to show that the right-hand side of the mild formulation \eqref{eq:mild_sol_kol} is a mapping from  $ C_T\C^{1+\alpha}$ into itself,  for all $\alpha\in[\beta, 1-\beta)$. Subsequently, using Schauder estimates   Lemma \ref{lm:schauder}, Bernstein's inequality \eqref{eq:nabla} and the equivalent norm \eqref{eq:rho_norm} one concludes that in fact it is also a contraction on the same space.

\subsection{Rough martingale problem}\label{sec:sol-singular-SDE}

Let us start with the case  without law dependence. Indeed  the definition of solutions to McKean singular SDEs, given in  Section \ref{sec:McKean-point}, will heavily rely on  the definition of solution
introduced here, which is formulated as rough martingale problem. In the literature one can find other notions of solutions to singular SDEs, and for completeness we will review them in Section \ref{sec:solution_sing_SDE} below. 

We consider singular $d$-dimensional SDEs formally of the type
\begin{equation}\label{eq:SingSDE}
X_t =X_0+ \int_0^t  b(s, X_s) ds +  W_t,
\end{equation}
where  the  drift $b$ is a generalised function in the $x$-variable and the initial condition $X_0$ is distributed according to a given probability measure $\nu$. 
Here we restrict to the case where $b\in C_T \C^{(-\beta)+}$ with $\beta\in(0,\frac12)$ and $W$ is a $d$-dimensional Brownian motion, but it is possible to consider different noises such as
$\alpha$-stable noise and different distributional spaces for measuring
the (ir)regularity of the drift, such as general Besov spaces $B^{-\beta}_{p,q}$ (see e.g.\ \cite{chaudru_menozzi}) 
or fractional Sobolev spaces $H^{-\beta}_p$, see e.g.~\cite{flandoli_et.al14}. It is moreover possible to go beyond the Young regime (in this case $\beta>\frac12$) by making use of paracontrolled calculus or regularity structures, but this is not exploited further in the present paper. For some examples beyond the Young regime with Brownian noise see \cite{diel, cannizzaro, ZhangZhao, hao2021singular}.

The idea  is to frame the SDE  \eqref{eq:SingSDE} as a martingale problem in order to avoid the evaluation of the drift $b$ in $X_t$. Suppose for a moment that $b$ is a suitable function. If we denote by $$L = \partial_t f +\frac12\Delta f + \nabla f \cdot b$$ the Dynkin operator of the process $X$, the
classical martingale problem formulation requires that the quantity 
\begin{equation}\label{eq:MP}
f(t, X_t) - f(0,X_0) - \int_0^t (L f)(s, X_s) ds
\end{equation}
is a local martingale under some probability $\mathbb P$ for all test functions
$f\in C^{1,2}$, or equivalently $f \in C^{1,2}$ bounded or compactly supported.
 That formulation is also equivalent to the one in Definition \ref{DefMP}.
 It is well-known that this formulation is equivalent to  solutions in law. 
When $X$ is the canonical process, denoted by $Z$, 
this is the celebrated {\it Stroock-Varadhan} martingale problem formulation,
see \cite{stroock_varadhan}. In this case 
the solution to the martingale problem is the probability measure $\P$.
From now on we use the formulation with the canonical process $Z$.

Now we consider the case when $b$ is singular, for example when $b(t, \cdot) \in \C^{(-\beta)+}$ for $\beta\in(0,\frac12)$. Then the Dynkin operator
$L$ has a  term  of the form  $\nabla f \cdot b$ which  is potentially problematic  since $b$ is a distribution. 
It is however possible to properly define it if $\nabla f$ is smooth enough (see Section \ref{sc:Schauderpoint}) and the result is a distribution with the same (ir)regularity of $b$, namely in general $(L f)(t, \cdot) \in \C^{-\beta}$ even if
$f\in C^{1,2}$ with compact support, so we have simply shifted the problem from $b$ to $L f$. Indeed,  the integrand $(L f)(s, Z_s) $ in \eqref{eq:MP} is again ill-posed because it cannot be evaluated at $Z_t$. 
However, one can choose a different domain  $\mathcal D$ of test functions $f$ to make sure that $L f$ is a
function with enough regularity.
This can be done if one is able to give a meaning, to solve and study the singular Kolmogorov PDE $L f = g$ where $g$ belongs to  a suitable class of functions,
see Section \ref{ssc:singularPDE}:
 this is the idea behind the definition of  \emph{rough martingale problem} for SDEs with distributional drift.
 Notice the
 difference between the Stroock-Varadhan formulation and  the rough martingale problem: the domain of the first one is piloted by
 functions $f$ that appear in the quantity \eqref{eq:MP} while the domain of the latter is piloted by functions $g$ that are `de facto' the values of $L f$ in  \eqref{eq:MP}.

The precise definition  of solution to the rough martingale problem with distributional drift $b$ and initial condition $\nu$ is given below, see also \cite{issoglio_et.al24}.
\begin{definition}[Rough martingale problem]
\label{def:MK_MP}
  Let $\Omega_T$ be the canonical space  $C([0,T];\R^d)$ equipped with
  its Borel $\sigma$-field $\shf_T$, and let $Z$ denote the canonical process on it.
  A solution (in law) of the martingale problem
  with distributional drift $b$ and initial law $\nu\in \mathcal P(\R^d)$ 
  is  a Borel  probability measure $\P$ on  $(\Omega_T, \shf_T)$,
  such that 
$Z_0 \sim \nu $ under $\P$ and
 the quantity
\[
f(t, Z_t) - f(0,Z_0) - \int_0^t (L f)(s, Z_s) ds
\]
 is a local martingale under $\P$ for every $f\in \shd$, where
\begin{equation}\label{eq:DBis}
  \shd:=\{ f\in C_T\C^{1+\be}:\exists g\in C_T \shs \text{ such that } Lf=g\text{ and }f_T\in\shs\}.
\end{equation}
We say that the martingale problem problem with distributional drift $b$ and initial law $\nu\in \mathcal P(\R^d)$ 
admits {\it uniqueness} if, whenever we have two solutions $\P$
and  $\hat \P$,  then $\P = \hat \P$.

We denote by rMP$(b, \nu)$ the rough martingale problem with distributional drift $b$ and initial condition~$\nu$.
By a slight abuse of notation, if $\nu$ admits a density $v_0$ then we write rMP$(b, v_0)$.
\end{definition}

 Supposing that $b$ is a bounded function, 
   then one could consider alternatively the corresponding
   Stroock-Varadhan martingale problem related to $(I_d,b)$
   according to Definition \ref{DefMP}.
 Lemma 4.1 in \cite{issoglio_russoMPb} shows the equivalence between the two martingale
   problem formulations, when $b\in C_T \C^{0+}$.

\subsection{Solving the rough martingale problem}\label{sec:linearSDE}

Here we continue following \cite{issoglio_russoMPb}. The notion of solution to the rough martingale problem was explained in Section \ref{sec:sol-singular-SDE}.
The first thing to notice is that the domain $\mathcal D$ given  in \eqref{eq:DBis} is a subset of $\mathcal D_0$ introduced in Definition \ref{def:L}, and by Theorem \ref{thm:u} the PDE $L f =g$ is well-posed.

The main result of this section is the following  theorem of well-posedness for the rMP$(b, \nu)$.
\begin{theorem}[Existence and Uniqueness -- Theorem 4.5 in \cite{issoglio_russoMPb}] \label{thm:MP}
Let us suppose $b$ to fulfill Assumption \ref{ass:b} and let $\nu$ be a probability measure on $\R^d$. Then there exists a unique solution $\P$ to the rough martingale problem rMP$(b,\nu)$.
\end{theorem}
  The proof of the above theorem is based on the following equivalence result.
\begin{theorem}[Equivalence -- Theorem 3.9 in \cite{issoglio_russoMPb}] \label{thm:XY}
Let us suppose $b$ to fulifill Assumption \ref{ass:b}.
\begin{itemize}
\item[(i)] If $\mathbb P $ is a solution to  rMP$(b, \nu)$ then $(Y, \mathbb P) $ is a solution in law  to \eqref{eq:Y},   where $Y_t := \phi(t, Z_t)$ and  $Y_0\sim \eta$, where  $\eta$ is the pushforward measure of $\nu$ given by
$\eta:= \nu (\phi^{-1}(0, \cdot))$. 
\item[(ii)] If $(Y, \mathbb Q) $ is a solution in law  to \eqref{eq:Y} with $Y_0\sim \eta$ then $\P:= \mathcal L_{X}^\mathbb Q$  is a solution to rMP$(b, \nu)$, where $X_t := \phi^{-1}(t, Y_t)$ and $\nu$ is the pushforward measure of $\eta$ given by $\nu:= \eta (\phi(0, \cdot))$.
\end{itemize}
\end{theorem}

\begin{proof}[Idea of the proof of Theorem \ref{thm:MP}]
By Theorem \ref{thm:XY} the proof reduces to the study of well-posedness of 
SDE \eqref{eq:Y}, with initial condition distributed according
to a given probability law $\eta$.
For this equation we can apply classical Stroock-Varadhan arguments since
the  drift  and diffusion coefficients  are continuous, bounded, and the
diffusion coefficient is  non-degenerate. Since those results are expressed only for deterministic
initial conditions, for the comfort of the reader we briefly explain how to  adapt them to
the context of random initial conditions.

As anticipated, Theorem 10.2.2 in \cite{stroock_varadhan} ensures existence (and measurability
with respect to $y$) of a  solution in law
$(Z, \mathbb Q^y)$ to \eqref{eq:Y} when the initial condition is deterministic, namely when $Z_0=y$ for any $y\in \R^d$.
Now we use a {\em superposition argument} to glue together the laws $\mathbb Q^y$ according to the initial condition $\eta$ and construct the probability $\mathbb Q(\cdot) := \int_{\R^d} \mathbb Q^y(\cdot) \eta(dy)$ so that the
$(Z,\Q)$ is a solution of \eqref{eq:Y}.

To prove {\em uniqueness} of \eqref{eq:Y}, without restriction of generality,
let $\P$ be a probability on the canonical space such that $(Z,\P)$ is a solution to \eqref{eq:Y}.
At this point one  uses a {\em disintegration argument}, see e.g.\ Chapter III, nos.\ 70–72 of \cite{dellacherie}, which guarantees the existence a random kernel $\P^y$
such that $\P(\cdot)= \int \P^y(\cdot) \eta(dy)$.
Then one uses 
 Theorem 10.2.2 in \cite{stroock_varadhan}, which ensures uniqueness
of the solution of \eqref{eq:Y} with $Z_0 = y$. 
This completely identifies $\P$ so that uniqueness is established.
\end{proof}

\subsubsection*{Alternative  techniques for existence and uniqueness}

{\em Existence} for the rough martingale problem rMP$(b,\nu)$  can be established also
by making use of tightness arguments as it was done for instance in
 Section 5.1 of \cite{issoglio_et.al24},
in a degenerate framework.

\begin{theorem}[Tightness -- Lemma 3.12 in \cite{issoglio_russoMPb}]\label{thm:tightnessY}
 Let $(b^n)$ be a sequence in $C_T\C^{(-\be)+}$ converging to $b$ in $C_T\C^{-\be}$.
Let $Y^n$ be the solution of SDE \eqref{eq:Y}, with fixed initial condition $\eta$, where the drift $b$ is replaced by $b^n$. 
Then the sequence of laws $(\mathbb Q^n)$ of $(Y^n)$ is tight.
\end{theorem}
\begin{ideaproof}
  For a general sequence of processes $(Y^n)$,  tightness  can be shown  according to Theorem 4.10 in Chapter 2 of \cite{karatzasShreve}, hence one needs to prove that
 \begin{equation}\label{eq:KS1}
\lim_{\gamma \to\infty} \sup_{n\geq 1} \mathbb  Q^{n}(| Y^{n}_0 |>\gamma) =0
\end{equation}
and for every $\varepsilon > 0$,
\begin{equation}\label{eq:KS2} 
\lim_{\delta\to0} \sup_{n\geq 1} \mathbb  Q^{n} \Big ( \sup_{\substack{
s,t \in [0,T] \\|s-t|\leq \delta }} |Y^{n}_t-Y^{n}_s|>\varepsilon   \Big ) =0 .
\end{equation}
The requirement \eqref{eq:KS1} follows from the finiteness of the initial condition $Y_0^n \sim \eta $. The second requirement \eqref{eq:KS2} can be obtained  if one has an estimate on the moments 
\begin{equation}\label{eq:kolm}
  \mathbb E^{\mathbb Q^n}[|Y_t^{n} - Y_s^{n}|^p] \leq C |t-s|^{1+\sigma}, \qquad \forall { t,s\in [0,T]}, \quad \text{for some } p>2, \sigma>0.
\end{equation}
Indeed if \eqref{eq:kolm} holds,  Garsia–Rodemich–Rumsey Lemma (see e.g.  Section 3 in \cite{grr})  ensures that for every $m\in(0,\sigma)$ there exists a constant $C'$ and a sequence of random variables $(\Gamma_n)$ uniformly in $L^1(\Omega)$  such that
\begin{equation}\label{eq:grr}
 |Y_t^{n} - Y_s^{n}|^p  \leq C' |t-s|^m \Gamma_n.
\end{equation}
Now the requirement \eqref{eq:KS2} follows by an application of \eqref{eq:grr} and Chebyshev's inequality, noticing finally that $\sup_n\E(\Gamma_n)<\infty$.

The key estimate \eqref{eq:kolm} on the $p$-th moment is therefore fundamental to show tightness. In the specific setting we are in, it is possible to obtain this estimate for $p=4$ and $\sigma=1$ thanks to the good properties of the coefficients of the transformed SDE \eqref{eq:Y}.  
\end{ideaproof}

Using the tightness proved in Theorem \ref{thm:tightnessY} 
together with the equivalence result stated in Theorem \ref{thm:XY} one  gets the continuity theorem stated below. The full proof can be found in Theorem 4.3 in \cite{issoglio_russoMPb}, see also Proposition 29 in \cite{flandoli_et.al14} for more details.

\begin{theorem}[Continuity -- Theorem 4.3 in \cite{issoglio_russoMPb}]\label{thm:contMP}
  Let $(b^n)$ be a sequence in $C_T\C^{(-\be)+}$ converging to $b$ in $C_T\C^{-\be}$.
  Let $\P$   (respectively $\P^n$) be a solution to the rough martingale problem with distributional drift $b$ (respectively $b^n$) and initial condition $\nu$. Then the sequence  $(\P^n)$ converges weakly to  $\P$.
\end{theorem}
The theorem above will be used for the proof of Theorem \ref{thm:EUMcKean}
of existence  of the McKean rough martingale problem.

As explained in 
 Section \ref{sc:ethier-kurtz}, 
{\em uniqueness}
can be established alternatively via Markov marginal uniqueness
techniques, which rely on uniqueness of the marginal laws, which corresponds to Property M in Definition \ref{def:Prop_P}.
In Section \ref{sc:ethier-kurtz}, we have also illustrated how to check Property M in the case when the martingale problem is formulated
in Definition \ref{def:MK_MP},
which is indeed the case for rMP$(b,\nu)$ as illustrated in \eqref{eq:DBis}.

This method for proving uniqueness is used for example in Theorem 5.11 in \cite{issoglio_et.al24} in a degenerate setting,
in  \cite{zhang2021second} in the  purely kinetic model with
convoluted irregular
drifts, and  in Theorem 6.3 in \cite{hao2021singular} 
outside the Young regime.
Section \ref{sc:sing_deg_McK} will focus  on these degenerate and kinetic settings.
Markov marginal uniqueness techniques are also used by \cite{chaudru_menozzi} in the case of additive $\al$-stable noises 
and distributional drifts in Besov space $L^{r}([0,T];B_{p,q}^{-\be})$, where however the drift is of convolutional type like the one presented in Section \ref{sec:LpLq_p}.

\subsection{On the dynamics  of the singular SDEs}\label{sec:solution_sing_SDE}

In this section we still focus on SDE  \eqref{eq:SingSDE}, which is not law-dependent, but where the  drift $b(t, \cdot)$ is a distribution, hence it cannot be evaluated in $X_t$. 
 We have already  seen  in Section \ref{sec:sol-singular-SDE} one way of treating such equations through the notion of rough martingale problem. There are however other possibilities and we review them below.

Stochastic processes which can be somehow expressed formally as solutions to SDEs with distributional drift,
are for instance processes with reflection such as a {\it one-dimensional Bessel process} (resp.\ a {\it skew Brownian motion}), see e.g.\ Chapter XI (resp.\ Chapter VII) of \cite{ry}. In this case the drift can be seen as an atomic measure such a $\delta$-measure.
Those processes are interesting semimartingales but their study is out of scope of the present survey paper,
which mainly focuses on cases when each component of the (distributional)  drift is a derivative of
a vector of continuous functions.
An early pioneering work on generalised diffusions,
even though the solutions were still semimartingales,
appears in the monograph \cite{portenko}.
The next works to tackle this problem are from the early 2000s, see \cite{frw1, frw2} for time-homogeneous distributional drifts (in dimension 1), or \cite{basschen} for drifts which are measures of Kato class.
In \cite{frw1} two main notions of solution have emerged: {\em rough martingale problem} formulation and
{\it local time $B$-solution}, whereas the notion of  
{\em limit solution} has appeared in \cite{basschen}.

Before explaining those into  details we mention a further notion of solution, called {\em virtual solution}, which was introduced in \cite{flandoli_et.al14}: indeed
it is instructive to illustrate the power of Zvonkin transformation combined with the so-called {\em It\^o-Tanaka trick}.
Following the computations done in Section \ref{ssc:zvonkin}, for suitable large enough
$\lambda>0$, one replaces the ill-defined drift term  $\int_0^t b(s, X_s ) ds$ in \eqref{eq:SingSDE} with (component by component)
\[
\int_0^t  b^i(s, X_s) ds = u^i(0, X_0) - u^i(t, X_t)  + \lambda  \int_0^t u^i (s, X_s) ds +  \int_0^t \nabla u^i (s, X_s) dW_s,
\] 
where $u^i$ is the solution of \eqref{eq:Zvonkin_modif}. This leads to a new equation
\begin{equation}\label{eq:virtualSDE}
X_t = X_0 + u(0, X_0) - u(t, X_t)  + \lambda  \int_0^t u(s, X_s) ds  +  \int_0^t \nabla u (s, X_s) dW_s +W_t.
\end{equation}
In \cite{flandoli_et.al14} a  {\it Virtual Solution} of \eqref{eq:SingSDE} is a process $X$ that solves  \eqref{eq:virtualSDE}. It can be shown that the solution is actually independent of $\lambda$.  Notice that  \eqref{eq:virtualSDE}  is not a standard SDE because of the term $u(t, X_t)$, so  one way to actually  find a virtual solution is to transform  \eqref{eq:virtualSDE} into a standard SDE using Zvonkin transformation $\phi(t,x) = x + u(t,x)$, namely solving and SDE for $Y_t:= \phi(t, X_t)$,  as explained in Section \ref{ssc:zvonkin}.

In Section \ref{sec:sol-singular-SDE}, the meaning associated to the SDE involving a distributional drift was formulated
in terms of rough martingale problem, under the inspiration of Stroock-Varadhan equivalence
between solutions in law of SDEs and classical martingale problem.
One natural question  concerns the possibility of
formulating a direct SDE in the form
\begin{equation} \label{eq:ProperSDE}
  X_t = X_0 +  W_t + A_t^X(b),\quad t \ge 0,
  \end{equation}
  where $A^X$, defined on  a space $B$ of distributions (for instance $B = C_T\C^{(-\be)+}$),
  extends continuously the map
$ g \mapsto \int_0^t g(s,X_s) ds, $ when $g: [0,T]\times \R^d \rightarrow \R$
is a H\"older continuous function.
In most cases, the stochastic process $A^X(g)$ has zero quadratic variation 
so that $X$ is a Dirichlet process,
or at least  satisfies martingale orthogonality, see \cite{BandiniRusso_RevisedWeakDir},  so that it is  a weak Dirichlet process.

 \begin{definition}[Local time $B$-solution]
   Let $\mathbb P$ be a probability measure on
   some measurable space $(\Omega, \mathcal F)$.
 Let ${\mathcal C}_c^{0+}$ be the space defined in Section \ref{sec:top-prelim}.
We say that a process $X$ is a \emph{local time  $B$-solution}
if the following holds.
\begin{enumerate}
  \item
The linear space $C_T { \mathcal C}_c^{0+}$ is densely embedded in $B$.
\item 
  The map from $  C_T {\mathcal C}_c^{0+}$ with values in the space of continuous processes
  equipped with the u.c.p. (uniform convergence in probability) metric topology, defined by 
\[
 l \mapsto \int_0^t l(s, X_s) ds
\]
admits a continuous extension to $B$, that 
we denote by $A^{X}$.
\end{enumerate}
\end{definition}

Previous notion was first introduced in dimension $d= 1$ in Definition 6.3. of
\cite{russo_trutnau07} under the inspiration
of Proposition 3.11 of \cite{frw1}. 
This was
 extended to the general case in
Definition 6.2 of 
\cite{issoglio_russo20}.
When the reference filtration is the canonical filtration of $W$ then
the solution can be called {\it strong solution}.

Whatever is the sense we give to \eqref{eq:ProperSDE}, we need to connect it with
the solution of rough martingale problem with distributional drift $b$ considered in previous sections;
in particular, for $\varphi \in \shd$ we need
to show that
\begin{equation} \label{eq:martPB}
  M^\varphi_t:= \varphi(t,X_t) - \varphi(0,X_0) - \int_0^t L \varphi(s,X_s) ds, t \ge 0,
\end{equation}
  is a local martingale.
  When $X$ is a Dirichlet process
  (and it is often the case),  one can apply It\^o's formula for finite quadratic variation processes, see e.g.\ \cite{fo, rv4}, to $\phi(t,X_t) $,
 where $\phi \in C^{1,2}([0,T] \times \R^d)$. Thus  we can  write
\begin{equation} \label{eq:follmer}
\phi(t,X_t) - \phi(0,X_0) = \int_0^t \nabla_x \phi(s,X_s) dW_s + \sha_t^X(\phi), t \ge 0, 
\end{equation}
where
$$ \sha_t^X(\phi) = \int_0^t (\partial_s \phi (s,X_s) - \frac{1}{2} \Delta \phi(s,X_s)) ds - \int_0^t \nabla_x \phi(s,X_s) d^- A^X_s(b),$$
where the latter integral is a forward integral (extending It\^o integral), see e.g.\ \cite{Russo_Vallois_Book}.
Since $\shd \cap  C^{1,2}$ is too small, it looks quite difficult to 
prove that \eqref{eq:martPB} is a local martingale for all $\varphi \in \shd$ using
 \eqref{eq:follmer}.
In  Corollary 6.11  of \cite{issoglio_russoMPb} (resp.\ Proposition 6.10 of \cite{russo_trutnau07} in the case $d=1$),
one shows that the rough martingale problem with distributional drift and
the notion of local time $B$-solution (with $B = C_T \shc^{(-\beta)+}$) are equivalent,
under some technical restrictions.

A close formulation to the notion of local time $B$-solution that can be found in the literature,  consists in giving a proper meaning to the (a priori only formal)
SDE \eqref{eq:SingSDE} by approximating the quantity $\int_0^t b(s,X_s)ds$
by $\int_0^t b^n(s,X_s)ds$, where $(b^n)$ is a sequence of smooth functions converging to $b$.
That notion first appeared 
(at least when the drift $b$ is time-homogeneous) in \cite{basschen}
and later was exploited by various authors such as
\cite{anzelletti24, athreya2020, goudenege23}. 
One says that $(X, W)$ defined on some filtered probability space $(\Omega, \shf, (\shf_t)_{t\leq T},\mathbb P)$
is a {\em limit solution} to \eqref{eq:SingSDE} if the following holds.
\begin{itemize}
\item $X= (X_t)$ is adapted to $\shf_t$ for all $t\in[0,T]$;
\item there exists an $\R^d$-valued process $K= (K_t)$ such that $X_t = X_0 + K_t + W_t$ for all $t\in [0,T]$;
\item $W= (W_t)$ is an $\R^d$-valued Brownian motion  with respect to  $(\shf_t)$;
\item for every sequence $(b^n)$ of smooth bounded functions converging to $b$ in $C_T \C^{-\beta}$ we have
        \begin{equation*}
        \sup_{t\in[0,T]} \bigg|\int_0^t b^{n}(s, X_s)ds-K_t\bigg|\underset{n\to\infty}{\to}0 \text{ in probability}.
        \end{equation*}
\end{itemize}

\medskip

Another notion of  solution to a singular  SDE, which has some similarities with the local time $B$-solution,  is the one introduced
in Definition 5.3 of
 \cite{GaleatiGerencser25}.
The authors consider the SDE
\[
X_t = X_0 +  \int_0^t b(s, X_s) ds + B^H_t,
\] 
where $b$ is a time-dependent Schwartz distribution and $B^H$ is a fractional Brownian motion.
The concept of solution needs the definition of
 operators similar to $\mathcal A$ from Lemma \ref{lm:ssl}, defined as non-linear Young integrals,
whose existence can be established via a stochastic sewing lemma,
see Section \ref{sec:ssl} for basic ideas.

\subsection{The non-linear Fokker-Planck PDE}\label{sec:singularFP}

In finding the solution of the McKean SDE \eqref{eq:MK_dist_SDE_intro}
a fundamental analytical tool is
the singular non-linear Fokker-Planck PDE,
whose solution will coincide with the marginal densities
of the solution of the McKean SDEs. 

We first formulate a standing assumption on a non-linear matrix-valued function $F$.
That assumption is not optimal but it  allows to better illustrate
the strategy.
\begin{assumption}\label{ass:F}
  Let $F:\mathbb R \to \mathbb R^{d\times m}$ and denote
  $\tilde F:\mathbb R \to \mathbb R^{d\times m}$ the
  matrix-valued function defined by $\tilde F(z) := z F(z)$.

  We suppose that both $F$ and $\tilde F$ are bounded
  of class $C^1$ with Lipschitz derivatives.
\end{assumption}

We also consider the following assumptions on a distribution $\shb$.

\begin{assumption}\label{ass:B}
Let $0<\beta<\frac12$ and $\shb \in C_T \shc^{(-\beta)+}(\R^m)$. In particular $\shb \in C_T\shc^{-\beta}(\R^m)$.
\end{assumption}

 $F$ and $\shb$  will compose the drift $b:=F(v)\shb$ of the McKean-Vlasov equation, which also appears in the non-linear Fokker-Plank PDE studied in this section.
The assumptions above will imply that $b$ satisfies Assumption \ref{ass:b}.
For a given $\shb$,
let us consider the Fokker-Planck PDE
\begin{equation}\label{eq:nonlin_sing_FP}
\left\{
\begin{array}{l}
\partial_t v = \frac12 \Delta v - { \rm div}(\tilde F(v)\shb),\\
v(0)=v_0.
\end{array}
\right.
\end{equation}
We highlight that,
for every $t \in [0,T]$,
 the pointwise product (see Section \ref{sc:Schauderpoint} \eqref{eq:bony} and \eqref{eq:bonyt})
$\tilde F(v(t,\cdot)) \shb(t, \cdot)$
is a well-defined element of $\shc^{-\beta}$,
when $\tilde F(v(t,\cdot))$ 
has at least  $\alpha$-H\"older regularity, for some $\alpha>\beta$.

\begin{definition}\label{def:weak_sol_FP}
We say that $v\in C_T\C^{\beta+}$ is a weak solution to \eqref{eq:nonlin_sing_FP} if for all $t \in [0,T]$ and  $\varphi\in\shs(\R^d)$ the equation
\begin{equation}
  \langle v(t), \varphi\rangle-\langle v_0,\varphi \rangle=\int_0^t \frac12 \langle v(s) , \Delta \varphi \rangle ds + \int_0^t  \langle \tilde F(v(s))\shb(s) ,\nabla \varphi \rangle ds 
\end{equation}
holds true.
\end{definition}
We remark that the test functions here are taken in the Schwartz space $\shs(\R^d)$
which are far from the domain $\shd$ of the martingale problem discussed in Section \ref{ssc:singularPDE}.
In particular $\shs(\R^d)$ is generally not included in $\shd$.
\begin{definition}\label{def:mild_sol_FP}
A  mild solution to \eqref{eq:nonlin_sing_FP} is a function $v\in\C_T\C^{\beta+}$ such that 
\begin{equation}\label{eq:mild_sol_FP}
v(t)= P_{t}v_0-\int_0^t P_{t-s}\Big({\rm div}(v(s)F(v(s))\shb(s))\Big)ds, \quad t\in [0,T].
\end{equation}
\end{definition}

\begin{proposition}[Equivalence -- Proposition 4.5 in \cite{issoglio_russoPDEa}] \label{lemma:WeakMild}
  Weak and Mild solutions for \eqref{eq:nonlin_sing_FP} are equivalent.
  \end{proposition}

  \begin{theorem}[Existence -- Theorem 3.7 in \cite{issoglio_russoMK}]
    Let Assumptions \ref{ass:B} and \ref{ass:F} hold. If $v_0\in \C^{\alpha}$ with $\al\in(\be,1-\be)$ then there exists a
    unique weak
    solution $v \in C_T \C^{\alpha}$.
  \end{theorem}

  The proof of this theorem relies on the equivalence of mild and weak formulations stated in Proposition \ref{lemma:WeakMild},
  so that it is enough to concentrate on a mild solution.
  To make use of Banach fixed point theorem, one needs the solution map to
  be contractive.
This happens because the operator $f \mapsto \tilde F(f)$ can be shown to be Lipschitz on the Besov spaces $\C^{\al}$ for $\al\in(0,1)$, see
Proposition 3.1 in \cite{issoglio19}. The proof of this makes use of Schauder
estimates, see Lemma \ref{lm:schauder}.
Indeed, if  $ \tilde F(v)\shb \in \C^{-\be}$, then the divergence term will be an element of the space $\C^{-(\be+1)}$, so after the application of the semigroup one obtains an element in $\C^{-\be+1-\epsilon}$, for every $\epsilon > 0$.
Consequently  the fixed point solution map can be found when $\alpha\in(\beta, 1-\beta)$,
which explains
the constraint on the regularity parameter $\al\in(\beta, 1-\beta)$, w.r.t $x$.

We state now a continuity result, whose item (ii)
will be used for the proof of existence  of
the McKean rough martingale problem, see Theorem \ref{thm:EUMcKean} below.
 \begin{lemma}[Continuity -- Proposition 4.4 in \cite{issoglio_russoMK}]\label{lm:cont_reg_PDE}  \textcolor{white}{a}  
 \begin{itemize}
  \item[(i)]  Let Assumption \ref{ass:F} hold and $\shb^1, \shb^2$ fulfilling  Assumption \ref{ass:B}.
Let  
  $v^1$ (resp.\ $v^2$) be a  mild solution of \eqref{eq:nonlin_sing_FP} with $\shb=\shb^1$ (resp.\ $\shb=\shb^2$).
  For any   $\alpha\in (\beta,1-\beta)$ such that $v^1,v^2\in C_T \mathcal C^\alpha$, there exists a function
    $\ell_\alpha:\R^+\times \R^+\to \R^+$, increasing in the second variable,
     such that 
\[
\|v^1(t) - v^2(t)\|_{\alpha}  \leq \ell_\alpha(\|v_0\|_\alpha, \|\shb^1\| \vee\|\shb^2\|  )  \| \shb^1 - \shb^2 \|_{C_T\mathcal C^{-\beta}},
\]
 for all $t\in[0,T]$.
\item[(ii)]  Let $\shb$, and a sequence $(\shb^n)$ 
fulfilling Assumption \ref{ass:B}. 
 Let $v^n$ be a mild solution of   \eqref{eq:nonlin_sing_FP} with $\shb=\shb^n$ and $v$ be a mild solution of \eqref{eq:nonlin_sing_FP}.
If $\shb^n\to \shb$ in  $C_T \mathcal C^{-\beta}$ then $v^n \to v$ in $C_T \mathcal C^{\alpha}$ for some $\alpha>\beta$.
\end{itemize} 
\end{lemma}

\subsection{McKean SDEs with density dependence}\label{sec:McKean-point}

We  recall the McKean-Vlasov system \eqref{eq:MK_dist_SDE_intro} introduced at the beginning of the section, that is  
\begin{equation}\label{eq:MK_dist_SDE}
\left\{
\begin{array}{l}
X_t=X_0+\int_0^t F(v(s,X_s))\shb(s,X_s)ds+W_t,\\
v(t,\cdot) \text{ is the law density of } X_t,\\
X_0\sim \nu,
\end{array}
\right.
\end{equation}
where   $W$ is a $d$-dimensional Brownian motion, $F$ is a non-linear function $F:\mathbb R \to \mathbb R^{d\times m}$, the singular part $\shb$ of the drift is an $m$-dimensional vector whose components are elements of $ C_T\C^{(-\be)+}(\R^d)$ and $\nu $ is the initial law density.
In what follows we will omit the dimensions $d$ and $m$ for brevity. 
The rigorous definition  of solution to \eqref{eq:MK_dist_SDE} is given below.
\begin{definition}[Definition 5.1 in \cite{issoglio_et.al24}, Definition 6.2 in \cite{issoglio_russoMK}]\label{def:MK_MP_Alt}
  Let $\Omega_T$ be the canonical space  $C([0,T];\R^d)$ equipped with
  its Borel $\sigma$-field $\shf_T$ and $Z$ be the canonical process.
  A solution (in law) of the McKean problem \eqref{eq:MK_dist_SDE} 
  is  a Borel  probability measure $\P$ on  $\Omega_T$,
  such that the following holds.
\begin{itemize}
\item[-] The law $\shl^\P_{Z_t}$ admits a density denoted by $v(t,\cdot)$;
\item[-] $v: [0,T ]\times\R^d  \to \R$  is an element of $C_T \mathcal C^{\beta+}$;
\item[-] $\P$ solves the rough martingale problem as in Definition \ref{def:MK_MP} with distributional drift $b(t,\cdot) :=F(v(t,\cdot))\shb(t,\cdot)$.
\end{itemize}
We say that the McKean problem  \eqref{eq:MK_dist_SDE}, denoted by McKean--rMP$(F(v)\shb,\nu)$ for shortness, admits {\it uniqueness} if, whenever
we have two solutions $\P$
and  $\bar \P$,  then $\P = \bar \P$.

By a slight abuse of notation, when $\nu$ has a density $v_0$, we denote the McKean problem  \eqref{eq:MK_dist_SDE} by McKean--rMP$(F(v)\shb,v_0)$.
\end{definition}
Below we state the main result of existence and uniqueness of a solution to the McKean problem  \eqref{eq:MK_dist_SDE}.
\begin{theorem}[Existence and Uniqueness -- Theorem 6.3 in \cite{issoglio_russoMK}] \label{thm:EUMcKean}
Let Assumption \ref{ass:F} hold and let $\shb$ fulfilling
   Assumption \ref{ass:B}. Let the initial probability $\nu$ have a density $v_0$  such that $v_0 \in\C^{\be+}$. 
Then  there exists a solution $\P$ to the McKean problem  \eqref{eq:MK_dist_SDE}. Furthermore, the McKean problem \eqref{eq:MK_dist_SDE} admits uniqueness.
\end{theorem}

The strategy of the proof for existence is summarised below.
\begin{description}
\item[Step 1] Solve the non-linear Fokker-Planck equation \eqref{eq:nonlin_sing_FP} to get a candidate law density $v(t,\cdot)$ for the unknown solution $Z_t$. Here results from Section \ref{sec:singularFP} are used.
\item[Step 2] Plug $v(t,\cdot)$ into the drift of the McKean--rMP$(F(v)\shb,v_0)$ \eqref{eq:MK_dist_SDE}, namely define $b:=F(v)\shb$, and solve the  rough martingale problem rMP$(b, v_0)$.  Here results from Section \ref{sec:linearSDE} are used.  
\item[Step 3]  The idea now, is to use Figalli-Trevisan's superposition principle.
\end{description}
  The last step cannot be done directly because the coefficient $b$ is too singular to apply Figalli-Trevisan's superposition principle, so we need to go through the regularised versions of the SDEs and PDEs.  In particular, we mollify $\shb$ with a smooth sequence $(\shb^n)$ and consider the unique solution $v^n$ of the regularised Fokker-Planck equation \eqref{eq:nonlin_sing_FP} with $b^n:= F(v^n)\shb^n$, and the unique solution  $\P^n$ of the rough martingale problem rMP$(b^n, v_0)$. For each fixed $n$, Figalli-Trevisan  superposition principle (Theorem \ref{thm:figalli}) applies and so, given the solution $v^n$ of the regularised and linearised\footnote{we interpret the drift $v^n F(v^n)\shb^n$ as $v^n b^n$} Fokker-Planck equation,  we know that there exists a probability $\bar\P^n$ which is a  solution  to rMP$(b^n,v_0)$ such that  $\shl^{\bar \P^n}_{Z_t}(dx)  = v^n(t,x)dx$.  By uniqueness of the rough martingale problem rMP$(b^n, v_0)$ we have $ \bar\P^n = \P^n$, therefore $\shl^{\P^n}_{Z_t}(dx)  = v^n(t,x)dx$.
To conclude, we must take the limit on both sides as $n\to\infty$ to show that   $\shl^\P_{Z_t} (dx) = v(t, x)dx$, which implies that $\P$ is a solution of the McKean-rMP$(F(v)\shb, v_0)$. The right-hand side limit  $v^n\to v$ is guaranteed by continuity results on the Fokker-Planck equation (Lemma \ref{lm:cont_reg_PDE} (ii)) which crucially also ensures that $v$ is the unique solution to the non-linear Fokker-Planck equation \eqref{eq:nonlin_sing_FP}.  The left-hand side limit   $\shl^{\P^n}_Z \to \shl^\P_Z $ holds  by continuity results on the rough martingale problem (Theorem \ref{thm:contMP}) under the assumption that $b^n =F(v^n)\shb^n  \to F(v)\shb = b$ in $C_T\shc^{-\beta}$, which is true thanks to the continuity properties of the pointwise product.

\medskip Concerning {\em uniqueness} of the  McKean--rMP$(F(v)\shb,v_0)$, the idea is to show that, given a solution $\P$, its marginals $v(t,\cdot)$
fulfill the non-linear Fokker-Planck equation which is well-posed, see Section \ref{sec:singularFP}. Finally the uniqueness reduces to the uniqueness of the rough martingale problem rMP$(b, v_0)$.

\subsubsection*{A convolutional-type drift}

So far, we have focused on singular non-linear  McKean type SDE  with drift
$F(v) b $ being  a Schwartz distribution,  $v$ being the law density, $F(v)$ being a function and the product $F(v) b $ being indeed a pointwise product.
In the literature other types of singular McKean type drifts have been considered, such as those   of convolutional type $b * v$.
We have mentioned this kind of convolutional drifts in Section \ref{sec:LpLq_p} in the $L^p$-$L^q$ framework.
When $b$ is a (good enough) distribution, 
the overall drift is a 
function due to the smoothing properties of
the convolution. 
In  \cite{chaudru_jabir_menozzi22, chaudru_jabir_menozzi23} the coefficient $b$ is an element of $L^r$ in time and of a negative Besov space in space and the noise is an $\alpha$-stable process with $\alpha\in(1,2]$, the case $\alpha=2$ corresponding to Brownian  noise.
Exploiting the regularisation effect of the convolution, the McKean SDE  can be formulated as a true stochastic differential equation, without the need of martingale problem.
In that paper the authors prove existence in law for the McKean-Vlasov SDE, and absolute continuity with respect to the Lebesgue measure for the law of the solution, extending their results obtained in the non-McKean framework \cite{chaudru_menozzi}, 
see  also
Section \ref{sec:sol-singular-SDE}.
Moreover under stronger assuptions they also prove strong well-posedness. 
We draw the attention on the recent article  \cite{GaleatiGerencser25}, where the noise is a fractional Brownian motion, following the ideas of \cite{mayorcas}. 
It  first focuses on the non-McKean problem 
 via the stochastic sewing lemma, see Section \ref{sec:ssl}.
Later  those results are applied to  a class of McKean-Vlasov SDEs (called distribution dependent SDEs therein) with additive fractional Brownian noise.
The drift must be Lipschitz in the Wasserstein topology but can be very singular in the Besov sense in space.
We remark that they have to give a meaning to the equation due to the distributional nature of the drift, and they do so by extending an integral operator
analogously to the case of
 the local time $B$-solution, see the end of Section \ref{sec:solution_sing_SDE}.

\subsection{Singular degenerate McKean-Vlasov SDEs}
\label{sc:sing_deg_McK}

In this section we follow \cite{issoglio_et.al24}  and review a generalisation of the McKean SDE studied in Section \ref{sec:McKean-point},
where we allow the diffusion matrix to be degenerate, in particular non-invertible. 
From a {\em regularisation by noise} perspective, this case is more difficult to solve, because the effect of the noise is present only in some directions of $\R^d$. 
Let us denote $X= (V,U)$, where the first $l$ components $V$ have a non-zero noise, while the last $d-l$ components $U$ are noise-free.
The case considered in \cite{issoglio_et.al24} is of the form
\begin{equation}\label{eq:mkv_deg}
\begin{cases}
d V_t = \Big( F\big( v(t, X_t)  \big) \mathcal B(t, X_t) +  B_0 X_t \Big) d t   +     d W_t       \\
d U_t =  B_1 X_t d t \\
v(t,\cdot) \text{ is the law density of  } X_t,\\
X_0=(V_0,U_0) \sim\nu \text{ with }\nu(dx) = v_0(x) dx .
\end{cases}
\end{equation}
Here  
$W$ is an $l$-dimensional Brownian motion, $ B_0\in\mathcal{M}^{l \times d}$ and 
$ B_1\in\mathcal{M}^{(d-l)\times d}$ are constant matrices, and $F:[0,+\infty) \to \mathcal{M}^{l\times m}$ is a  matrix-valued Borel function fulfilling Assumption \ref{ass:F}. The singular term $\mathcal B$
fulfills
 Assumption \ref{ass:B} as in Section \ref{sec:McKean-point}
 and this is an $m$-dimensional vector of functions of time taking values in a suitable Besov space. Notice that here we have the linear terms $ B_0$ and $ B_1$
 which were not present in \eqref{eq:MK_dist_SDE}. This linear term $BX:= \binom{B_0}{B_1} X$ is fundamental to be able to obtain well-posedness of the equation, even in the case
 when $\mathcal B$ is smooth, and the condition that one requires is known as {\em hypoellipticity} (see e.g.\ \cite{hoermander} or Remark 2.1 of \cite{issoglio_et.al24})
 of the Kolmogorov operator $\mathcal K$ associated to SDE \eqref{eq:mkv_deg} when $F\equiv 0$. We  mention that this condition is equivalent to requiring a special block form for the matrix $ B_1$. Further details and precise references can be found in Section 1.3 in \cite{issoglio_et.al24}.

The system  \eqref{eq:mkv_deg}  is a generalisation of the so-called {\em kinetic} equation, that is obtained when $d=2l$ and $U,V$ have the same dimension $l$ and
represent position and velocity, respectively, of a particle in the phase-space $\R^{2l}$.
This setting is the one considered for example in \cite{hao2021singular} 
and where the coefficients involve some rough distributions
as in \cite{issoglio_et.al24}. 
In the latter case the hypoellipticity condition is fulfilled.

The general roadmap to show well-posedness of equation \eqref{eq:mkv_deg} is the same that has been illustrated in details in Sections \ref{sec:top-prelim}--\ref{sec:McKean-point}. This plan however, needs extensive preliminary work to retrieve the  setting from Section \ref{sec:top-prelim}, which includes the definition of Besov spaces and Schauder's estimates in the degenerate case.
As already mentioned, the main problem is that the noise acts only in some directions, thus the smoothing effect of the noise is compromised. 
To ensure that the noise is propagated in all directions, hence affects all components of $X$,  we  impose the block form condition on the linear drift $ B_1$. This leads to a different rescaling of the components, which translates into the so-called {\em anisotropic norm}. For more details see Section 2.1 in \cite{issoglio_et.al24}.\footnote{As an example, the (non-McKean) simplest  kinetic equation is obtained when $d=2, l=1, F\equiv 0$ and $B=\binom 01$ to get
\[
\begin{cases}
d V_t =     d W_t       \\
d U_t = V_t d t \\
X_0=(V_0,U_0) \sim\nu. 
\end{cases}
\]
In this case, intuitively one has the rescaling $(\Delta V_t)^2 \approx (\Delta W_t)^2 \approx \Delta t$ for the first component and  $(\Delta U_t)^2 \approx (W \Delta t)^2\approx (\Delta t)^3$ for the second component. The anisotropic norm is given by $|(u,v)|_B =|v|^1+ |u|^{\frac13}$. }

The anisotropic norm must be embedded in the setting we work in, in particular one needs to define the anisotropic version of the Besov spaces  defined in Definition \ref{def:Besov} point 2., which are denoted by $\mathcal C^\gamma_B$ and are called {\em anisotropic Besov-H\"older spaces}. The construction follows the
 Definition \ref{def:Besov}, which is the one of standard Besov spaces,
 provided that anisotropic anuli
are used to define the dyadic partition of unity $\{\varphi_j\}$, see also Figure \ref{fig:an}. 
\begin{figure}
\centering
\includegraphics[width=300pt]{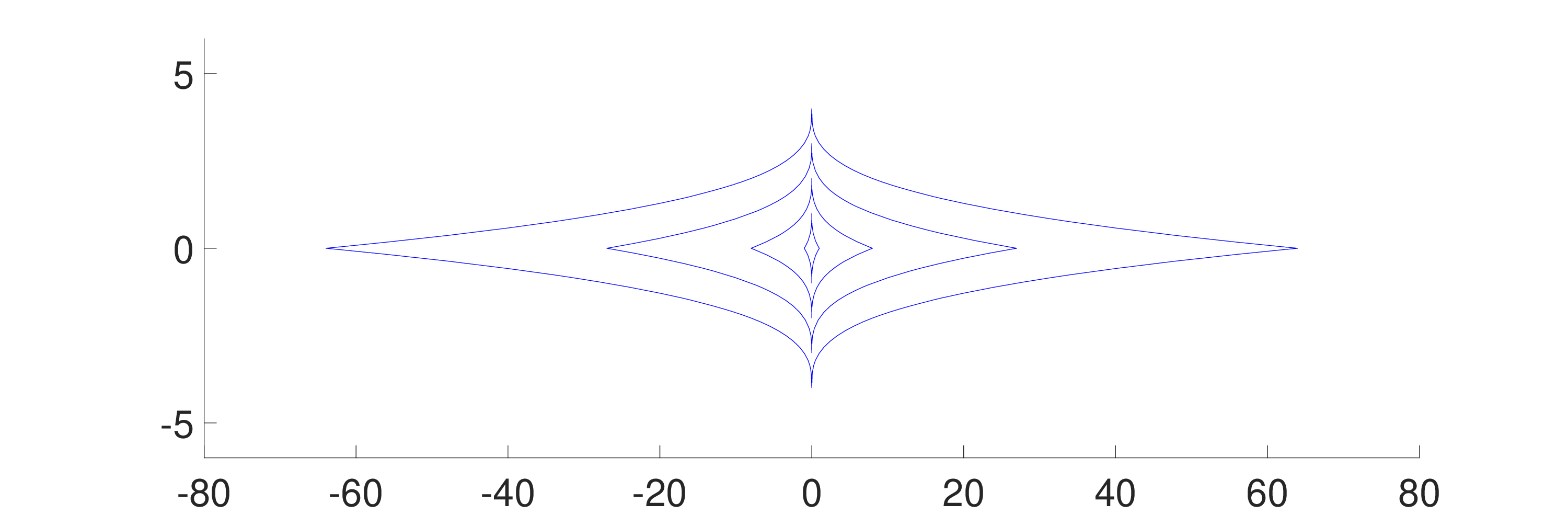}
\caption{Anisotropic anuli}\label{fig:an}
\end{figure}
Similarly to the case of standard Besov norms,  one also derives  the pointwise product estimate \eqref{eq:bony} in the corresponding anisotropic Besov-H\"older spaces.
The last key ingredient are Schauder estimates  (the equivalent of Lemma \ref{lm:schauder}) for the corresponding semigroup generated by the degenerate Kolmogorov operator $\mathcal K$.  Among the main tools to prove those,
are the special geometry induced by the anisotropic norm, and  Gaussian-type upper bounds for the fundamental solution to the Kolmogorov equation  $\mathcal K f =0$, which can be found in  \cite{lucertini2022optimal}.
Notice that  similar estimates for the purely kinetic case were previously obtained in \cite{zhang2024cauchy}. 

Once these preliminary results are available, one can carry out the study of the well-posedness for the Kolmogorov equation, the non-linear Fokker-Planck equation,
the rough martingale problem and finally the degenerate McKean SDE \eqref{eq:mkv_deg} in the anisotropic setting in \cite{issoglio_et.al24}.
Well-posedness for the latter
follows by similar arguments to those of the proof of Theorem \ref{thm:EUMcKean}.

\subsubsection*{Elsewhere in the literature}

As mentioned in the Introduction,  well-posedness for equations with McKean-Vlasov coefficients
for drift and diffusion of regular convolutional type
of the marginal measure with a bounded function $\xi \mapsto \shb(t,x, \xi)$
both in the drift and the diffusion,
was recently proved
in \cite{veretennikov2023weak, rondelli} for the  kinetic (and more general degenerate) case.
The function $\shb$ is supposed 
to be uniformly continuous w.r.t.\ to $(x, \xi)$
and the diffusion convolution function is non-degenerate.
\cite{veretennikov2023weak}
shows  weak existence using tightness arguments and
\cite{rondelli} shows weak well-posedness under some more restrictive assumptions on the coefficients.
Another significant contribution in the kinetic case
is provided by
\cite{zhang2021second}, where
the convolution function is more irregular and density dependent. 
In \cite{hao2021singular} the authors consider the kinetic case with convolutional dependence on the law through
a generalised function living in some negative Besov space of order $-\beta$. 
They allow $\beta\in(1/2, 2/3)$, hence going beyond the Young regime, and make use of paracontrolled calculus. 
In both papers \cite{zhang2021second, hao2021singular}, mollification techniques on the original process $X$ are exploited to prove  existence, while uniqueness is proved by Markov marginal uniqueness techniques. 
In particular, in   \cite{hao2021singular}  they  apply   Girsanov's transformation to obtain an SDE for which Markov marginal uniqueness techniques are subsequently used. 

In \cite{HaoZhangRoeck}  the authors  consider a kinetic mean-field type SDE driven by an $\alpha$-stable noise for $\alpha\in(1, 2]$. The drift lives in a (possibly) negative Besov space, and it is convoluted with the unknown law of the solution, like in the setting presented in Section \ref{sec:LpLq_p}. The convolution operation smoothens the drift, so that it becomes a function and there are no issues related to the definition of solution to the SDE. Then the authors study (strong and weak) well-posedness of mean-field kinetic SDE,
under different hypotheses.
Moreover in \cite{hao2024} the authors investigate the propagation of chaos for a particle system whose mean-field limit is the equation studied in \cite{HaoZhangRoeck}. In the interacting particle system  the drift is of convolutional type like in Section \ref{sec:LpLq_p}, and since the interaction is of moderate nature, the final drift form includes a Besov distribution convoluted twice, once with the empirical measure and once with the mollifying kernel typical of moderate interaction.

\bigskip
{\bf Acknowledgments}
The second named authors acknowledges partial financial support under the National Recovery and Resilience Plan (NRRP), through the Project  “Non–Markovian Dynamics and Non-local Equations” – 202277N5H9 - CUP: D53D23005670006, funded by the Italian Ministry of Ministry of University and Research.

   The research of the third named author was partially supported by the  ANR-22-CE40-0015-01 (SDAIM).

\bibliographystyle{plain}
\bibliography{../../../BIBLIO_FILE/Biblio_Tesi_Luca}

\def\cprime{$'$}
\begin{thebibliography}{100}

\bibitem{BallyAlfonsi}
A.~Alfonsi and V.~Bally.
\newblock Construction of {Boltzmann} and {McKean}-{Vlasov} type flows (the
  sewing lemma approach).
\newblock {\em Ann. Appl. Probab.}, 33(5):3351--3386, 2023.

\bibitem{Ambrosio2004}
L.~Ambrosio.
\newblock Transport equation and {C}auchy problem for {BV} vector fields and
  applications.
\newblock {\em Journées Équations aux dérivées partielles}, pages 1--11, 6
  2004.

\bibitem{anzelletti24}
L.~Anzeletti, A.~Richard, and E.~Tanr{\'e}.
\newblock Regularisation by fractional noise for one-dimensional differential
  equations with distributional drift.
\newblock {\em Electron. J. Probab.}, 28:49, 2023.
\newblock Id/No 135.

\bibitem{athreya2020}
S.~Athreya, O.~Butkovsky, and L.~Mytnik.
\newblock Strong existence and uniqueness for stable stochastic differential
  equations with distributional drift.
\newblock {\em Ann. Probab.}, 48(1):178--210, 2020.

\bibitem{bahouri2011fourier}
H.~Bahouri, J.-Y. Chemin, and R.~Danchin.
\newblock {\em Fourier analysis and nonlinear partial differential equations},
  volume 343.
\newblock Springer, 2011.

\bibitem{bak86}
P.~Bak.
\newblock {\em How nature works. {The} science of self-organized criticality}.
\newblock Berlin: Springer, 1996.

\bibitem{BandiniRusso_RevisedWeakDir}
E.~Bandini and F.~Russo.
\newblock Weak {D}irichlet processes and generalized martingale problems.
\newblock {\em Stochastic Process. Appl.}, 170(104261), 2024.

\bibitem{BanJa}
P.~B{\'a}ntay and I.~M. J{\'a}nosi.
\newblock Avalanche dynamics from anomalous diffusion.
\newblock {\em Physical review letters}, 68(13):2058, 1992.

\bibitem{RehmeierPLap}
V.~Barbu, M.~Rehmeier, and M.~R{\"o}ckner.
\newblock $p$-{Brownian} motion and the $p$-{Laplacian}.
\newblock {\em Preprint Arxiv 2409.18744, 2024}, 2024.

\bibitem{BarbuRockSIAM}
V.~Barbu and M.~R\"{o}ckner.
\newblock Probabilistic representation for solutions to nonlinear
  {F}okker-{P}lanck equations.
\newblock {\em SIAM J. Math. Anal.}, 50(4):4246--4260, 2018.

\bibitem{BarbuRoeckSuperposition}
V.~Barbu and M.~R{\"o}ckner.
\newblock From nonlinear {Fokker}-{Planck} equations to solutions of
  distribution dependent {SDE}.
\newblock {\em Ann. Probab.}, 48(4):1902--1920, 2020.

\bibitem{BarbuRoeckJFA}
V.~Barbu and M.~R{\"o}ckner.
\newblock Uniqueness for nonlinear {Fokker}-{Planck} equations and for
  {McKean}-{Vlasov} {SDEs}: the degenerate case.
\newblock {\em J. Funct. Anal.}, 285(4):37, 2023.
\newblock Id/No 109980.

\bibitem{BRR2}
V.~Barbu, M.~R\"ockner, and F.~Russo.
\newblock Probabilistic representation for solutions of an irregular porous
  media type equation: the irregular degenerate case.
\newblock {\em Probab. Theory Related Fields}, 151(1-2):1--43, 2011.

\bibitem{Baren}
G.~I. Barenblatt.
\newblock On some unsteady motions of a liquid and gas in a porous medium.
\newblock {\em Akad. Nauk SSSR. Prikl. Mat. Meh.}, 16:67--78, 1952.

\bibitem{basschen}
R.~F. Bass and Z.-Q. Chen.
\newblock Brownian motion with singular drift.
\newblock {\em Ann. Probab.}, 31(2):791--817, 2003.

\bibitem{cuvelier1}
N.~Belaribi, F.~Cuvelier, and F.~Russo.
\newblock A probabilistic algorithm approximating solutions of a singular {PDE}
  of porous media type.
\newblock {\em Monte Carlo Methods Appl.}, 17(4):317--369, 2011.

\bibitem{cuvelier2}
N.~Belaribi, F.~Cuvelier, and F.~Russo.
\newblock Probabilistic and deterministic algorithms for space multidimensional
  irregular porous media equation.
\newblock {\em Stoch. Partial Differ. Equ., Anal. Comput.}, 1(1):3--62, 2013.

\bibitem{BCR2}
N.~Belaribi and F.~Russo.
\newblock Uniqueness for {F}okker-{P}lanck equations with measurable
  coefficients and applications to the fast diffusion equation.
\newblock {\em Electron. J. Probab.}, 17:no. 84, 28, 2012.

\bibitem{Ben_Vallois}
S.~Benachour, P.~Chassaing, B.~Roynette, and P.~Vallois.
\newblock Processus associ\'es \`a\ l'\'equation des milieux poreux.
\newblock {\em Ann. Scuola Norm. Sup. Pisa Cl. Sci. (4)}, 23(4):793--832, 1996.

\bibitem{benilan}
Ph. Benilan, H.~Br{\'e}zis, and M.~G. Crandall.
\newblock A semilinear equation in {$L^1(\mathbb R^N)$}.
\newblock {\em Ann. Sc. Norm. Super. Pisa, Cl. Sci., IV. Ser.}, 2:523--555,
  1975.

\bibitem{BenilanCrandall}
Ph. Benilan and M.~G. Crandall.
\newblock The continuous dependence on $\phi$ of solutions of {$(u_ t - \Delta
  (\phi)(u))=0$}.
\newblock {\em Indiana Univ. Math. J.}, 30:161--177, 1981.

\bibitem{BRR}
P.~Blanchard, M.~R{\"o}ckner, and F.~Russo.
\newblock Probabilistic representation for solutions of an irregular porous
  media type equation.
\newblock {\em Ann. Probab.}, 38(5):1870--1900, 2010.

\bibitem{BKRRS}
V.I. Bogachev, T.I. Krasovitskii, M.~R{\"o}ckner, F.~Russo, and S.V.
  Shaposhnikov.
\newblock The superposition principle to {Fokker-Planck Kolmogorov} equation
  with potential terms.
\newblock {\em To appear: Pure and Applied Functional Analysis}, 2025.

\bibitem{bogachev_superposition}
V.I. Bogachev, M.~R{\"o}ckner, and S.~V. Shaposhnikov.
\newblock On the {Ambrosio}-{Figalli}-{Trevisan} superposition principle for
  probability solutions to {Fokker}-{Planck}-{Kolmogorov} equations.
\newblock {\em J. Dyn. Differ. Equations}, 33(2):715--739, 2021.

\bibitem{bony}
J.-M. Bony.
\newblock Calcul symbolique et propagation des singularites pour les
  \'{e}quations aux d\'{e}riv\'{e}es partielles non lin\'{e}aires.
\newblock {\em Ann. Sci. Ec. Norm. Super.}, 14:209--246, 1981.

\bibitem{bossy2011conditional}
M.~Bossy, J.-F. Jabir, and D.~Talay.
\newblock On conditional {McKean Lagrangian} stochastic models.
\newblock {\em Probability theory and related fields}, 151:319--351, 2011.

\bibitem{butkovsky.et.al2021}
O.~Butkovsky, K.~Dareiotis, and M.~Gerencs{\'e}r.
\newblock Approximation of {SDEs}: a stochastic sewing approach.
\newblock {\em Probab. Theory Relat. Fields}, 181(4):975--1034, 2021.

\bibitem{cannizzaro}
G.~{Cannizzaro} and K.~{Chouk}.
\newblock {Multidimensional SDEs with singular drift and universal construction
  of the polymer measure with white noise potential}.
\newblock {\em {Ann. Probab.}}, 46(3):1710--1763, 2018.

\bibitem{cardaliaguet}
P.~Cardaliaguet.
\newblock Mean field games: the master equation and the mean field limit.
\newblock {\em S{\'e}min. Laurent Schwartz, EDP Appl.}, 2015-2016, 2016.

\bibitem{cardaliaguetPrinceton}
P.~Cardaliaguet, F.~Delarue, J.-M. Lasry, and P.-L. Lions.
\newblock {\em The master equation and the convergence problem in mean field
  games}, volume 201 of {\em Ann. Math. Stud.}
\newblock Princeton, NJ: Princeton University Press, 2019.

\bibitem{carmona-delarueI}
R.~Carmona and F.~Delarue.
\newblock {\em Probabilistic theory of mean field games with applications I},
  volume~84 of {\em Probability Theory and Stochastic Modelling}.
\newblock Springer International Publishing, 2018.

\bibitem{carmona-delarueII}
R.~Carmona and F.~Delarue.
\newblock {\em Probabilistic theory of mean field games with applications II},
  volume~84 of {\em Probability Theory and Stochastic Modelling}.
\newblock Springer International Publishing, 2018.

\bibitem{catellier_chouk}
R.~{Catellier} and K.~{Chouk}.
\newblock {Paracontrolled distributions and the 3-dimensional stochastic
  quantization equation}.
\newblock {\em {Ann. Probab.}}, 46(5):2621--2679, 2018.

\bibitem{deRaynal}
P.~E. {Chaudru de Raynal}.
\newblock {Strong well-posedness of McKean-Vlasov stochastic differential
  equations with {H}\"older drift}.
\newblock {\em {Stochastic Processes Appl.}}, 130(1):79--107, 2020.

\bibitem{chaudru_jabir_menozzi23}
P.E. Chaudru~de Raynal, J.F. Jabir, and S.~Menozzi.
\newblock Multidimensional stable driven {M}c{K}ean-{V}lasov {SDE}s with
  distributional interaction kernel: critical thresholds and related models.
\newblock {\em Preprint ArXiv 2302.09900}, 2023.

\bibitem{chaudru_jabir_menozzi22}
P.E. Chaudru~de Raynal, J.F. Jabir, and S.~Menozzi.
\newblock Multidimensional stable driven {M}c{K}ean-{V}lasov {SDE}s with
  distributional interaction kernel: a regularization by noise perspective.
\newblock {\em Stoch PDE: Anal Comp}, 13:367--420, 2025.

\bibitem{chaudru_menozzi}
P.E. Chaudru~de Raynal and S.~Menozzi.
\newblock On {M}ultidimensional stable-driven {S}tochastic {D}ifferential
  {E}quations with {B}esov drift.
\newblock {\em Electronic Journal of Probability}, 27:1--52, 2022.

\bibitem{ciotir-goreac}
I.~Ciotir, R.~Fukuizumi, and D.~Goreac.
\newblock The stochastic fast logarithmic equation in {{\(\mathbb{R}^d\)}} with
  multiplicative {Stratonovich} noise.
\newblock {\em J. Math. Anal. Appl.}, 542(1):25, 2025.
\newblock Id/No 128786.

\bibitem{ciotir}
I.~Ciotir and F.~Russo.
\newblock Probabilistic representation for solutions of a porous media type
  equation with {Neumann} boundary condition: {The} case of the half-line.
\newblock {\em Differ. Integral Equ.}, 27(1-2):181--200, 2014.

\bibitem{Davie}
A.~M. Davie.
\newblock Uniqueness of solutions of stochastic differential equations.
\newblock {\em Int. Math. Res. Not.}, 2007(24):26, 2007.
\newblock Id/No rnm124.

\bibitem{diel}
F.~Delarue and R.~Diel.
\newblock Rough paths and 1d {SDE} with a time dependent distributional drift:
  application to polymers.
\newblock {\em Probab. Theory Related Fields}, 165(1-2):1--63, 2016.

\bibitem{dellacherie}
C.~Dellacherie and P.-A. Meyer.
\newblock {\em Probabilities and potential. {Transl}. from the {French}},
  volume~29 of {\em North-Holland Math. Stud.}
\newblock Elsevier, Amsterdam, 1978.

\bibitem{djete}
M.~F. Djete.
\newblock Non--regular {McKean}--{Vlasov} equations and calibration problem in
  local stochastic volatility models.
\newblock {\em Preprint Arxiv 2208.09986, 2022}, 2022.

\bibitem{dunford-schwartz}
N.~Dunford and J.~T. Schwartz.
\newblock {\em Linear operators. {P}art {I}}.
\newblock Wiley Classics Library. John Wiley \& Sons Inc., New York, 1988.
\newblock General theory, With the assistance of William G. Bade and Robert G.
  Bartle, Reprint of the 1958 original, A Wiley-Interscience Publication.

\bibitem{dupire}
B.~Dupire.
\newblock Pricing with a smile.
\newblock {\em Risk Mag.}, (7):18--20, 1994.

\bibitem{engelbert}
H.~J. Engelbert and W.~Schmidt.
\newblock On solutions of one-dimensional stochastic differential equations
  without drift.
\newblock {\em Z. Wahrscheinlichkeitstheor. Verw. Geb.}, 68:287--314, 1985.

\bibitem{ethier_kurtz}
S.~N. Ethier and T.~G. Kurtz.
\newblock {\em Markov processes. {Characterization} and convergence.}
\newblock Wiley Ser. Probab. Stat. Hoboken, NJ: John Wiley \& Sons, 2005.

\bibitem{FeyelDeLaPradelle1}
D.~Feyel and A.~de~La~Pradelle.
\newblock Curvilinear integrals along enriched paths.
\newblock {\em Electron. J. Probab.}, 11:860--892, 2006.
\newblock Id/No 34.

\bibitem{figalli}
A.~Figalli.
\newblock Existence and uniqueness of martingale solutions for {SDE}s with
  rough or degenerate coefficients.
\newblock {\em J. Funct. Anal.}, 254(1):109--153, 2008.

\bibitem{flandoli_et.al14}
F.~Flandoli, E.~Issoglio, and F.~Russo.
\newblock Multidimensional {SDE}s with distributional coefficients.
\newblock {\em T. Am. Math. Soc.}, 369:1665--1688, 2017.

\bibitem{FlandoliOlivera2019}
F.~Flandoli, M.~Leimbach, and Ch. Olivera.
\newblock Uniform convergence of proliferating particles to the {FKPP}
  equation.
\newblock {\em J. Math. Anal. Appl.}, 473(1):27--52, 2019.

\bibitem{FlandoliOlivera2020}
F.~Flandoli, Ch. Olivera, and M.~Simon.
\newblock Uniform approximation of 2 dimensional {Navier}-{Stokes} equation by
  stochastic interacting particle systems.
\newblock {\em SIAM J. Math. Anal.}, 52(6):5339--5362, 2020.

\bibitem{frw2}
F.~Flandoli, F.~Russo, and J.~Wolf.
\newblock Some {SDE}s with distributional drift. {I}. {G}eneral calculus.
\newblock {\em Osaka J. Math.}, 40(2):493--542, 2003.

\bibitem{frw1}
F.~Flandoli, F.~Russo, and J.~Wolf.
\newblock Some {SDE}s with distributional drift. {II}. {L}yons-{Z}heng
  structure, {I}t\^o's formula and semimartingale characterization.
\newblock {\em Random Oper. Stochastic Equations}, 12(2):145--184, 2004.

\bibitem{fo}
H.~F{\"o}llmer.
\newblock Calcul d'{I}t\^o sans probabilit\'es.
\newblock In {\em Seminar on Probability, XV (Univ. Strasbourg, Strasbourg,
  1979/1980) (French)}, volume 850 of {\em Lecture Notes in Math.}, pages
  143--150. Springer, Berlin, 1981.

\bibitem{friedman}
A.~Friedman.
\newblock Stochastic differential equations and applications. {Vol}. 1.
\newblock Probability and {Mathematical} {Statistics}. {Vol}. 28. {New} {York}
  - {San} {Francisco} - {London}, 1975.

\bibitem{ruediger}
M.~Friesen, B.~R{\"u}diger, and P.~Sundar.
\newblock The {Enskog} process for hard and soft potentials.
\newblock {\em NoDEA, Nonlinear Differ. Equ. Appl.}, 26(3):42, 2019.
\newblock Id/No 20.

\bibitem{GaleatiGerencser25}
L.~Galeati and M.~Gerencs{\'e}r.
\newblock Solution theory of fractional {SDEs} in complete subcritical regimes.
\newblock {\em Forum Math. Sigma}, 13:66, 2025.
\newblock Id/No e12.

\bibitem{mayorcas}
L.~Galeati, F.~A. Harang, and A.~Mayorcas.
\newblock Distribution dependent {SDEs} driven by additive fractional
  {Brownian} motion.
\newblock {\em Probab. Theory Relat. Fields}, 185(1-2):251--309, 2023.

\bibitem{GaleatiLing23}
L.~Galeati and Ch. Ling.
\newblock {Stability estimates for singular SDEs and applications}.
\newblock {\em Electronic Journal of Probability}, 28(none):1 -- 31, 2023.

\bibitem{grr}
A.~M. Garsia, E.~Rodemich, and H.~jun. Rumsey.
\newblock A real variable lemma and the continuity of paths of some {Gaussian}
  processes.
\newblock {\em Indiana Univ. Math. J.}, 20:565--578, 1970.

\bibitem{goudenege23}
L.~Gouden{\`e}ge, El~M. Haress, and A.~Richard.
\newblock Numerical approximation of {SDEs} with fractional noise and
  distributional drift.
\newblock {\em Stochastic Processes Appl.}, 181:38, 2025.
\newblock Id/No 104533.

\bibitem{gubinelli}
M.~Gubinelli.
\newblock Controlling rough paths.
\newblock {\em J. Funct. Anal.}, 216(1):86--140, 2004.

\bibitem{gubinelli_imkeller_perkowski}
M.~Gubinelli, P.~Imkeller, and N.~Perkowski.
\newblock Paracontrolled distributions and singular {PDE}s.
\newblock {\em Forum of Mathematics, Pi}, 3:75 pages, 2015.

\bibitem{labordere}
J.~Guyon and P.~Henry-Labord{\`e}re.
\newblock {\em Nonlinear option pricing}.
\newblock Chapman Hall/CRC Financ. Math. Ser. Boca Raton, FL: CRC Press, 2014.

\bibitem{gyongy}
I.~Gy{\"o}ngy.
\newblock Mimicking the one-dimensional marginal distributions of processes
  having an {I}t{\^o} differential.
\newblock {\em Probability theory and related fields}, 71(4):501--516, 1986.

\bibitem{hao2024}
Z.~Hao, J.-F. Jabir, S.~Menozzi, M.~R{\"o}ckner, and X.~Zhang.
\newblock Propagation of chaos for moderately interacting particle systems
  related to singular kinetic {Mckean}-{Vlasov} {SDEs}.
\newblock {\em Preprint Arxiv 2405.09195}, 2024.

\bibitem{HaoZhangRoeck}
Z.~Hao, M.~R{\"o}ckner, and X.~Zhang.
\newblock Second order fractional mean-field {SDEs} with singular kernels and
  measure initial data.
\newblock Preprint, {arXiv}:2302.04392 [math.{AP}] (2023), 2023.

\bibitem{hao2021singular}
Z.~Hao, X.~Zhang, R.~Zhu, and X.~Zhu.
\newblock Singular kinetic equations and applications.
\newblock {\em Ann. Probab.}, 52(2):576--657, 2024.

\bibitem{haussmann_pardoux}
U.~G. Haussmann and \'E. Pardoux.
\newblock Time reversal of diffusions.
\newblock {\em Ann. Probab.}, 14(4):1188--1205, 1986.

\bibitem{profeta}
F.~Hirsch, Ch. Profeta, B.~Roynette, and M.~Yor.
\newblock {\em Peacocks and associated martingales, with explicit
  constructions}, volume~3 of {\em Bocconi Springer Ser.}
\newblock New York, NY: Springer, 2011.

\bibitem{hoermander}
L.~H{\"o}rmander.
\newblock Hypoelliptic second order differential equations.
\newblock {\em Acta Math.}, 119:147--171, 1967.

\bibitem{HuangWangSPA}
X.~{Huang} and F.-Yu. {Wang}.
\newblock {Distribution dependent SDEs with singular coefficients}.
\newblock {\em {Stochastic Processes Appl.}}, 129(11):4747--4770, 2019.

\bibitem{HuangWang2021}
X.~{Huang} and F.-Yu. {Wang}.
\newblock {McKean-Vlasov SDEs with drifts discontinuous under Wasserstein
  distance}.
\newblock {\em {Discrete Contin. Dyn. Syst.}}, 41(4):1667--1679, 2021.

\bibitem{issoglio19}
E.~Issoglio.
\newblock A non-linear parabolic {PDE} with a distributional coefficient and
  its applications to stochastic analysis.
\newblock {\em J. Differential Equations}, 267(10):5976--6003, 2019.

\bibitem{issoglio_et.al24}
E.~Issoglio, S.~Pagliarani, F.~Russo, and D.~Trevisani.
\newblock Degenerate {M}c{K}ean-{V}lasov equations with drift in anisotropic
  negative {B}esov spaces.
\newblock {\em Preprint Arxiv 2401.09165}, 2024.

\bibitem{issoglio_russo20}
E.~Issoglio and F.~Russo.
\newblock A {F}eynman–{K}ac result via {M}arkov {BSDE}s with generalised
  drivers.
\newblock {\em Bernoulli}, 26(1):728–766, 2020.

\bibitem{issoglio_russoMK}
E.~Issoglio and F.~Russo.
\newblock Mc{K}ean {SDE}s with singular coefficients.
\newblock {\em Annales de l'Institut Henri Poincar\'e. Probabilit\'es et
  Statistiques.}, 59(3):1530--1548, 2023.

\bibitem{issoglio_russoPDEa}
E.~Issoglio and F.~Russo.
\newblock A {PDE} with drift of negative {B}esov index and linear growth
  solutions.
\newblock {\em Differential and Integral Equations.}, 37(9-10):585–622, 2024.

\bibitem{issoglio_russoMPb}
E.~Issoglio and F.~Russo.
\newblock Stochastic differential equations with singular coefficients: The
  martingale problem view and the stochastic dynamics view.
\newblock {\em Journal of Theoretical Probability}, pages 1--42, 2024.

\bibitem{LucasFully}
L.~Izydorczyk, N.~Oudjane, and F.~Russo.
\newblock A fully backward representation of semilinear {PDEs} applied to the
  control of thermostatic loads in power systems.
\newblock {\em Monte Carlo Methods Appl.}, 27(4):347--371, 2021.

\bibitem{Izydorczyk}
L.~Izydorczyk, N.~Oudjane, F.~Russo, and G.~Tessitore.
\newblock Fokker-{Planck} equations with terminal condition and related
  {McKean} probabilistic representation.
\newblock {\em NoDEA, Nonlinear Differ. Equ. Appl.}, 29(1):41, 2022.
\newblock Id/No 10.

\bibitem{JourMeleard}
B.~Jourdain and S.~M{\'e}l{\'e}ard.
\newblock Propagation of chaos and fluctuations for a moderate model with
  smooth initial data.
\newblock {\em Ann. Inst. H. Poincar\'e Probab. Statist.}, 34(6):727--766,
  1998.

\bibitem{Kamin}
S.~Kamin and J.-L. Vazquez.
\newblock Asymptotic behaviour of solutions of the porous medium equation with
  changing sign.
\newblock {\em SIAM J. Math. Anal.}, 22(1):34--45, 1991.

\bibitem{karatzasShreve}
I.~Karatzas and S.E. Shreve.
\newblock {\em Brownian Motion and Stochastic Calculus}.
\newblock Graduate Texts in Mathematics. Springer New York, 1991.

\bibitem{kry-rock}
N.~V. Krylov and M.~R{\"o}ckner.
\newblock Strong solutions of stochastic equations with singular time dependent
  drift.
\newblock {\em Probab. Theory Relat. Fields}, 131(2):154--196, 2005.

\bibitem{lacker}
D.~Lacker.
\newblock On the convergence of closed-loop {Nash} equilibria to the mean field
  game limit.
\newblock {\em Ann. Appl. Probab.}, 30(4):1693--1761, 2020.

\bibitem{ladyzhenskaya}
O.~A. Ladyzhenskaya, V.~A. Solonnikov, and N.~N. Ural'tseva.
\newblock {\em Linear and quasi-linear equations of parabolic type.
  {Translated} from the {Russian} by {S}. {Smith}}, volume~23 of {\em Transl.
  Math. Monogr.}
\newblock American Mathematical Society (AMS), Providence, RI, 1968.

\bibitem{lasrylions}
J.-M. Lasry and P.-L. Lions.
\newblock Mean field games.
\newblock {\em Jpn. J. Math. (3)}, 2(1):229--260, 2007.

\bibitem{lattes1969method}
R.~Latt\`es and J.-L. Lions.
\newblock {\em The method of quasi-reversibility. {A}pplications to partial
  differential equations}.
\newblock Translated from the French edition and edited by Richard Bellman.
  Modern Analytic and Computational Methods in Science and Mathematics, No. 18.
  American Elsevier Publishing Co., Inc., New York, 1969.

\bibitem{KhoaLe}
K.~L{\^e}.
\newblock A stochastic sewing lemma and applications.
\newblock {\em Electron. J. Probab.}, 25:55, 2020.
\newblock Id/No 38.

\bibitem{LOR2}
A.~Le~Cavil, N.~Oudjane, and F.~Russo.
\newblock Particle system algorithm and chaos propagation related to a
  non-conservative {M}c{K}ean type stochastic differential equations.
\newblock {\em Stochastics and Partial Differential Equations: Analysis and
  Computation.}, pages 1--37, 2016.

\bibitem{LOR1}
A.~Le~Cavil, N.~Oudjane, and F.~Russo.
\newblock Probabilistic representation of a class of non-conservative nonlinear
  partial differential equations.
\newblock {\em ALEA Lat. Am. J. Probab. Math. Stat}, 13(2):1189--1233, 2016.

\bibitem{lindqvist}
Peter Lindqvist.
\newblock {\em Notes on the stationary {{\(p\)}}-{Laplace} equation}.
\newblock SpringerBriefs Math. Cham: Springer; Bilbao: BCAM -- Basque Center
  for Applied Mathematics, 2019.

\bibitem{McKean}
H.~P.~jun. McKean.
\newblock A class of {Markov} processes associated with nonlinear parabolic
  equations.
\newblock {\em Proc. Natl. Acad. Sci. USA}, 56:1907--1911, 1966.

\bibitem{montecatini}
S.~M{\'e}l{\'e}ard.
\newblock Asymptotic behaviour of some interacting particle systems;
  {McKean}-{Vlasov} and {Boltzmann} models.
\newblock In {\em Probabilistic models for nonlinear partial differential
  equations. Lectures given at the 1st session of the Centro Internazionale
  Matematico Estivo, Montecatini Terme, Italy, May 22-30, 1995}, pages 42--95.
  Berlin: Springer, 1996.

\bibitem{coppoletta}
S.~Meleard and S.~Roelly-Coppoletta.
\newblock A propagation of chaos result for a system of particles with moderate
  interaction.
\newblock {\em Stochastic Processes Appl.}, 26:317--332, 1987.

\bibitem{nualart}
D.~Nualart.
\newblock {\em The {M}alliavin calculus and related topics}.
\newblock Probability and its Applications (New York). Springer-Verlag, Berlin,
  second edition, 2006.

\bibitem{Oelschlaeger}
K.~Oelschläger.
\newblock A law of large numbers for moderately interacting diffusion
  processes.
\newblock {\em Z. Wahrscheinlichkeitstheorie verw Gebiete}, 69:279--322, 1985.

\bibitem{Olivrichtoma}
Ch. Olivera, A.~Richard, and M.~Tomasevic.
\newblock Quantitative approximation of the {Burgers} and {Keller}-{Segel}
  equations by moderately interacting particles.
\newblock {\em Preprint Arxiv 2004.03177}, 2020.

\bibitem{Olivrichtoma23}
Ch. Olivera, A.~Richard, and M.~Tomasevic.
\newblock Quantitative particle approximation of nonlinear fokker-planck
  equations with singular kernel.
\newblock {\em Ann. Sc. Norm. Super. Pisa Cl. Sci. (5)}, 24(2):691--749, 2023.

\bibitem{lucertini2022optimal}
S.~Pagliarani, G.~Lucertini, and A.~Pascucci.
\newblock Optimal regularity for degenerate {Kolmogorov} equations in
  non-divergence form with rough-in-time coefficients.
\newblock {\em J. Evol. Equ.}, 23(4):37, 2023.
\newblock Id/No 69.

\bibitem{PardouxPeng92}
{\'E}.~Pardoux and S.~G. Peng.
\newblock Backward stochastic differential equations and quasilinear parabolic
  partial differential equations.
\newblock {\em Rozuvskii B L, Sowers R B, eds. Stochastic Partial Differential
  Equations and Their Applications. Lect Notes Control Inf Sci.}, 176:200--217,
  1992.

\bibitem{rondelli}
A.~Pascucci, A.~Rondelli, and A.~Yu Veretennikov.
\newblock Existence and uniqueness results for strongly degenerate
  {McKean}-{Vlasov} equations with rough coefficients.
\newblock {\em Preprint Arxiv2409.14451}, 2024.

\bibitem{portenko}
N.~I. Portenko.
\newblock {\em Obobshchennye diffuzionnye protsessy}.
\newblock Kiev: Naukova Dumka, 1982.

\bibitem{ry}
D.~Revuz and M.~Yor.
\newblock {\em Continuous martingales and {B}rownian motion}, volume 293 of
  {\em Grundlehren der Mathematischen Wissenschaften [Fundamental Principles of
  Mathematical Sciences]}.
\newblock Springer-Verlag, Berlin, third edition, 1999.

\bibitem{nonloc_SP_Roeckner}
M.~R{\"o}ckner, L.~Xie, and X.~Zhang.
\newblock Superposition principle for non-local {Fokker}-{Planck}-{Kolmogorov}
  operators.
\newblock {\em Probab. Theory Relat. Fields}, 178(3-4):699--733, 2020.

\bibitem{RocknerZhang}
M.~{R\"ockner} and X.~{Zhang}.
\newblock {Well-posedness of distribution dependent SDEs with singular drifts}.
\newblock {\em {Bernoulli}}, 27(2):1131--1158, 2021.

\bibitem{russo_trutnau07}
F.~Russo and G.~Trutnau.
\newblock Some parabolic {PDE}s whose drift is an irregular random noise in
  space.
\newblock {\em Ann. Probab.}, 35(6):2213--2262, 2007.

\bibitem{rv4}
F.~Russo and P.~Vallois.
\newblock Stochastic calculus with respect to continuous finite quadratic
  variation processes.
\newblock {\em Stochastics Stochastics Rep.}, 70(1-2):1--40, 2000.

\bibitem{Russo_Vallois_Book}
F.~Russo and P.~Vallois.
\newblock {\em Stochastic Calculus via Regularizations}, volume~11.
\newblock Springer International Publishing. Springer-Bocconi, 2022.

\bibitem{sawano}
Y.~Sawano.
\newblock {\em Theory of {Besov} spaces}, volume~56 of {\em Dev. Math.}
\newblock Singapore: Springer, 2018.

\bibitem{song}
Y.~Song, J.~Sohl-Dickstein, D.~P. Kingma, A.~Kumar, S.~Ermon, and B.~Poole.
\newblock Score-based generative modeling through stochastic differential
  equations. in international conference on learning representations.
\newblock {\em International Conference on Learning Representations}, 2021.

\bibitem{stroock_varadhan}
D.~W. Stroock and S.~R.~S. Varadhan.
\newblock {\em Multidimensional diffusion processes}, volume 233 of {\em
  Grundlehren der Mathematischen Wissenschaften [Fundamental Principles of
  Mathematical Sciences]}.
\newblock Springer-Verlag, Berlin, 1979.

\bibitem{Sznitman}
A.~Sznitman.
\newblock {\em Topics in propagation of chaos}, volume 1464 of {\em Ecole
  d’été de probabilités de Saint-Flour XIX—1989}.
\newblock Springer, 1989.

\bibitem{tanaka}
H.~Tanaka.
\newblock Probabilistic treatment of the {Boltzmann} equation of {Maxwellian}
  molecules.
\newblock {\em Z. Wahrscheinlichkeitstheor. Verw. Geb.}, 46:67--105, 1978.

\bibitem{tikhonov1977solutions}
A.~N. Tikhonov and V.~Y. Arsenin.
\newblock {\em Solutions of ill-posed problems}.
\newblock V. H. Winston \& Sons, Washington, D.C.: John Wiley \& Sons, New
  York-Toronto, Ont.-London, 1977.
\newblock Translated from the Russian, Preface by translation editor Fritz
  John, Scripta Series in Mathematics.

\bibitem{Trevisan}
D.~Trevisan.
\newblock Well-posedness of multidimensional diffusion processes with weakly
  differentiable coefficients.
\newblock {\em Electron. J. Probab.}, 21:41, 2016.
\newblock Id/No 22.

\bibitem{VazquezDecay}
J.-L. V{\'a}zquez.
\newblock {\em Smoothing and decay estimates for nonlinear diffusion equations.
  {Equations} of porous medium type}, volume~33 of {\em Oxf. Lect. Ser. Math.
  Appl.}
\newblock Oxford: Oxford University Press, 2006.

\bibitem{VazquezBookPorous}
J.~L. V{\'a}zquez.
\newblock {\em The porous medium equation. {Mathematical} theory}.
\newblock Oxford Math. Monogr. Oxford: Oxford University Press, 2007.

\bibitem{VazquezLog}
J.~L. Vazquez, J.~R. Esteban, and A.~Rodriguez.
\newblock The fast diffusion equation with logarithmic nonlinearity and the
  evolution of conformal metrics in the plane.
\newblock {\em Adv. Differ. Equ.}, 1(1):21--50, 1996.

\bibitem{veretennikov2023weak}
A.~Veretennikov.
\newblock On weak existence of solutions of degenerate {McKean}-{Vlasov}
  equations.
\newblock {\em Stoch. Dyn.}, 24(5):15, 2024.
\newblock Id/No 2450032.

\bibitem{zhang2021second}
X.~Zhang.
\newblock Second order {M}c{K}ean-{V}lasov {SDE}s and kinetic
  {F}okker-{P}lanck-{K}olmogorov equations.
\newblock {\em Preprint ArXiv:2109.01273, 2021}.

\bibitem{zhang2024cauchy}
X.~Zhang and X.~Zhang.
\newblock Cauchy problem of stochastic kinetic equations.
\newblock {\em Ann. Appl. Probab. 34}, pages 148--202, 2024.

\bibitem{ZhangZhao}
X.~Zhang and G.~Zhao.
\newblock Heat kernel and ergodicity of {SDE}s with distributional drifts.
\newblock {\em Preprint Arxiv 1710.10537, 2017}.

\bibitem{zvonkin}
A.~K. Zvonkin.
\newblock A transformation of the phase space of a diffusion process that
  removes the drift.
\newblock {\em Math. USSR, Sb.}, 22:129--149, 1975.

\end{thebibliography}
\end{document}